\documentclass[a4paper,11pt]{elsarticle}
\usepackage{bm,amsmath,amssymb,amsthm,mathrsfs,graphicx} 
\usepackage{subfigure}
\usepackage{mathrsfs}
\usepackage{xcolor}
\usepackage[colorlinks,linkcolor=blue,citecolor=blue, hyperindex,pagebackref,bookmarks]{hyperref}
\usepackage[capitalise,nosort]{cleveref}

\usepackage{hyperref}
\hypersetup{hypertex=true, colorlinks=true, linkcolor=blue, anchorcolor=blue, citecolor=blue}

\journal{?}

\begin{document}
\def\x{\bm x }
\def\df{\widehat {\bm d}^f }
\def\etaf{\widehat {\bm \eta}^f }
\def\ms{\widehat {\bm m}^s }
\def\mf{\widehat {\bm m}^f }
\def\u{\bm u }
\def\v{\bm v }
\def\V{\bm V }
\def\w{\bm w^f }
\def\c{\bm c }
\def\U{\bm U }
\def\D{\bm D}
\def\O{\Omega}
\def\A{\widehat {\mathcal A} }
\def\J{ J }
\def\P{\widehat{ \mathcal P} }
\def\F{\mathcal F }
\def\S{\mathcal S }
\def\X{\mathcal X }
\def\T{\mathcal T_h }
\def\F{\mathcal F }
\def\E{{\mathcal E}_h }
\def\Bs{\mathbb S }
\def\Bl{\mathbb L }
\def\Bm{\mathbb M }
\def\P{\mathcal P }

\def\div{\mbox{div}}
\def\uf{\bm u^f }
\def\us{\bm u^s }
\def\ds{\widehat {\bm d}^s }
\def\et{\bm{\eta} }
\def\ch{\bm{\chi} }
\def\Os{\Omega_s}
\def\Of{\Omega_f}
\def\Ors{\widehat \Omega_{s}}
\def\Orf{\widehat \Omega_{f}}
\def\leftnorm{ \left|\!\left|\!\left|}
\def\rightnorm{ \right|\!\right|\!\right|}

\numberwithin{equation}{section}
\newtheorem{mytheorem}{Theorem}[section]
\newtheorem{mylemma}{Lemma}[section]
\newtheorem{myremark}{Remark}[section]
\newtheorem{myassumption}{Assumption}
\newtheorem{mycorollary}[mytheorem]{Corollary}
\newtheorem{myprop}[mytheorem]{Proposition}
\newtheorem{mydefinition}{Definition}[section]

\begin{frontmatter}

\title{Stability Analysis of Monolithic Globally Divergence-Free ALE-HDG Methods for Fluid-Structure Interaction}
\tnotetext[t1]
{This work was supported by National Natural Science Foundation of China (12571434) and Open Research Project of National Key Laboratory of Fundamental Algorithms and Models for Engineering Simulation.}
\author{Shuaijun Liu}
\ead{sj\_liu123@163.com}
\author{Xiaoping Xie\corref{cor1}}
\ead{xpxie@scu.edu.cn}
\cortext[cor1]{Corresponding author.}

\address{School of Mathematics,  Sichuan University, Chengdu 610064, China  \\
 National Key Laboratory of Fundamental Algorithms and Models for Engineering Simulation, Sichuan University, Chengdu 610207, China
}

\begin{abstract}
In this paper, we propose two monolithic fully discrete finite element methods for fluid-structure interaction (FSI) based on a novel Piola-type Arbitrary Lagrangian-Eulerian (ALE) mapping.
For the temporal discretization, we apply the backward Euler method to both the non-conservative and conservative formulations. For the spatial discretization, we adopt arbitrary order hybridizable discontinuous Galerkin (HDG) methods for the incompressible Navier-Stokes and linear elasticity equations, and a continuous Galerkin (CG) method for the fluid mesh movement. We  derive stability results for both the temporal semi-discretization and the fully discretization, and   show that the  velocity approximations of the fully discrete schemes are globally divergence-free. Several numerical experiments are performed to verify the performance  of the proposed methods.

\end{abstract}


\end{frontmatter}

\section{Introduction}

Fluid-structure interaction (FSI) is  a   multi-physics phenomenon that describes the mutual interaction and interdependence between a movable or deformable structure and an internal or surrounding fluid flow, and  has a broad range of applications in many engineering fields such as aerospace,  civil and wind engineering, and biomedical engineering. 
As a fundamental multi-physics discipline characterized by strong, highly nonlinear coupling between the fluid flow and the structure,   its accurate numerical simulation has always been one of the most challenging frontiers in computational engineering; see  \cite{Bazilevs1,Bazilevs2013a,Bungartz,Chakrabarti2005,Hou2012, Howe,Jaiman2022,Morand1995,Richter2017} for some books and review articles.

FSI problems inherently involve moving domains, requiring accurate tracking or capturing of the dynamic interface between the fluid and structural components to ensure the accuracy and stability of numerical schemes.
To address this challenge, there are two types of numerical approaches in the literature: the interface-capturing approach and the interface-tracking approach. 
The former, represented by the immersed boundary method \cite{Peskin} and the fictitious domain method \cite{Glowinski}, uses body-unfitted meshes (not attaching to the interface)  and  captures the interface through additional transport equations. 
The latter typically employs body-fitted meshes that adapt to the moving interface, ensuring precise representation of the evolving boundary. The most representative  interface-tracking method is the Arbitrary Lagrangian-Eulerian (ALE) method \cite{Donea1982,Hirt, Hughes, Nitikitpaiboon,Souli2013,Takashi1992,Zhai2025}.  The ALE method combines the advantages of two classical kinematical descriptions, the Lagrangian description and the Eulerian description,  where the computational mesh of the fluid domain can be moved during the solution process  in some arbitrarily specified way so as to avoid  large distortions while  preserving the clear description of interfaces.

Another classification  of FSI numerical methods is  based on the coupling strategy between the fluid and structural solvers:  monolithic \cite{Averweg,Bazilevs,Gee,Ryzhakov,Schwarzacher1} and partitioned approaches \cite{Badia,Bukac,Degroote,Nobile2008,Seboldt}. 
The monolithic approach implicitly enforces the coupling conditions and solves the entire coupled system as a single algebraic system, whereas the partitioned approach explicitly applies the coupling conditions to decouple the system into  the fluid and structure components.   
The former typically incurs higher computational cost, but has superior stability than the latter.  
The latter allows flexible use of existing fluid and structure solvers, but   may be stable  only  within specific physical parameter ranges of the model; for instance, the partitioned algorithm may suffer from numerical instabilities when the fluid and structure densities are comparable due to the so-called added-mass effect  \cite{Causin}. 

The finite element method, as one of the most important numerical methods for solving partial differential equations, has been extensively applied to the spatial discretization of FSI problems \cite{Bazilevs1,Bungartz,Donea1983,Takashi1994};  see, for instance,      \cite{Braun2009, Donea1982,Hubner2004,Schott2019} and \cite{Antonietti2022,Balazsova2018,Froehle,Wang2009,Wang} for some numerical schemes using continuous Galerkin (CG) finite elements  and  discontinuous Galerkin (DG) finite elements, respectively.  
The hybridizable discontinuous Galerkin (HDG) framework, developed in \cite{Cockburn2009} for diffusion problems, provides a unifying strategy for the hybridization of finite element methods. 
On the one hand, the HDG framework preserves the advantages of the DG framework, e.g. the local conservation of physical quantities and the flexibility in meshing. On the other hand, it usually leads to a discrete system of significantly reduced sizes due to the local elimination property, i.e. the unknowns defined in the interior of elements can be locally eliminated by using the numerical traces defined on the interfaces of elements.   We refer to \cite{Chen2023,Cockburn2014,Fu2019,Fu2020, Rhebergen, ZhangM2026} for some mass-conserving HDG discretizations of incompressible flow problems that yield divergence-free velocity approximations.

In the past decade, there have been several works applying HDG methods to the FSI problems. 
In \cite{Sheldon,Sheldon2018} Sheldon et al. proposed monolithic ALE-HDG schemes for modelling FSI,  where the HDG method is applied to discretize the fluid and structure equations.
La Spina et al.  \cite{La2020} gave a monolithic ALE method for FSI problems with weakly compressible flows using an HDG scheme for  the fluid  and a standard CG scheme for  the structure. 
Later, Neunteufel and Sch\"{o}berl~\cite{Neunteufel} presented a monolithic ALE scheme in which the fluid equations are, for the first time, formulated using a Piola-type transformation and discretized with an $H(\div)$-conforming HDG approach to ensure an exactly divergence-free velocity approximation, while the structure equations are discretized using a standard CG method.   
Subsequently, Fu \cite{Fu2023} developed monolithic and strongly coupled partitioned schemes based on an ALE divergence-free HDG method for the Navier-Stokes equations, and a hybridized tangential-velocity-normal-normal-stress scheme for nonlinear hyperelasticity equations.  
We note that the above contributions have primarily focused on the development of discretization techniques without rigorous stability analysis. 

This paper  develops  two  new monolithic Piola-type ALE HDG  schemes for  a FSI model where the fluid and the structure are described respectively by  incompressible Navier-Stokes equations and     linearly elasticity equations,  and establishes  energy stability results. 
We apply the backward Euler method to both the non-conservative and conservative formulations so as to get two temporally semi-discrete schemes.
For the spatial discretization, we employ HDG methods for the fluid  and structure equations, and a continuous Galerkin method for the fluid mesh movement equation.  
More precisely,  on the fluid domain the discontinuous $P_{k-1}$ ($k\geq 1$) element   is used for the fluid stress tensor approximation, the discontinuous $P_k/P_k$ elements are adopted for the fluid velocity /  its trace approximations, the discontinuous  $P_{k-1}/P_k$ elements are applied  for the pressure / its trace approximations, and the continuous $P_k$ element is utilized   for the fluid mesh displacement approximation;  On the  structure domain,  the discontinuous $P_{k-1}$ element is used for the stress tensor approximation, and the discontinuous $P_k/P_k$ elements are adopted for both the structural velocity/ its trace approximations and the displacement / its trace approximations. 
 The choice  of the fluid velocity and pressure spaces, following \cite{Chen2016,Chen2023,Rhebergen}, ensures  globally divergence-free velocity approximations of the two resulting fully discrete  ALE-HDG schemes, as the pressure trace acts as a Lagrange multiplier enforcing continuity of the normal component of the fluid velocity across interelement boundaries. 
We show the  energy stability for both the temporally semi-discrete schemes and the fully discrete schemes.


The remainder of the paper is organized as follows. Section \ref{Set_Pre} introduces the governing equations, the Piola-type ALE mapping, the weak formulations of the FSI problem,  and gives  an energy stability result.  Section \ref{sec_time} applies  the backward Euler method to construct two temporally semi-discrete schemes and analyzes their stability. Section \ref{Sec_Full} presents two fully discrete HDG schemes  and derives    stability results. Section \ref{Sec_Num} provides several numerical experiments to demonstrate the accuracy, stability, and robustness of the proposed schemes. Finally, Section \ref{Sec_con} gives some  concluding remarks.

\section{Model problem and weak formulation}\label{Set_Pre}
\subsection{Notation}
For any  bounded domain $D$, integer $k\geq 0$ and $1\leq p\leq \infty$,  we denote  by $W^{k,p}(D)$ the standard Sobolev space with norm  $\|\cdot\|_{W^{k,p}(D)}$.  
In particular, we have $W^{0,p}(D)=L^p(D)$,  and   $W^{k,2}(D)=H^k(D )$  with   norm abbreviated as  $\|\cdot\|_{k,D}$. We also denote by $(\cdot,\cdot )_D$ and $\langle\cdot,\cdot \rangle_{\partial D}$ the $L^2$-inner product over a domain $D$ and over a boundary $\partial D$, respectively. 
To simplify notation, we use  $A\lesssim B$ to indicate that there exists a constant $C$, independent of mesh size $h$, time step $\tau$ and the numerical solution, such that $A\leq CB$. 

Let $t \in [0,T]$ denote the time variable. We consider a time-dependent domain $\Omega(t)= \overline{\Of(t)} \cup \overline{\Os(t)}\subset \mathbb{R}^d$ ($d\in{2,3}$) with Lipschitz boundary $\partial\Omega(t)$, where   $\Of(t)$ is the fluid subdomain  and $\Os(t)$ is the structural subdomain,  with $\Of(t) \cap \Os(t) = \emptyset$. The two subdomains share a smooth interface $\Gamma(t) := \partial\Of(t) \cap \partial\Os(t)$, across which fluid-structure interaction occurs (cf.  Figure~\ref{fig:dom}).  Denote by  
$$\Orf := \Of(0),\quad  \Ors := \Os(0), \quad  \widehat{\Gamma} := \Gamma(0), $$ 
the initial/reference fluid domain, structural domain and interface, respectively.

\begin{figure}[htbp]
\centering
\subfigure{\includegraphics[width=0.88\textwidth,
height=40mm]{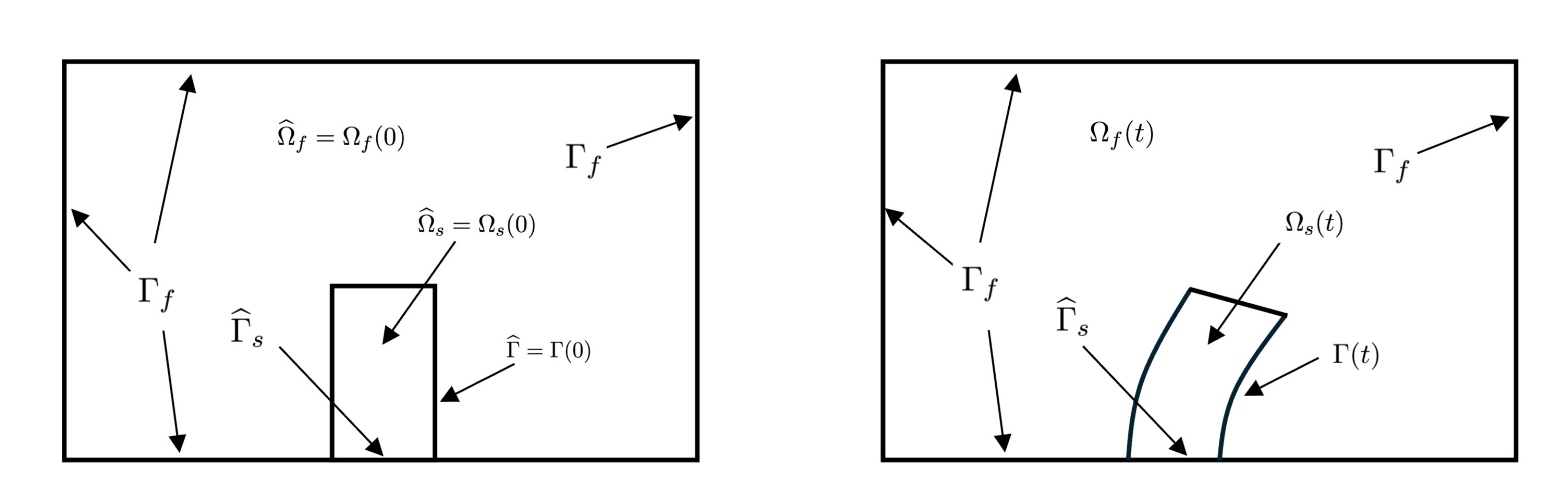}}
\caption{ \footnotesize Schematic of the FSI domain: reference configuration at $t=0$ (left) and configuration at $t>0$ (right). 
}
\label{fig:dom}
\end{figure}

For $i=f,s$, let $\widehat \x_i$ denote the initial/reference/Lagrangian coordinates in $\widehat \Omega_i$ and  $\x_i=\widehat {\bm {X}}_i(\widehat \x_i, t)$   the current/Eulerian coordinates in $ \Omega_i(t)$. Here   $\widehat {\bm {X}}_i(\widehat \x, t):\widehat \Omega_i \rightarrow \Omega_i (t)$ is a bijective flow mapping given by 
 \begin{align}\label{bijective}
\widehat {\bm {X}}_i(\widehat \x_i, t)= \widehat \x_i+ \widehat {\bm{d}}^i(\widehat \x_i, t) , \ \forall \widehat \x_i\in\widehat\Omega_i,
\end{align}
where $\widehat {\bm{d}}^i$ denotes the displacement of material point $\widehat \x_i$   at time $t$. Throughout the paper any quantity `` $ {\cdot}$ '' in the material (Lagrangian) description is denoted with a hat `` $\widehat{\cdot}$ ''. 

\subsection{Model description}
In the following, we show the governing equations of  the fluid and structure, as well as     the fluid-structure coupling conditions.

\subsubsection{Fluid equations}
 In the FSI model, the fluid dynamic behavior is described  by incompressible Navier-Stokes equations in terms of the fluid velocity $\u^f$ and the fluid pressure $p^f$:
\begin{align}\label{fluid_eq}
\left\{
\begin{aligned}
\rho^f{\partial _t\u^f}+\rho^f(\uf\cdot\nabla)\uf &=\nabla\cdot\bm{\sigma}^f(\u^f,p^f) +\bm g^f, &&\text{in}\ \Of(t) \times(0,T], \\
\nabla\cdot \u^f&=0,&&\text{in}\ \Of(t)\times(0,T],\\
\u^f&=\bm 0, &&\text{on}\ \Gamma_f(t)\times (0,T],\\
\u^f(\cdot, 0)&=\u_0^f,&& \text{on}\ \Of(0),
\end{aligned}
\right.
\end{align}
where $\rho^f>0$  denotes the  fluid density  (assumed constant),  $\mathbf{g}^f$   the fluid volume external force,  $\u^f_0$   the initial fluid velocity field, and $\Gamma_f(t) := \partial \Omega_f(t) \backslash \Gamma(t)$. The fluid Cauchy stress $\bm{\sigma}^f$ reads
\begin{align*}
	\bm{\sigma}^f(\u^f,p^f):=2\rho^f\mu^f\bm{D}(\u^f)-p^f\bm{I},
\end{align*} 
 with the constant viscosity coefficient $\mu^f>0$, where $\bm{I}$ is the identity tensor and $\bm{D}(\u^f):=\tfrac{1}{2}(\nabla\u^f+(\nabla\u^f)^\top )$ is the fluid strain rate tensor. 
 
\subsubsection{Structure equations}
As for the structural dynamic behavior, we consider the following linear elasticity system  in terms of the structure velocity $\widehat{\u}^s$ and the structure displacement $\ds$ in the  Lagrangian description:
  \begin{align}\label{structure_eq}
\left\{
\begin{aligned}
 \rho^s{\partial_t\widehat \u^s}-\widehat\nabla\cdot\widehat{\bm{\sigma}}^s( \ds) &= {\bm 0 }, &&\text{in}\ \Ors\times(0,T], \\
\partial_t \ds-\widehat \u^s&={\bm 0 } ,&&\text{in}\ \Ors\times(0,T],\\
\widehat \u^s&=\bm 0, && \text{on}\ \widehat \Gamma_s\times (0,T],\\
\ds(\cdot, 0)=\ds_0,\ \widehat\u^s(\cdot, 0)&=\widehat \u^s_0, &&\text{on}\ \Ors,
\end{aligned}
\right.
\end{align}
where $\widehat{\rho}^s>0$ denotes the structure density (assumed constant), $\widehat{\u}^s_0$ and $\mathbf{d}^s_0$   the initial structure velocity and displacement, respectively, and  $\widehat{\Gamma}_s := \partial \widehat{\Omega}_s \backslash \widehat{\Gamma}$.
The structure Cauchy stress tensor $\widehat{\bm{\sigma}}^s$ is given by
\begin{align*}
	\bm{\widehat \sigma}^s(\ds):= \lambda^s(\widehat\nabla\cdot\ds)\bm I+2\mu^s\bm D(\ds) ,
\end{align*}
where $\bm{D}(\ds):=\tfrac{1}{2}(\widehat\nabla\ds+(\widehat\nabla\ds)^\top )$ is the strain tensor and $\lambda^s,\mu^s$ are the Lam\'{e} constants such that $0<\lambda^s<\infty$ and $0<\mu_1<\mu^s<\mu_2$ for   positive constants $\mu_1$ and $\mu_2$. Here we mention that the gradient operator $\widehat\nabla$   refers to differentiation with
respect to the Lagrangian coordinates  $\widehat \x_s$. 

\subsubsection{Coupling conditions}
For the FSI model, we introduce the following no-slip type interface conditions: 
\begin{itemize}
	\item Kinematic coupling condition: 
	\begin{align}\label{Kine_cond}
		\u^f =\widehat \u^s\circ(\widehat {\bm {X}}_s)^{-1},\qquad \text{on}\  \Gamma(t),\ t\in (0,T].
	\end{align}
	\item Dynamic coupling condition: 
	\begin{align}\label{Dyn_cond}
		\widehat \J_b \bm{\sigma}^f{\bm n}^f+(\widehat {\bm{\sigma}}^s\widehat{\bm n}^s)\circ(\widehat {\bm {X}}_s)^{-1}=\bm 0,\qquad \text{on}\ \Gamma(t),\ t\in (0,T].
	\end{align}
\end{itemize}
Here, $\widehat \J_b$ is the surface Jacobian of the mapping from $\widehat\Gamma$ to $\Gamma(t)$, $\bm n^f$ and $\widehat{\bm n}^s$ denote the fluid and structure outer normal vector on the fluid-structure interface $\Gamma(t)$ and  $\widehat \Gamma$, respectively.  Notice that $\widehat {\bm {X}}_s$ is given by \eqref{bijective} and it holds  $\Gamma(t)=\widehat\Gamma\circ(\widehat {\bm {X}}_s)^{-1}$.

\subsection{The ALE mapping}\label{ALE_sec}
In FSI problems, directly adopting the fluid material (Lagrangian) mapping $\widehat {\bm {X}}_f(\widehat \x_f, t)$ may lead to severe mesh distortion due to large displacements of the fluid material points.
 To avoid this, we introduce for $ t\in(0,T]$ the classical Arbitrary Lagrangian-Eulerian (ALE) mapping $\A(\cdot,t)$, defined as a time-dependent bijection from the reference domain $\Orf$ to the current domain $\Of(t)$, with 
 \begin{align}\label{ALE map}
   \Of(t)=\A(\Orf,t) \text{ and }	\Of(t) \ni \x=\A(\widehat \x, t)=\widehat\x+\etaf(\widehat\x,t) ,\quad \forall \widehat \x\in\widehat\Omega_f.
 \end{align}
Here, $\etaf$ denotes the displacement field of the fluid domain, describing the motion of the computational grid in the fluid domain. It coincides with the structural displacement $\ds(t)(\widehat\x):=\ds(\widehat\x,t)$ on the fluid-structure interface $\widehat \Gamma$ and is smoothly extended to the entire fluid domain $\widehat{\Omega}_f$. A common choice of the extension operator is the weighted harmonic extension: 
\begin{align}\label{mesh_equa}
\left\{
\begin{aligned}
-\widehat \nabla\cdot (\kappa^f \widehat \nabla \etaf) &=0, &&\text{in}\ \Orf, \\
 \etaf &= \ds(t),&&\text{on}\ \widehat \Gamma ,\\
 \etaf &=\bm 0, &&\text{on}\ \widehat \Gamma_f:=\partial\Orf\backslash \widehat \Gamma. 
\end{aligned}
\right.
\end{align}
Here $\kappa^f=\kappa^f(\widehat \x)$ is a positive function. Note that $\kappa^f=1$ is a simple and convenient   choice of the weight function, but in some situations one may choose a more complex function to enable mesh adaptation.

Introduce the fluid domain velocity, the fluid deformation gradient and its Jacobian determinant respectively as follows:
\begin{align}\label{mesh_de}
	\w(\x,t)\circ\A =\widehat{\bm w}^f(\widehat \x, t):= {\partial_t \A(\widehat \x, t) } 
	=\partial_t\etaf(\widehat \x, t), \quad \widehat{\bm F}_f:=\widehat \nabla\A, \quad \widehat \J_f=\det(\widehat{\bm F}_f).
\end{align}  
Applying the chain rule gives
\begin{align}\label{J_F}
	{\partial_t \widehat \J_f}=\widehat \J_f\widehat{\bm F}_f^{-\top}\bm{:}\widehat\nabla\widehat{\bm w}^f
	=\nabla\cdot\w\widehat \J_f,\qquad {\partial_t \widehat{\bm F}_f }=\widehat\nabla\widehat{\bm w}^f
	=\nabla\w \widehat{\bm F}_f,
\end{align}
where `\ $\bm{:}$\ ' denote the Frobenius inner product of two $d\times d$ martrices, i.e. $\mathbb A\bm{:} \mathbb B=\sum\limits_{i,j=1}^d a_{ij}b_{ij}$ for $\mathbb A=[a_{ij}]_{d\times d}$ and $\mathbb B=[b_{ij}]_{d\times d}$.
 
Under the ALE mapping $\A$, we introduce the Piola transformation to ensure that a function $\widehat \v\in\bm H(\widehat \div,\Orf)$ remains in the space $ \bm H(\div,\Of(t))$ after the domain deformation: 
\begin{mydefinition}\label{Piola_defin}
	\rm{(The Piola-type ALE mapping of functions)}	Assume that $\widehat \J_f$ does not vanish at any point and time, then for each $\widehat\v(\widehat\x,t)\in L^\infty (0,T;\bm H(\widehat \div,\Orf))$ we define  
	\begin{align}\label{Piola-type ALE}
		&\v(\x,t):=\P(\widehat \v)\circ {\A}^{-1} (\x, t) \quad \text{with } \P(\widehat \v):=\frac{1}{\widehat \J_f}\widehat{\bm F}_f \widehat \v, \ \x=\A(\widehat \x, t).
	\end{align}
\end{mydefinition}
For the Piola transformation \eqref{Piola-type ALE}, it is standard that 
\begin{align}\label{div_in}
	\nabla\cdot \v(\x,t) =\frac{1}{\widehat \J_f}\widehat\nabla\cdot \widehat \v (\widehat \x, t).
\end{align}
Furthermore, we have $\v\in \bm H(\div,\Of(t))$ if $ {1}/{\widehat \J_f}\in L^{\infty}(\Orf)$. 
 
Let us introduce the following Sobolev space in view of the Piola-type ALE mapping: 
\begin{align*}
	\V_{\div}^{f}(t)&:=\left\{ \v:\Of(t)\rightarrow\mathbb R^d\mid \v=\P(\widehat\v)\circ \A^{-1},\ \widehat\v\in \bm H(\widehat \div,\Orf)  \right\},
\end{align*}
where
 \begin{align*}
	\bm H(\widehat \div,\Orf):=\{\widehat\v\in \bm L^2(\Orf)|\   \widehat\nabla\cdot \widehat\v\in  L^2(\Orf)\}.
\end{align*}
 The next lemma shows some fundamental properties of the Piola transform (cf. \cite[Lemma 2.1.6]{Boffi2013}).
 
\begin{mylemma}\label{lem2-2}
	 For any $\v\in \V_{\rm div}^{f}(t)$ and $$q\in  \widetilde Q^{f}(t):=\left\{ q:\Of(t)\rightarrow\mathbb R\mid q=\widehat q\circ \A^{-1},\ \widehat q\in L^2(\Orf)  \right\},$$  
	 it holds that
\begin{align}\label{212}
	&(\v,\nabla q)_{\Of} =(\widehat \v,\widehat \nabla\widehat q)_{\Orf},\quad 
	(q,\nabla\cdot\v )_{\Of}=(\widehat q,\widehat \nabla\cdot\widehat \v )_{\Orf},\quad
	(\v\cdot\bm n , q)_{\Of}=(\widehat \v\cdot\widehat{ \bm n} ,\widehat q)_{\Orf}.
\end{align}
\end{mylemma}

\subsection{Weak formulations and stability}

For any $\v(\x,t)\in H^1(0,T;\V_{\rm{div}}^f(t))$, define its Piola-type ALE time material derivative as 
\begin{align}\label{P-material}
\partial ^M_t \v(\x,t):=\P\big( {\partial_t \widehat \v } \big)\circ \A^{-1}(\x,t)=\big(\frac{1}{\widehat \J_f}\widehat{\bm F}_f {\partial_t \widehat \v}\big)  \circ \A^{-1}(\x,t).
\end{align} 
We easily get 
\begin{align}\label{Pialo_ma}
	 \partial ^M_t \v ={\partial_t} \v  +(\w\cdot\nabla)\v-\nabla\w\v+\nabla\cdot\w\v, \quad \forall \v\in H^1(0,T;\V_{\rm{div}}^f(t)),
\end{align}
where $\w$ is the fluid domain velocity given by \eqref{mesh_de}. In fact,  by Deﬁnition\ \ref{Piola_defin} and \eqref{J_F} we get
	\begin{align*}
	 {\partial_t} ( \v \circ\A )= {\partial_t}  (\frac{1}{\widehat \J_f}\widehat{\bm F}_f \widehat \v)=& \frac{1}{\widehat \J_f}(-\widehat{\bm F}_f^{-\top}\bm{:}\widehat\nabla\widehat{\bm{w}}^f\widehat{\bm F}_f+ \widehat\nabla\widehat{\bm{w}}^f)\widehat \v+\frac{1}{\widehat \J_f}\widehat{\bm F}_f {\partial_t} \widehat \v  \nonumber\\
	=&\frac{1}{\widehat \J_f}(\nabla\w \widehat{\bm F}_f- \nabla\cdot\w\widehat{\bm F}_f)\widehat \v+\frac{1}{\widehat \J_f}\widehat{\bm F}_f {\partial_t}  \widehat \v  \nonumber\\
	 =&\nabla\w(\v\circ\A)-\nabla\cdot\w(\v\circ\A) +\P( {\partial_t}  \widehat \v )\nonumber\\
	 =&\nabla\w\v-\nabla\cdot\w\v +\partial ^M_t \v.
\end{align*}
On the other hand, according to the chain rule we obtain
\begin{align*}
	 {\partial_t}  (\v \circ\A )= {\partial_t}  \v  \circ\A+(\w\cdot\nabla)(\v\circ\A)={\partial_t}  \v  +(\w\cdot\nabla)\v. 
\end{align*}
As a result, \eqref{Pialo_ma} follows.

 In light of the relation \eqref{Pialo_ma} we immediately obtain the following Reynolds transport formulas  under the Piola-type ALE frame:

\begin{mylemma}
\label{lemma2.3}
 For $\v,\u \in H^1(0,T;\V_{\rm{div}}^f(t))$ it holds
\begin{align}\label{Reynolds}
\begin{split}
	\frac{d}{dt}(\u,\v)_{\Of(t)}	   &=\left(\partial ^M_t \u+\nabla\w\u-\nabla\cdot\w\u+(\nabla\w)^{\top}\u, \v \right)_{\Of(t)}+\left(\u,\partial ^M_t \v  \right )_{\Of(t)}\\
	&=\left( {\partial_ t} \u  +(\w\cdot\nabla)\u, \v \right)_{\Of(t)}+\left(\u, {\partial_ t} \v +(\w\cdot\nabla)\v+\nabla\cdot\w \v\right)_{\Of(t)}.
\end{split}
\end{align}
In particular,   if $ \partial ^M_t \v=0$, then 
	\begin{align}\label{Reynolds2}
	\begin{split}
		\frac{d}{dt}(\u,\v)_{\Of(t)}&=\left(\partial^M_t \u+\nabla\w\u-\nabla\cdot\w\u+(\nabla\w)^{\top}\u, \v \right)_{\Of(t)}\\
		&=\left( {\partial_t \u } +(\w\cdot\nabla)\u +(\nabla\w)^{\top}\u, \v \right)_{\Of(t)} 
		\end{split}
	\end{align}
and 
	\begin{align}\label{Reynolds00}
		&\frac{d}{dt}\|\v\|^2_{L^2(\Of(t) )}=\left(2 \nabla\w\v-\nabla\cdot\w\v, \v \right)_{\Of(t)}.
	\end{align}
\end{mylemma}  
\begin{proof}  
In view of \eqref{Piola-type ALE},  \eqref{J_F} and \eqref{P-material} we have 
\begin{align*}
\begin{split}
	\frac{d}{dt}(\u,\v)_{\Of(t)}&=\frac{d}{dt}(\P(\widehat \u),\P(\widehat \v)\widehat \J_f )_{\Orf}=\frac{d}{dt}(\P(\widehat \u), \widehat{\bm F}_f  \widehat \v)_{\Orf}\\
	   &=(  {\partial_ t}(\u \circ\A ),  \widehat{\bm F}_f  \widehat \v )_{\Orf}+(\u \circ\A, {\partial_ t} \widehat{\bm F}_f \ \widehat \v+ \widehat{\bm F}_f {\partial_ t}\widehat \v )_{\Orf}\\
	   &=( {\partial_ t}(\u \circ\A ), \v )_{\Of(t)}+(\u,\partial ^M_t \v +\nabla\w\v )_{\Of(t)}.
\end{split}
\end{align*}
which, together with   \eqref{Pialo_ma}, implies the identities \eqref{Reynolds} and \eqref{Reynolds2}. 
The identity \eqref{Reynolds00} follows from \eqref{Reynolds2} and the relation
$$\left( \nabla\w\v,, \v \right)_{\Of(t)}=\left( (\nabla\w)^{\top}\v, \v \right)_{\Of(t)}.$$
\end{proof}

\begin{myremark}\label{gcl-rem}
By taking $\u=\v= \mathbf{e}_i, \; i=1,\dots,d$ in \eqref{Reynolds}, where $\{\mathbf{e}_i\}_{i=1}^d$ denotes the canonical basis of $\mathbb{R}^d$, we can obtain
	\begin{align}\label{gcl-conti}
		\frac{d}{dt}|\Of(t)|=\frac{d}{dt}\int_{\Of(t)}d\x=\int_{\Of(t)} \nabla\cdot\w   d\x=\int_{\partial \Of(t)} \w\cdot\bm n dS,  
	\end{align}
where $|\Of(t)|$ is the measure of the subdomain $\Of(t)$.  This relation is called  the Geometric Conservation Law (GCL) for  flow problems with a moving domain   \cite{Farhat,Guillard}. 
\end{myremark}

For the time derivative and convective terms in the fluid system \eqref{fluid_eq}, we have the following results. 
\begin{mylemma}\label{time-conv}
For $\u  \in \bm H^1(\Of(t))$  with $\nabla \cdot \u= 0$, $ \w \in \bm H^1(\Of(t))$ given by  \eqref{mesh_de}, and $\v \in \bm H_0^1(\Of(t))$, it holds 
\begin{align}\label{time_derivative}
\begin{split}
   \bigl(\partial_t \u + (\u \cdot \nabla)\u , \v \bigr)_{\Of(t)}  
   ={}&\Bigl( \partial_t^M \u +  \nabla \w\ \u - \tfrac{1}{2}\nabla \cdot \w\,\u , \v \Bigr)_{\Of(t)}\\
   &\quad + \tfrac{1}{2}\Bigl( (\v \cdot \nabla) \u - (\u \cdot \nabla) \v,\, \u- \w \Bigr)_{\Of(t)}.
\end{split}
\end{align}
Moreover, if $\partial_t^M \v^f = 0$ then 
\begin{align}\label{time_derivative_2}
\begin{split}
   \bigl(\partial_t \u + (\u \cdot \nabla)\u , \v\bigr)_{\Of(t)}  
   ={}& \frac{d}{dt}\bigl(\u,\v \bigr)_{\Of(t)}
   - \Bigl( (\nabla \w)^{\top}\u-\tfrac{1}{2} \nabla \cdot \w\,\u  , \v \Bigr)_{\Of(t)} \\
   &\quad+ \tfrac{1}{2}\Bigl( (\v \cdot \nabla) \u- (\u \cdot \nabla) \v,\, \u- \w \Bigr)_{\Of(t)}.
\end{split}
\end{align}

\begin{proof}
First,  applying integration by parts and    \eqref{Pialo_ma} we obtain
\begin{align*}
   &\bigl(\partial_t \u + (\u \cdot \nabla)\u , \v \bigr)_{\Of(t)} \\
   ={}&\Bigl( \partial_t^M \u + \nabla \w\u- \nabla \cdot \w\,\u + ((\u - \w)\cdot \nabla)\u , \v \Bigr)_{\Of(t)} \\
   ={}&\Bigl( \partial_t^M \u + \nabla \w\u - \tfrac{1}{2}\nabla \cdot \w\,\u , \v \Bigr)_{\Of(t)}
      + \tfrac{1}{2}\Bigl( (\v \cdot \nabla) \u - (\u \cdot \nabla) \v,\, \u - \w \Bigr)_{\Of(t)},
\end{align*}
which proves \eqref{time_derivative}.

Furthermore, if $\partial_t^M \v^f = 0$,   by the first equality of \eqref{Reynolds2} we get 
 \begin{align*}
 \Bigl( \partial_t^M \u + \nabla \w\u - \tfrac{1}{2}\nabla \cdot \w\,\u , \v \Bigr)_{\Of(t)}=
 \frac{d}{dt}(\u,\v)_{\Of(t)}+\left( \tfrac{1}{2}\nabla \cdot \w\,\u -(\nabla\w)^{\top}\u, \v \right)_{\Of(t)}.
 \end{align*}
 Substituting this relation into  \eqref{time_derivative} yields 
  \eqref{time_derivative_2}.  
\end{proof}
\end{mylemma}

\begin{myremark}
	Equations \eqref{time_derivative} and \eqref{time_derivative_2} are consistent in the continuous case, with the former often referred to as a non-conservative formulation and the latter as a conservative formulation for the time derivative and convective terms in the fluid equations. Based on them, we shall construct two types of  temporal discretizations,  a non-conservative scheme and a  conservative one, and     analyze their stability  in Section \ref{sec_time}. We mention that the non-conservative formulation \eqref{time_derivative} has been  adopted  in ~\cite{Fu2023, Neunteufel} to construct the temporal discretizations of FSI problems, whereas the conservative form  \eqref{time_derivative_2} is, to the best of our knowledge, introduced for the first time in our contribution.
\end{myremark}

To  give the ALE weak formulation of the FSI model \eqref{fluid_eq}-\eqref{Dyn_cond}, we need to introduce, for all $t\in[0,T]$, the following function spaces:
\begin{align*}
	\V^f(t)&:=\left\{ \v^f:\Of(t)\rightarrow\mathbb R^d\mid \v^f=\P(\widehat\v^f)\circ \A^{-1},\ \widehat\v^f\in \bm H^1(\Orf)  \right\},\\
	\V^f_0(t)&:=\left\{ \v^f\in \V^f(t)\mid \v=\bm 0\ \text{on}\ \Gamma_f(t)  \right\},\\
	 \bm W^f(t)&:=\left\{\v^f\in \V^f_0(t)\mid \nabla\cdot \v^f=0 \right\},\\
	Q^f(t)&:=\left\{ q^f:\Of(t)\rightarrow\mathbb R\mid q^f=\widehat q^f\circ \A^{-1},\ \widehat q^f\in L_0^2(\Orf)  \right\},\\
	\widehat \V_0^s&:=\left\{ \widehat \v^s:\Ors\rightarrow\mathbb R^d\mid \widehat \v^s\in \bm H^1(\Ors),\ \widehat \v^s=\bm 0 \ \text{on}\ \widehat \Gamma_s \right\},\\
	\widehat \V_D^{\mathcal M}&:= \left\{\widehat{\bm d} :\Orf\rightarrow\mathbb R^d\mid  \widehat{\bm d}\in \bm H^1(\Orf),\  \widehat{\bm d} =\bm 0 \ \text{on}\ \widehat \Gamma_f \right\},\\
	\widehat \V_0^{\mathcal M}&:= \left\{\mf :\Orf\rightarrow\mathbb R^d\mid  \mf\in \bm H^1(\Orf),\ \mf=\bm 0 \ \text{on}\ \partial\Orf \right\},\\
	\V^{FSI}(t)&:=\left\{ ( \v^f, \widehat \v^s)\in\V^f_0(t)\times\widehat \V^s_0  \mid  \v^f=\widehat \v^s\circ \A^{-1} \ \text{on}\  \Gamma(t) \right\}.
\end{align*}
Then, in view of \eqref{time_derivative}, we describe the ALE weak formulation for FSI as follows:

Find $\left((\u^f,\widehat \u^s),p^f,\ds, \etaf\right)\in \V^{FSI}(t)\times Q^f(t)\times\widehat \V^s_0 \times \widehat \V_D^{\mathcal M}$,  with $\ds=\etaf$ on $\widehat\Gamma$ and $\w =\partial_t\etaf \circ \A^{-1}$, such that 
\begin{align}\label{weak_pro}
\left\{
\begin{aligned}
\rho^f(\partial_t^M\u^f+\nabla\w\u^f-\tfrac{1}{2} \nabla\cdot\w\u^f, \v^f)_{\Of(t)}+\rho^f T^{f,t}(\u^f-\w, \u^f,\v^f)\\+2\rho^f\mu^fB^{f,t}(\u^f,\v^f) 
	-Q^{f,t}(p^f,\v^f ) +\rho^s(\partial_t\widehat \u^s,\widehat \v^s )_{\Ors}
	\\+2\mu^sB^s( \ds,\widehat \v^s) +\lambda^s(\widehat\nabla\cdot\ds, \widehat\nabla\cdot\widehat\v^s)_{\Ors}-(\bm g^f, \v^f)_{\Of(t)}&=0,
	\\Q^{f,t}(q^f,\u^f ) &=0,\\
	(\partial_t \ds,\ms )_{\Ors}-(\widehat \u^s,\ms )_{\Ors}&=0,\\
	(\kappa^f\widehat\nabla \etaf,\widehat\nabla \mf)_{\Orf} &=0, 
\end{aligned}
\right.
\end{align}
for all $ \left((\v^f,\widehat\v^s), q^f, \ms,\mf\right)\in \V^{FSI}(t)\times Q^f(t)\times\widehat \V_0^s \times\widehat \V_0^{\mathcal M}$, where the bilinear/trilinear forms   in \eqref{weak_pro} are defined respectively by
\begin{align*}
    &B^{f,t}(\u^f,\v^f) :=
   \left( \bm D(\u^f),\bm D(\v^f)\right)_{\Of(t)},\qquad Q^{f,t}(q^f,\v^f ) :=\left(q^f,\nabla\cdot\v^f\right)_{\Of(t)},\\
    &B^s( \ds,\widehat \v^s) :=\left( \bm D(\ds), \bm D(\widehat \v^s)\right)_{\Ors},\qquad T^{f,t}(\w, \u^f,\v^f) :=\tfrac{1}{2} \left( \v^f\cdot(\nabla\u^f)- \u^f\cdot(\nabla\v^f),  \w \right)_{\Of(t)}.
\end{align*}
\begin{myremark}
	By means of a change of variables,  the identity   $((\bm  a\cdot\nabla )\bm  b, \bm c)=(\bm c\otimes \bm a\bm :\nabla  \bm b ) $ and Lemma \ref{lem2-2}, we may reformulate the weak formulation \eqref{weak_pro} from the current domain $\Omega(t)$ onto the reference domain $\widehat\Omega$ as follows: 
	
Find $\left((\widehat\u^f,\widehat \u^s), \widehat p^f,\ds, \etaf\right)\in \widehat\V^{FSI}\times \widehat Q^f\times\widehat \V^s_0 \times \widehat \V_D^{\mathcal M}$,  with $\ds=\etaf$ on $\widehat\Gamma$, such that 
\begin{align}\label{weak_pro_ref}
\left\{
\begin{aligned}
\rho^f\left( (\P( \partial_t^M\widehat \u^f)+(\widehat\nabla\widehat{\bm{w}}\widehat{\bm F}_f^{-\top}-\tfrac{1}{2}\widehat{\bm F}_f^{-1}\bm{:}\widehat\nabla\widehat{\bm{w}}^f)\P(\widehat \u),  \widehat \J_f\P(\widehat \v^f)\right)_{\Orf}
	\\+2\rho^f\mu^f\widehat B^f(\widehat\u^f,\widehat\v^f)+\rho^f \widehat T^f(\widehat\u^f-\widehat\w, \widehat\u^f,\widehat\v^f)-\widehat Q^f(p^f,\v^f )+\rho^s(\partial_t\widehat \u^s,\widehat \v^s )_{\Ors}
	\\+2\mu^sB^s( \ds,\widehat \v^s) +\lambda^s(\widehat{\nabla}\cdot\ds, \widehat{\nabla}\cdot\widehat\v^s)_{\Ors}-(\widehat{ \bm g}^f, \J_f\P(\widehat \v^f))_{\Orf}&=0,\\
	\widehat Q^f(\widehat q^f,\widehat \u^f ) &=0,\\
	(\partial_t \ds,\ms )_{\Ors}-(\widehat \u^s,\ms )_{\Ors}&=0,\\
	(\kappa^f \widehat\nabla \etaf,\widehat\nabla \mf)_{\Orf} &=0, 
\end{aligned}
\right.
\end{align}
for all $ \left((\widehat \v^f,\widehat\v^s), q^f, \ms,\mf\right)\in \widehat\V^{FSI}\times \widehat Q^f\times\widehat \V_0^s \times\widehat \V_0^{\mathcal M}$, where   
\begin{align*}
& \widehat\V^{FSI}:=\left\{ ( \widehat\v^f, \widehat \v^s)\in\bm H^1(\Orf)\times\widehat \V^s_0  \mid   \widehat\v^f=\widehat \v^s \ \text{on}\  \widehat\Gamma, \ \widehat\v^f=0 \text{ on }\  \widehat\Gamma_f \right\},\quad \widehat Q^f:=  L_0^2(\Orf) ,\\
&\widehat {\bm{w}}^f =\partial_t\etaf, \quad 
\bm D(\widehat \u^f):=\tfrac{1}{2} \left(\widehat \nabla\P(\widehat \u^f)\widehat{\bm F}_f^{-1}+ \widehat{\bm F}_f^{-\top}(\widehat \nabla \P(\widehat \u^f))^{-\top}\right), \\
&\widehat B^f(\widehat\u^f,\widehat\v^f) :=\left(\bm D(\widehat \u^f),\bm D(\widehat \v^f)\widehat \J_f\right)_{\Orf},\qquad \widehat Q^f(\widehat p^f,\widehat \v^f ) := \left(\widehat p^f,\widehat \nabla\cdot \widehat \v^f\right)_{\Orf},\\
    &\widehat T^f(\widehat\u^f-\widehat{\bm  w}^f, \widehat\u^f,\widehat\v^f): =\tfrac{1}{2} \left( \big(\P(\widehat \v^f) \otimes (\P( \widehat\u^f)-\widehat{\bm  w}^f), \widehat\nabla \P(\widehat \u^f)\widehat{\bm F}_f^{-1}\J_f\big)_{\Orf}\right.
    \\&\qquad\qquad\qquad\qquad\qquad\quad - \left.\big(\P(\widehat \u^f)\otimes (\P( \widehat\u^f)-\widehat{\bm  w}^f), \widehat\nabla \P(\widehat \v^f)\widehat{\bm F}_f^{-1}\widehat \J_f  \big)_{\Orf}\right).	
\end{align*}
\end{myremark}

Let $E(t)$ denote the sum of the kinetic energy of the fluid, the kinetic and elastic energy of the structure, i.e., 
\begin{align*}
	E(t):=\frac{\rho^f}{2}\|\u^f(t)\|^2_{L^2({\Of(t)})}+\frac{\rho^s}{2} \|\widehat \u^s(t)\|^2_{L^2(\Ors) }+\left (\mu^s\|\widehat{ \bm D}( \ds(t))\|^2_{L^2(\Ors )}+\frac{\lambda^s }{2}\|\widehat\nabla\cdot \ds(t)\|^2_{L^2(\Ors )}\right ).
\end{align*}
Then we can follow a standard routine to get  the  energy stability of the weak solution.
\begin{mytheorem} 
	Let $\left((\u^f,\widehat \u^s),p^f,\ds, \etaf\right)\in \V^{FSI}(t)\times Q^f(t)\times\widehat \V^s_0 \times \widehat \V_D^{\mathcal M}$ be a solution to the weak problem \eqref{weak_pro}, then for any $T>0$ it holds   the   energy identity 
	\begin{align}\label{sta_weak}
	\begin{split}
	E(T)+ \int_0^T 2\rho^f\mu^f\|\bm D(  \u^f )\|^2_{L^2(\Of(t))}dt = \int_0^T (\bm g^f ,\u_f )_{\Of(t) }dt+E(0).
	\end{split}
	\end{align}
	In particular, if the source term $\bm g^f=\bm 0$, then 
	\begin{align}\label{gf=0}
		E(T)+\int_{0}^{T} 2\rho^f\mu^f\|\bm D(  \u^f )\|^2_{L^2(\Of(t))}dt=E(0),
	\end{align} 
	which means that the energy $E(t)$   decreases monotonically over time $0<t\leq T$.	
	\begin{proof}
		Taking  $ (\v^f,\widehat\v^s, q^f)=(\u^f,\widehat \u^s,p^f)$ in \eqref{weak_pro} and using $T^{f,t}(\u^f-\w, \u^f,\u^f)=0$,  $\partial_t \ds=\widehat \u^s$ and \eqref{Reynolds}, we have
		\begin{align*}
			(\bm g^f, \u^f)_{\Of(t)}
=&\rho^f(\partial_t^M\u^f+\nabla\w\u-\tfrac{1}{2} \nabla\cdot\w\u, \u^f)_{\Of(t)}+2\rho^f\mu^f\|\bm D(\u^f)\|^2_{L^2(\Of(t))}+\frac{\rho^s}{2}\frac{d}{ dt} \|\widehat \u^s\|^2_{L^2(\Ors) }
			\\&\ \ +2\mu^sB^s( \ds,\partial_t \ds )_{\Ors}+\lambda^s(\widehat \nabla\cdot\ds,\widehat\nabla\cdot\partial_t\ds )_{\Ors}\\
			=&\frac{\rho^f}{2}\frac{d}{dt}\|\u^f\|^2_{L^2(_{\Of(t)})}+\frac{\rho^s}{2}\frac{d}{ dt} \|\widehat \u^s\|^2_{L^2(\Ors) }+2\rho^f\mu^f\|\bm D(\u^f)\|^2_{L^2(\Of(t))}
			\\
			&\ \ +\mu^s\frac{d}{dt}\|\D(\ds)\|^2_{L^2(\Ors )}+\frac{\lambda^s }{2}\frac{d}{dt}\|\widehat\nabla\cdot \ds\|^2_{L^2(\Ors )}\\
			=&\frac{dE(t)}{dt}+2\rho^f\mu^f\|\bm D(\u^f)\|^2_{L^2(\Of(t))}.						
		\end{align*}
Integrating this relation from $0$ to $T$ gives the desired energy identity \eqref{sta_weak} immediately, and the result \eqref{gf=0} follows. 	
	\end{proof}
\end{mytheorem}

\section{Temporal semi-discretization}\label{sec_time}
In this section, we first introduce the time-discrete Piola-type ALE mapping. Based on this, we present two semi-discrete schemes corresponding to the non-conservative and conservative formulations, and establish the associated energy stability results.

\subsection{The semi-discrete Piola-type ALE mapping}\label{sec_ALE_time}
Let $0=t_0<t_1<...<t_{\mathrm{N}}=T$ be a uniform division of the time interval $[0,T]$,  with $t_n=n\tau$ ($n=0,1,\cdots, N$)  and  the time step size $\tau=T/N$. We  discretize in time the ALE mapping $\A(\widehat \x, t)$ in \eqref{ALE map} by using a continuous piecewise linear interpolation. To this end, let  
\begin{align}\label{semi-discrete ALE-n}
\A^n(\widehat \x)=\widehat\x+ \widehat{\bm \eta}^{f,n}(\widehat\x):\Orf\rightarrow \Of^n:=\A^n(\Orf) 
\end{align}
be an approximation of the ALE mapping $\A(\widehat \x,t_n)$ for $n=1,\dots, N$. Then the linear interpolation of  the mapping $\A(\widehat \x, t)$   for $t\in [t_{n-1},t_{n}]$, denoted by 
$$ \A^t(\widehat \x,t): \Orf\rightarrow \Of^t:=\A^t(\Orf),$$  is defined as
\begin{align}\label{semi-discrete ALE}
	 \x=\A^t(\widehat \x,t):=\frac{t-t_{n-1}}{\tau}\A^{n}(\widehat \x)+\frac{t_{n}-t}{\tau}\A^{n-1}(\widehat \x), \ \ t\in[t_{n-1}, t_{n}].
\end{align}
Notice that $ \A^t(\widehat \x,t_n)=\A^n(\widehat \x)$ for any $n$.

We  define the fluid deformation gradient $\widehat{\bm F}_f^t$   and   its Jacobian determinant $\widehat{\J}_f^t$  
 as 
\begin{align}\label{semi-fluid-deformation}
	  \widehat{\bm F}_f^t:=\widehat \nabla\A^t=\frac{t-t_{n-1}}{\tau}\widehat \nabla\A^{n}+\frac{t_{n}-t}{\tau}\widehat \nabla\A^{n-1} , \quad \widehat{\J}_f^t:=\det(\widehat{\bm{F}}_f^t), \quad  t\in [t_{n-1}, t_{n}]
\end{align}
and set
 \begin{align*}
 &\widehat{\bm F}_f^n:=\widehat{\bm F}_f^t|_{t = t_n}=\widehat \nabla\A^n, \quad \widehat{\J}_f^n: = \widehat{\J}_f^t|_{t = t_n}=\det(\widehat{\bm {F}}_f^n).
 \end{align*}
Similar to \eqref{J_F}, with   \eqref{semi-fluid-velocity} and \eqref{semi-fluid-deformation} there hold for  $t\in [t_{n-1}, t_{n}] $
\begin{align}\label{J_F_D}
	\partial_t\widehat{\J}_f^t=\widehat \J_f^t(\widehat{\bm F}_f^{t})^{-\top}\bm{:}\widehat\nabla\widehat{\bm w}^{f,n}=\nabla\cdot{\bm w}_t^{f,n}\widehat \J_f^t,\qquad {\partial_t \widehat{\bm F}^t_f }=\widehat\nabla\widehat{\bm w}^{f,n}=\nabla{\bm w}_t^{f,n}\widehat{\bm F}^t_f.
\end{align}
 
With the temporally semi-discrete ALE mapping 
\eqref{semi-discrete ALE}, we set the fluid domain velocities ${\bm w}^{f,n}$  on $\Omega_f^n$  
and ${\bm w}^{f,n}_t$ on $\Omega_f^t$   as follows:
\begin{align}\label{semi-fluid-velocity}
\left\{
\begin{array}{l}
{\bm w}^{f,n}(\x)={\bm w}^{f,n}\circ\A^n(\widehat \x)=\widehat{\bm w}^{f,n}(\widehat \x) 
	:=\left.{\partial_t \A^t(\widehat \x,t)}\right|_{t=t_n}  =\frac1\tau \left(\A^{n}(\widehat \x)-\A^{n-1}(\widehat \x)\right),\\
	{\bm w}^{f,n}_t(\x,t):=	\widehat{\bm w}^{f,n}\circ (\A^t)^{-1}(  \x,t), \quad t\in [t_{n-1},t_{n}].
\end{array}
	\right. 
	\end{align}
For convenience, we write 
	\begin{align}\label{w-nj}
{\bm w}_j^{f,n}:= {\bm w}^{f,n}_t|_{t=t_j},  \quad j=n, n-1
\end{align}
and  easily see ${\bm w}^{f,n}_n ={\bm w}^{f,n} .$  
	
Introduce the semi-discrete Piola-type ALE mapping
\begin{align}\label{semi-Piola}
\v\circ \A^t=\P^t(\widehat{\v})  := \frac{1}{\widehat{\J}_f^t} \widehat{\bm{F}}_f^t \widehat{\v}, \quad \text{for }\widehat\v\in   \bm H(\widehat \div,\Orf) 
\end{align}
and, similar to  \eqref{P-material},  define the corresponding Piola-type ALE  material derivative of $\v$ as 
\begin{align}\label{P-material-semi}
\partial ^M_t \v(\x,t):=\P^t\big( {\partial_t \widehat \v  (\widehat \x, t)} \big) \quad \text{for } \x=\A^t(\widehat \x,t).
\end{align} 
For simplicity we set 
 \begin{align}\label{semi-Piola-v}
 & \v^n\circ \A^n=\P^n(\widehat{\v}) :=\P^t(\widehat{\v})|_{t = t_n}= \frac{1}{\widehat \J_f^n} \widehat{\bm{F}}_f^n \widehat{\v}.
 \end{align}
 Take $\X_i^j:= \A^j\circ (\A^i)^{-1}$, then by \eqref{semi-Piola} we get  
\begin{align}\label{X_de}
	\v \circ\X_i^j&=(\frac{\widehat{\bm F}_f^{i}}{\widehat \J_f^{i}}\widehat\v) \circ (\A^i)^{-1} = \left (\frac{\widehat \J_f^{j}}{\widehat \J_f^{i}}\widehat{\bm F}_f^{i}(\widehat{\bm F}^j_f)^{-1}\v\circ \A^j\right )\circ (\A^i)^{-1}, \quad \forall \v\in \V^f(t_j)   .
\end{align}

\begin{myremark}
As the semi-discrete ALE mapping $\A^t$ is just a special case of the general ALE mapping $\A$,  the identities \eqref{Reynolds}-\eqref{Reynolds00}  in  Lemma \ref{lemma2.3} still hold for  $\v=\P^t(\widehat{\v})\circ (\A^t)^{-1} ,\u=\P^t(\widehat{\u})\circ (\A^t)^{-1}  $ with $ \widehat\v,  \widehat\u\in   H^1 (0,T;\bm H(\widehat {\rm{div}},\Orf)).$ Hence,  for $\v=\u$ we have the following  integral-form  Reynolds transport formulas:
\begin{align}\label{GCL_d-1}
	\begin{split}
	\|\v^n\|^2_{L^2(\Of^{n} )}-\|\v^{n-1}\|^2_{L^2(\Of^{n-1} )}
   		=&2\int^{t_{n}}_{t_{n-1}} \left(\partial ^M_t \v+\nabla {\bm{w}_t^{f,n}}\v -\frac12\nabla\cdot \bm{w}_t^{f,n}\v  ,  \v\right)_{\Of(t)} dt\\
		 =&2\int^{t_{n}}_{t_{n-1}} \left(\partial _t \v+ ({\bm{w}_t^{f,n}}\cdot \nabla) \v  +\frac12\nabla\cdot \bm{w}_t^{f,n}\v ,  \v\right)_{\Of(t)} dt.
	\end{split}
\end{align}
which is an  integral form of the  GCL  \eqref{gcl-conti}, and that 
\begin{align}\label{GCL_d-3}
	\|\v^n\|^2_{L^2(\Of^{n} )}-\|\v^{n-1}\|^2_{L^2(\Of^{n-1} )}
   		=&2\int^{t_{n}}_{t_{n-1}} \left((\nabla {\bm{w}_t^{f,n}} -\frac12\nabla\cdot \bm{w}_t^{f,n})\v  ,  \v\right)_{\Of(t)} dt
\end{align}
 if $  \partial ^M_t \v=0$ or equivalently,  $\partial_t \widehat \v =0$.  
 \end{myremark}

\subsection{Temporally semi-discrete schemes}\label{semi-discrete schemes}
We first introduce the following temporally semi-discrete  finite element spaces:
\begin{align*}
	\V^{f,n}&:=\left\{ \v^f:\Of^n\rightarrow\mathbb R^d\mid \v^f=\P^n(\widehat\v^f)\circ (\A^n)^{-1},\ \widehat\v^f\in \bm H^1(\Orf)  \right\},\\
	\V^{f,n}_0&:=\left\{ \v^f\in \V^{f,n}\mid \v=\bm 0\ \text{on}\ \Gamma_f^n  \right\},\ \bm W^{f,n}:=\left\{\v^f\in \V^{f,n}_0\mid \nabla\cdot \v^f=0 \right\},\\
	Q^{f,n}&:=\left\{ q^f:\Of^n\rightarrow\mathbb R\mid q^f=\widehat q^f\circ (\A^n)^{-1},\ \widehat q^f\in L_0^2(\Orf)  \right\},\\
	\V^{FSI,n}&:=\left\{ ( \v^f, \widehat \v^s)\in\V^{f,n}_0\times\widehat \V^s_0  \mid  \v^f=\widehat \v^s\circ (\A^n)^{-1} \ \text{on}\  \Gamma^n \right\}.
\end{align*}

For any  time-dependent function $g$, denote by $g^n$ an approximation of $g(t_n)$ ($n=1,2,\cdots,N$) 
 and  introduce the following two  backward difference operators:
\begin{align}\label{D_t}
	\tilde{D}_t g^{n}:=\frac{g^{n}-g^{n-1}}{\tau},\qquad \tilde{D}_t ^Mg^{n}:=\frac{g^{n}-g^{n-1}\circ \X_{n}^{n-1} }{\tau}.
\end{align}

Given $(\u^{f,0},\widehat \u^{s,0},\widehat{\bm d}^{s,0})=(\u_0^f,\widehat \u^{s}_0,\widehat{\bm d}^{s}_0)$, based on  equations \eqref{time_derivative} and weak formulation \eqref{weak_pro} we  give the following   two types of implicit temporal  semi-discretizations for the FSI model  \eqref{fluid_eq}-\eqref{Dyn_cond}: 
\begin{itemize}
	\item \textbf{Scheme-A}:  
	
Find $\left((\u^{f,n},\widehat \u^{s,n}),p^{f,n},\widehat{\bm d}^{s,n}, \widehat{\bm \eta}^{f,n}\right)\in \V^{FSI,n}\times Q^{f,n}\times\widehat \V^s_0 \times \widehat \V_D^{\mathcal M}$ for $n=1,2,\cdots,N$, with $\widehat{\bm \eta}^{f,n}=\widehat{\bm d}^{s,n}$ on $\widehat\Gamma$ and  $\bm w^{f,n}=\tilde{D}_t\widehat {\bm \eta}^{f,n}  \ \circ (\A^{n})^{-1}$, such that 
\begin{align}\label{semi_discrete_A}
\left\{
\begin{aligned}
\rho^f(\tilde{D}_t^M\u^{f,n}+\nabla{\bm{w}^{f,n}} \u^{f,n}-\tfrac{1}{2} \nabla\cdot {\bm{w}^{f,n}} \u^{f,n}, \v^f)_{\Of^{n}}\\+\rho^f T^{f,n}(\u^{f,n}-{\bm{w}^{f,n}}, \u^{f,n},\v^f)+2\rho^f\mu^fB^{f,n}(\u^{f,n},\v^f)
\\-Q^{f,n}(p^{f,n},\v^f )+\rho^s(\tilde{D}_t\widehat \u^{s,n},\widehat \v^s )_{\Ors}+2\mu^sB^s( \widehat{\bm{d}}^{s,n}  ,\widehat \v^s) 
\\+\lambda^s(\widehat\nabla\cdot \widehat{\bm{d}}^{s,n}, \widehat\nabla\cdot\widehat\v^s)_{\Ors}-(\bm g^{f,n}, \v^f)_{\Of^{n}}&=0,\\
	Q^{f,n}(q^f,\u^{f,n} )&=0,\\
	(\tilde{D}_t \widehat{\bm{d}}^{s,n},\ms )_{\Ors}-(\widehat \u^{s,n},\ms )_{\Ors}&=0,\\
	(\kappa^f\widehat\nabla \widehat{\bm{\eta}}^{f,n},\widehat\nabla \mf)_{\Orf} &=0, 
\end{aligned}
\right.
\end{align}
   for all $ \left((\v^f,\widehat\v^s), q^f, \ms,\mf\right)\in \V^{FSI,n}\times Q^{f,n}\times\widehat \V_0^s \times\widehat \V_0^{\mathcal M}$.   
Here and in what follows,
 \begin{align*}  
     &T^{f,n}(\bm{w}^{f,n}, \u^{f,n},\v^f) :=\tfrac{1}{2} \left( \v^f\cdot(\nabla\u^{f,n})- \u^{f,n}\cdot(\nabla\v^f),  \bm{w}^{f,n} \right)_{\Of^n},\\
     &B^{f,n}(\u^{f,n},\v^f) :=
   \left( \bm D(\u^{f,n}),\bm D(\v^f)\right)_{\Of^n},\qquad Q^{f,n}(q^f,\v^f ) :=\left(q^f,\nabla\cdot\v^f\right)_{\Of^n}.
\end{align*}
\item  \textbf{Scheme-B}:  
	
Find $\left((\u^{f,n},\widehat \u^{s,n}),p^{f,n},\widehat{\bm d}^{s,n}, \widehat{\bm \eta}^{f,n}\right)\in \V^{FSI,n}\times Q^{f,n}\times\widehat \V^s_0 \times \widehat \V_D^{\mathcal M}$  for $n=1,2,\cdots,N$, with $\widehat{\bm \eta}^{f,n}=\widehat{\bm d}^{s,n}$ on $\widehat\Gamma$ and  $\bm w^{f,n}=\tilde{D}_t\widehat {\bm \eta}^{f,n}  \ \circ (\A^{n})^{-1}$, such that 
\begin{align}\label{semi_discrete_B}
\left\{
\begin{aligned}
\frac{\rho^f}{\tau} \left((\u^{f,n}, \v^f)_{\Of^{n}}- (\u^{f,n-1}, \v^f\circ\X_{n-1}^{n} )_{\Of^{n-1}}\right)
\\-\rho^f\left((\nabla{\bm{w}^{f,n}} )^{\top}\u^{f,n}-\tfrac{1}{2} \nabla\cdot {\bm{w}^{f,n}} \u^{f,n}, \v^f\right)_{\Of^{n}}
+\rho^f T^{f,n}(\u^{f,n}-{\bm{w}^{f,n}}, \u^{f,n},\v^f) 
\\+2\rho^f\mu^fB^{f,n}(\u^{f,n},\v^f)-Q^{f,n}(p^{f,n},\v^f )+\rho^s(\tilde{D}_t\widehat \u^{s,n},\widehat \v^s )_{\Ors}
\\+2\mu^sB^s( \widehat{\bm{d}}^{s,n}  ,\widehat \v^s) +\lambda^s(\widehat\nabla\cdot \widehat{\bm{d}}^{s,n}, \widehat\nabla\cdot\widehat\v^s)_{\Ors}-(\bm g^{f,n}, \v^f)_{\Of^{n}}&=0,\\
	Q^{f,n}(q^f,\u^{f,n} )&=0,\\
	(\tilde{D}_t \widehat{\bm{d}}^{s,n},\ms )_{\Ors}-(\widehat \u^{s,n},\ms )_{\Ors}&=0,\\
	(\kappa^f\widehat\nabla \widehat{\bm{\eta}}^{f,n},\widehat\nabla \mf)_{\Orf} &=0, 
\end{aligned}
\right.
\end{align}
   for all $ \left((\v^f,\widehat\v^s), q^f, \ms,\mf\right)\in \V^{FSI,n}\times Q^{f,n}\times\widehat \V_0^s \times\widehat \V_0^{\mathcal M}$.
\end{itemize}

\begin{myremark}\label{AB-equal}
Notice that Scheme-A is constructed directly from  the weak formulation \eqref{weak_pro}, where the  term 
$ (\partial_t^M\u^f+\nabla\w\u^f-\tfrac{1}{2} \nabla\cdot\w\u^f, \v^f)_{\Of(t)}$ is discretized as 
 $$ (\tilde{D}_t^M\u^{f,n}+\nabla{\bm{w}^{f,n}} \u^{f,n}-\tfrac{1}{2} \nabla\cdot {\bm{w}^{f,n}} \u^{f,n}, \v^f)_{\Of^{n}}.$$
   Scheme-B is based on  \eqref{weak_pro} and the identity
   \begin{align*}
  (\partial_t^M\u^f+\nabla\w\u^f-\tfrac{1}{2} \nabla\cdot\w\u^f, \v^f)_{\Of(t)}
   = \frac{d}{dt}\bigl(\u,\v \bigr)_{\Of(t)}
   - \Bigl( (\nabla \w)^{\top}\u-\tfrac{1}{2} \nabla \cdot \w\,\u  , \v \Bigr)_{\Of(t)} 
   \end{align*}
  for $ \partial_t^M \v^f = 0 $ (cf.    Lemma \ref{time-conv}), where the right-hand side of the identity is   discretized as   
   $$
   \frac{(\u^{f,n}, \v^f)_{\Of^{n}}-(\u^{f,n-1}, \v^f\circ\X_{n-1}^{n} )_{\Of^{n-1}}}{\tau} 
   -\left((\nabla{\bm{w}^{f,n}} )^{\top}\u^{f,n}-\tfrac{1}{2} \nabla\cdot {\bm{w}^{f,n}} \u^{f,n}, \v^f\right)_{\Of^{n}}.
$$ 
In particular, Scheme-A and Scheme-B are   equivalent  if $\nabla{\bm{w}^{f,n}}=0$ for any $n$, which means that the fluid   domain is fixed all the time. 
\end{myremark}
 
 \subsection{Energy stability analysis}
 
 We first introduce some notations.
 For $n=1,2,\cdots,N$ and $t\in [t_{n-1},t_n]$  we define
 \begin{align}\label{def:FFG}
& \mathbb{F}_{n}(t):=\widehat{\bm F}^t_f(\widehat{\bm F}^n_f)^{-1},\quad \mathbb{J}_{n}(t):=\frac{\widehat J_f^n}{\widehat J_f^t}, \quad \mathbb{G}_{n}(t):=\mathbb{J}_{n}(t)\mathbb{F}_{n}(t),
 \end{align}
 and denote 
\begin{align}\label{def:G}
\left\{
\begin{aligned}
	G_{1, n}:=&  \sup_{t\in[t_{n-1},t_n]}
	\Bigl(1+ \tfrac12\|(\mathbb{G}_{n-1}(t))^{-1}\|_{L^\infty({\Omega_f^{n-1}})} \|\mathbb{G}_{n-1}(t)\|_{L^\infty({\Omega_f^{n-1}})}\Bigr)  \|\mathbb{G}_{n-1}(t)\|_{L^\infty({\Omega_f^{n-1}})},\\
	 G_{2, n}  :=&\sup_{t\in [t_{n-1},t_n]}\tfrac{1}{2} \Bigl(\|\mathbb{J}_{n}(t)\|_{L^\infty(\Omega_f^n)}+\bigl(2+\tfrac12\|\mathbb{F}_{n}(t)\|^2_{L^\infty(\Omega_f^n)}\bigr) \|(\mathbb{F}_{n}(t))^{-\top}\|_{L^\infty(\Omega_f^n)}\|\mathbb{G}_{n}(t)\|_{L^\infty(\Omega_f^n)} 
	\\ &\ +\tfrac12\|(\mathbb{F}_{n}(t))^{-\top}\|^2_{L^\infty(\Omega_f^n)}\|\mathbb{F}_{n}(t)\|_{L^\infty(\Omega_f^n)}\|\mathbb{G}_{n}(t)\|_{L^\infty(\Omega_f^n)}
	 \Bigr),
	\end{aligned}
	\right.	
\end{align}
where $\|\mathbb A\|_{L^\infty(\Omega_f^n)}:=\operatorname*{ess\,sup}\limits_{x\in\Omega_f^n}\left(\mathbb A\bm{:} \mathbb A\right)^{1/2}$ for any $d\times d$ matrix function $\mathbb A(\x)$ with ` $\bm{:}$ ' denoting the Frobenius inner product.

We have the following lemma for the stability analysis of the two semi-discrete schemes.
\begin{mylemma}
Let  $ \widehat\v \in  \bm H(\widehat {\rm{div}},\Orf)$ and $\v=\P^t(\widehat{\v})\circ (\A^t)^{-1}  $. Then for $n=1,2,\cdots,N$ there hold
\begin{align}\label{eq::rem_J}
&\int_{t_{n-1}}^{t_n}\Bigl((\nabla \bm w_t^{f,n}-\tfrac12\nabla\!\cdot \bm w_t^{f,n})\v,\v\Bigr)_{\Omega_f(t)} dt
\leq 
  \tau G_{1, n}\|\nabla{\bm{w}}_{n-1}^{f,n}\|_{L^\infty({\Omega_f^{n-1}})}
\|\v^{n-1}\|^2_{L^2({\Omega_f^{n-1}})}
\end{align}
and 
\begin{align}\label{eq::rem_J2}
\begin{aligned}
&\int_{t_{n-1}}^{t_n}\Bigl((\nabla \bm w_t^{f,n}-\tfrac12\nabla\!\cdot \bm w_t^{f,n})\v,\v\Bigr)_{\Omega_f(t)} dt\, - \tau\Bigl((\nabla \bm w^{f,n}-\tfrac12\nabla\!\cdot \bm w^{f,n})\v^n,\v^n\Bigr)_{\Omega_f^n}\\
 \leq &\ \  \tau^2G_{2, n}\|\nabla\bm w^{f,n}\|^2_{L^\infty(\Omega_f^n)}
\|\v^{n}\|^2_{L^2({\Omega_f^{n}})}.
\end{aligned}
\end{align}
\begin{proof} We first show \eqref{eq::rem_J}. 
Apply \eqref{semi-fluid-velocity},  \eqref{J_F_D}  and variable substitution to get
\begin{align}\label{gt}
	g(t) :=&\Bigl((\nabla \bm w_t^{f,n}-\tfrac12\nabla\!\cdot \bm w_t^{f,n})\v,\v\Bigr)_{\Omega_f(t)}\nonumber\\
	=&\left(\left(\widehat\nabla\widehat{\bm  w}^{f,n}(\widehat{\bm F}_f^{t})^{-1}-\tfrac12((\widehat{\bm F}_f^{t})^{-\top}\bm{:}\widehat\nabla\widehat{\bm w}^{f,n} )\,\bm I \right) \frac{\widehat{\bm F}^t_f}{\widehat J_f^t}\widehat{\bm v},\frac{\widehat{\bm  F}^t_f}{\widehat J_f^t}\widehat{\v}\,\widehat J_f^t\right)_{\widehat\Omega_f}\nonumber\\
	=&\Bigl(\widehat\nabla\widehat{\bm  w}^{f,n}\widehat{\bm v},\frac{\widehat{\bm  F}^t_f}{\widehat J_f^t}\widehat{\v}\,\Bigr)_{\widehat\Omega_f}-\tfrac12\Bigl(((\widehat{\bm F}_f^{t})^{-\top}\bm{:}\widehat\nabla\widehat{\bm w}^{f,n} ) \frac{\widehat{\bm F}^t_f}{\widehat J_f^t}\widehat{\bm v},\widehat{\bm  F}^t_f\widehat{\v}\Bigr)_{\widehat\Omega_f}.
\end{align}
Using variable substitution again, in view of  \eqref{def:FFG} and the identity $A\bm{:}(BC)=(A C^{\top})\bm{:}B$ for $d\times d$ matrices $A,B$ and $C$  we further have 
\begin{align*}
\begin{split}
g(t)=&
\left( \nabla{\bm{w}}_{n-1}^{f,n} \widehat{\bm F}_f^{n-1}(\widehat{\bm F}_f^{t})^{-1} \mathbb{G}_{n-1}(t)\v^{n-1},\mathbb{G}_{n-1}(t)\v^{n-1} \frac{\widehat J_f^t}{\widehat J_f^{n-1}} \right)_{\Omega_f^{n-1}}  
\\
&
\quad -\tfrac12\left((\widehat{\bm F}_f^{t})^{-\top}\!\bm{:}(\nabla{\bm w}_{n-1}^{f,n}\widehat{\bm F}_f^{n-1}  )
 \mathbb{G}_{n-1}(t)\v^{n-1},\mathbb{G}_{n-1}(t)\v^{n-1} \frac{\widehat J_f^t}{\widehat J_f^{n-1}} \right)_{\Omega_f^{n-1}}  
\\
=& 
\left(\frac{\widehat J_f^t}{\widehat J_f^{n-1}}  \nabla{\bm{w}}_{n-1}^{f,{n}} \widehat{\bm F}_f^{{n-1}}(\widehat{\bm F}_f^{t})^{-1} 
 \ \mathbb{G}_{{n-1}}(t)\v^{n-1},\mathbb{G}_{{n-1}}(t)\v^{n-1} \right)_{\Omega_f^{{n-1}}}   
\\
& 
\quad -\tfrac12 \left(\frac{\widehat J_f^t}{\widehat J_f^{n-1}}  \Bigl(\bigl(\widehat{\bm F}_f^{{n-1}}\bigl(\widehat{\bm F}_f^{t})^{-1}\bigr)^{\top}\!\bm{:}\nabla{\bm w}_{n-1}^{f,{n}} 
\Bigr)\ \mathbb{G}_{{n-1}}(t)\v^{n-1},\mathbb{G}_{{n-1}}(t)\v^{n-1} \right)_{\Omega_f^{{n-1}}}   
\\
=& \left( \nabla{\bm{w}}_{n-1}^{f,{n}}\v^{n-1},\mathbb{G}_{{n-1}}(t)\v^{n-1} \right)_{\Omega_f^{{n-1}}}  - \tfrac12 
\left(  \bigl((\mathbb{G}_{{n-1}}(t))^{- \top}\!\bm{:}\nabla{\bm w}_{n-1}^{f,n}\bigr)\mathbb{G}_{n-1}(t)\v^{n-1},\mathbb{G}_{n-1}(t)\v^{n-1} \right)_{\Omega_f^{n-1}} ,
\end{split}
\end{align*}
which plus   the H\"{o}lder's  inequality  gives
\begin{align*}
\int_{t_{n-1}}^{t_n}g(t)dt
  \leq &\int_{t_{n-1}}^{t_n}
\Bigl (1+ \tfrac12\|(\mathbb{G}_{n-1}(t))^{-1}\|_{L^\infty({\Omega_f^{n-1}})} \|\mathbb{G}_{n-1}(t)\|_{L^\infty({\Omega_f^{n-1}})}\Bigr)  
\\ &\quad \cdot\|\mathbb{G}_{n-1}(t)\|_{L^\infty({\Omega_f^{n-1}})}\|\nabla{\bm{w}}_{n-1}^{f,n}\|_{L^\infty({\Omega_f^{n-1}})}\|\v^n\|^2_{L^2({\Omega_f^{n-1}})} dt\\
\leq  &\, \tau G_{1,n} \|\nabla{\bm{w}}_{n-1}^{f,n}\|_{L^\infty({\Omega_f^{n-1}})}\|\v^{n-1}\|^2_{L^2({\Omega_f^{n-1}})},	
\end{align*}
i.e. the desired estimate \eqref{eq::rem_J} follows.

Next let us prove \eqref{eq::rem_J2}. In view of \eqref{gt}, \eqref{J_F_D} and \eqref{def:FFG}, we have
\begin{align*}
	\partial_t g(t)=&\Bigl(\widehat\nabla\widehat{\bm  w}^{f,n}\widehat{\bm v},\frac{\widehat\nabla\widehat{\bm  w}^{f,n}}{\widehat J_f^t}\widehat{\v}\,\Bigr)_{\widehat\Omega_f}-\Bigl(\widehat\nabla\widehat{\bm w}^{f,n}\widehat{\bm v},\frac{\bigl((\widehat{\bm F}_f^{t})^{-\top}\bm{:}\widehat\nabla\widehat{\bm w}^{f,n}\bigr)\widehat{\bm  F}^t_f}{\widehat J_f^t}\widehat{\v}\,\Bigr)_{\widehat\Omega_f}
\\
&\quad -\tfrac12\Bigl(\bigl((\widehat{\bm F}^t_f)^{-\top} (\widehat\nabla \widehat{\bm w}^{f,n})^{\top}(\widehat{\bm F}^t_f)^{-\top}\bigr) {\bm :}\widehat\nabla \widehat{\bm w}^{f,n}
,\frac{1}{\widehat J_f^t} \bigl|\widehat{\bm F}_f^t\widehat{\bm v}\bigr|^2 \Bigr)_{\widehat\Omega_f}
\\
&\quad +\tfrac12\Bigl(\bigl((\widehat{\bm F}_f^{t})^{-\top}\bm{:} \widehat\nabla \widehat{\bm w}^{f,n}\bigr)\frac{1}{\widehat J_f^t}\bigl((\widehat{\bm F}_f^t)^{-\top}\bm{:} \widehat\nabla \widehat{\bm w}^{f,n}\bigr),\bigl|\widehat{\bm F}_f^t\widehat{\bm v}\bigr|^2\Bigr)_{\widehat\Omega_f}
\\
&\quad -\Bigl(\bigl((\widehat{\bm F}_f^{t})^{-\top}\bm{:} \widehat\nabla \widehat{\bm w}^{f,n}\bigr) \frac{\nabla\widehat{\bm w}^{f,n}\widehat\v }{{\widehat J_f^t}}, \widehat{\bm  F}^t_f\widehat{\v}\Bigr)_{\widehat\Omega_f}\\
=&\Bigl(\widehat J_f^n\widehat\nabla\widehat{\bm  w}^{f,n}(\widehat {\bm F}^n_f)^{-1}\frac{\widehat {\bm F}^n_f}{\widehat J_f^n}\widehat{\bm v},\frac{\widehat\nabla\widehat{\bm  w}^{f,n}(\widehat {\bm F}^n_f)^{-1}}{\widehat J_f^t}\frac{\widehat {\bm F}^n_f}{\widehat J_f^n} \widehat J_f^n\widehat{\v}\,\Bigr)_{\Orf}
\\ 
	 &\quad -\Bigl(\widehat J_f^n\widehat\nabla\widehat{\bm w}^{f,n}(\widehat {\bm F}^n_f)^{-1}\frac{\widehat {\bm F}^n_f}{\widehat J_f^n}\widehat{\bm v},\frac{\bigl((\widehat{\bm F}_f^{t}(\widehat {\bm F}^n_f)^{-1})^{-\top}\bm{:}\widehat\nabla\widehat{\bm w}^{f,n}(\widehat {\bm F}^n_f)^{-1}\bigr)\widehat{\bm  F}^t_f(\widehat {\bm F}^n_f)^{-1}}{\widehat J_f^t}\frac{\widehat {\bm F}^n_f}{\widehat J_f^n}\widehat{\v}\widehat J_f^n\,\Bigr)_{\widehat\Omega_f}
	\\
	&\quad -\tfrac12\Bigl(\mathbb{F}_{n}(t)\nabla{\bm w}^{f,n}(\mathbb{F}_{n}(t))^{-\top}\bm {:} \nabla{\bm w}^{f,n}\mathbb{G}_{n}(t)\v^n, \mathbb{F}_{n}(t)\v^n \Bigr)_{\Omega_f^n}\\
&\quad +\tfrac12\Bigl(|(\mathbb{F}_{n}(t))^{-\top}\bm {:} \nabla{\bm w}^{f,n}|^2\mathbb{G}_{n}(t)\v^n, \mathbb{F}_{n}(t)\v^n \Bigr)_{\Omega_f^n}
		\\
  &\quad - \Bigl((\mathbb{F}_{n}(t))^{-\top}: \nabla{\bm w}^{f,n})\nabla{\bm w}^{f,n}\v^n, \mathbb{G}_{n}(t)\v^n\Bigr)_{\Omega_f^n}\\
  =&\Bigl(\mathbb{J}_{n}(t)\nabla\bm w^{f,n}\v^n , \nabla\bm w^{f,n}\v^n \Bigr)_{\Omega_f^n}
	   -\Bigl(\nabla\bm w^{f,n}\v^n, (\mathbb{F}_{n}(t))^{-\top}\bm{:}\nabla\bm w^{f,n}\mathbb{G}_{n}(t)\v^n \Bigr)_{\Omega_f^n}
	\\
	&\quad -\tfrac12\Bigl(\mathbb{F}_{n}(t)\nabla{\bm w}^{f,n}(\mathbb{F}_{n}(t))^{-\top}\bm {:} \nabla{\bm w}^{f,n}\mathbb{G}_{n}(t)\v^n, \mathbb{F}_{n}(t)\v^n \Bigr)_{\Omega_f^n}\\
&\quad +\tfrac12\Bigl(|(\mathbb{F}_{n}(t))^{-\top}\bm {:} \nabla{\bm w}^{f,n}|^2\mathbb{G}_{n}(t)\v^n, \mathbb{F}_{n}(t)\v^n \Bigr)_{\Omega_f^n}
		\\
  &\quad - \Bigl((\mathbb{F}_{n}(t))^{-\top}: \nabla{\bm w}^{f,n})\nabla{\bm w}^{f,n}\v^n, \mathbb{G}_{n}(t)\v^n\Bigr)_{\Omega_f^n},
\end{align*}
which plus the H\"{o}lder's inequality and \eqref{def:G} yields
\begin{align*}
	|\partial_t g(t)|\leq &\sup_{t\in [t_{n-1},t_n]}\Bigl(\|\mathbb{J}_{n}(t)\|_{L^\infty(\Omega_f^n)}+\bigl(2+\tfrac12\|\mathbb{F}_{n}(t)\|^2_{L^\infty(\Omega_f^n)}\bigr) \|(\mathbb{F}_{n}(t))^{-\top}\|_{L^\infty(\Omega_f^n)}\|\mathbb{G}_{n}(t)\|_{L^\infty(\Omega_f^n)} 
	\\ &\ +\tfrac12\|(\mathbb{F}_{n}(t))^{-\top}\|^2_{L^\infty(\Omega_f^n)}\|\mathbb{F}_{n}(t)\|_{L^\infty(\Omega_f^n)}\|\mathbb{G}_{n}(t)\|_{L^\infty(\Omega_f^n)}
	 \Bigr)\|\nabla\bm w^{f,n}\|^2_{L^\infty(\Omega_f^n)}\|\v^n\|^2_{L^2(\Omega_f^n)}\\
\leq & 2G_{2,n}	 \ \|\nabla\bm w^{f,n}\|^2_{L^\infty(\Omega_f^n)}\|\v^n\|^2_{L^2(\Omega_f^n)}.
\end{align*}
This inequality,  together with the relation
\begin{equation}\label{eq:Taylor}
\int_{t_{n-1}}^{t_n} g(t)\,dt
- \tau g(t_n)= - \tfrac{\tau^2}{2}\,\partial_t g(\xi) \quad  \text{ for some }  \xi\in(t_{n-1},t_n)
\end{equation}
due to the Taylor expansion,
implies the desired estimate \eqref{eq::rem_J2}.  This completes the proof.
\end{proof}
\end{mylemma}

For the stability analysis we also  need  the following discrete Gr{\"o}nwall's inequality (cf. \cite{Heywood, Quarteroni}):
\begin{mylemma}\label{Growall}
	  Let $n$ be a positive  integer, and let $\tau $, $\rm B$ and $a_k, b_k, c_k, {\gamma}_k$  $(k=1,\cdots,n)$ be nonnegative numbers such that
	\begin{equation}\label{Gro1}
		a_n+\tau\sum_{k=1}^n b_k\leq \tau\sum_{k=1}^n\gamma_ka_k+\tau\sum_{k=1}^nc_k+ \rm B, \qquad for \ n\geq 1.
	\end{equation}
	Suppose that $\tau\gamma_k <1 $  for all $k$, and set $\delta_k=(1-\tau\gamma_k)^{-1}$. Then 
	\begin{equation}\label{Gro2}
		a_n+\tau\sum_{k=1}^nb_k\leq \rm{exp}\big(\tau\sum_{k=1}^n\gamma_k\delta_k \big)\big(\tau\sum_{k=1}^nc_k+B \big).
	\end{equation}
\end{mylemma}

We are now at a position to define the semi-discrete energy. Let $E^{n}$ denote the sum of the semi-discrete kinetic energy of the fluid and the semi-discrete kinetic and   elastic energy of the structure at time $t_n$, i.e.
\begin{align*}
	E^{n}:=\frac{\rho^f}{2}\|\u^{f,n}\|^2_{L^2({\Of^n})}+\frac{\rho^s}{2} \|\widehat \u^{s,n}\|^2_{L^2(\Ors) }+\left (\mu^s\|\D(\widehat{\bm{d}}^{s,n}) \|^2_{L^2(\Ors )}+\frac{\lambda^s }{2}\|\widehat\nabla\cdot {\widehat{\bm{d}}^{s,n} }\|^2_{L^2(\Ors )}\right ),
\end{align*}
and let $\tilde E_f^{n}$ and $\tilde E_{s,d}^{n}$ denote respectively the fluid viscous dissipation and the    structure numerical dissipation at time $t_n$ with
\begin{align*}
	\tilde E^{n}_f:=&\tau\rho^f\mu^f\|\D(\u^{f,n})\|^2_{L^2(\Of^{n})},\\
	\tilde E_{s,d}^{n}:=&\tfrac{1}{2}\bigl(\rho^s\|\widehat \u^{s,n}-\widehat \u^{s,n-1}\|^2_{L^2({\Ors})}+ 2\mu^s\|\D(\widehat{\bm{d}}^{s,n})-\D(\widehat{\bm{d}}^{s,n-1})\|^2_{L^2({\Ors})}
	+\lambda^s\|\widehat{\nabla}\cdot\widehat{\bm{d}}^{s,n}-\widehat{\nabla}\cdot\widehat{\bm{d}}^{s,n-1}\|^2_{L^2({\Ors})}\bigr).
\end{align*}

Theorems \ref{theorem_A} and \ref{theorem_B} show energy estimates of the two semi-discrete schemes.
\begin{mytheorem}[Stability of {Scheme-A}]
\label{theorem_A}
	Let $(\u^{f,n},\widehat \u^{s,n},p^{f,n},\widehat{\bm d}^{s,n}, \widehat{\bm \eta}^{f,n})$ solve the semi-discrete system \eqref{semi_discrete_A} for . If the time step $\tau$ satisfies
	\[
	\tau <\tau_0:=1/\max\limits_{n=1,\cdots,N}\Bigl(\tfrac{3}{2}\|\nabla {\bm{w}^{f,n}}\|_{L^\infty ({\Of^{n}})}+G_{1,n}\|\nabla {\bm{w}_{n-1}^{f,n}}\|_{L^\infty ({\Of^{n-1}})}\Bigr )
	\] 
	then the following energy estimate holds: 
	\begin{align}\label{sta_1}
		&E^N+\sum_{n=1}^{N}\left(\tilde E^{n}_f+\tilde E^{n}_{s,d}+\tfrac{\rho^f}{2}\|\u^{f,n-1}\circ\X^{n-1}_{n}-\u^{f,n}\|^2_{L^2(\Of^{n})}\right)\lesssim \tau\sum_{n=1}^{N}\|\bm g^{f,n}\|^2 _{L^{2}(\Of^{n}) }+E^0.
	\end{align}
	\begin{proof}
		Taking $(\v^f,\widehat\v^s, q^f)=(\tau\u^{f,n},\tau\widehat \u^{s,n},p^{f,n})$ in the first two equations of \eqref{semi_discrete_A},  using the skew-symmetry property of $T^{f,n}(\cdot,\cdot,\cdot)$ and the relation $\tilde{D}_t \widehat{\bm{d}}^{s,n}=\widehat \u^{s,n}$ due to the third equation of \eqref{semi_discrete_A}, and applying  the identity $2a(a-b)=a^2-b^2+(a-b)^2$ for the terms $(\tilde{D}_t\widehat \u^{s,n},\widehat \u^{s,n} )_{\Ors}$ and $B^s( \widehat{\bm{d}}^{s,n}  ,\widehat \v^s) $, we get
		\begin{align}\label{en-sta-relation1}
			&\rho^f \tau ( \tilde{D}_t^M\u^{f,n}, \u^{f,n})_{\Of^{n}}+\tau\rho^f (\nabla{\bm{w}^{f,n}} \u^{f,n}-\tfrac{1}{2} \nabla\cdot{\bm{w}^{f,n}} \u^{f,n}, \u^{f,n})_{\Of^{n}}\nonumber
			\\&+2\tau\rho^f\mu^f\|\bm D(\u^{f,n})\|^2_{L^2(\Of^{n})}+\tfrac{\rho^s}{2}\left(\|\widehat \u^{s,n}\|^2_{L^2({\Ors})}-\|\widehat \u^{s,n-1}\|^2_{L^2({\Ors})}+\|\widehat \u^{s,n}-\widehat \u^{s,n-1}\|^2_{L^2({\Ors})} \right)\nonumber
			\\&+\mu^s(\|\D(\widehat{\bm{d}}^{s,n})\|^2_{L^2({\Ors})}-\|\D(\widehat{\bm{d}}^{s,n-1})\|^2_{L^2({\Ors})}+\|\D(\widehat{\bm{d}}^{s,n})-\D(\widehat{\bm{d}}^{s,n-1})\|^2_{L^2({\Ors})} )\nonumber
			\\&+\tfrac{\lambda^s}{2}(\|\widehat{\nabla}\cdot\widehat{\bm{d}}^{s,n}\|^2_{L^2({\Ors})}-\|\widehat{\nabla}\cdot\widehat{\bm{d}}^{s,n-1}\|^2_{L^2({\Ors})}+\|\widehat{\nabla}\cdot\widehat{\bm{d}}^{s,n}-\widehat{\nabla}\cdot\widehat{\bm{d}}^{s,n-1}\|^2_{L^2({\Ors})} )\nonumber
			\\=&\tau(\bm g^{f,n},  \u^{f,n})_{\Of^{n}} \nonumber\\
			 \leq & \epsilon \tau\|\bm g^{f,n}\|^2 _{L^{2}(\Of^{n}) } + \tfrac{\tau}{4\epsilon}\| \u^{f,n} \|^2_{L^2(\Of^{n}) }, \quad \forall \epsilon>0.
		\end{align}
For the first term of \eqref{en-sta-relation1}, by \eqref{X_de} and the identity $-2ab=-a^2-b^2+(a-b)^2$ we get
		\begin{align}\label{en-sta-relation2}
			 \tau(\tilde{D}_t^M\u^{f,n},\u^{f,n})_{\Of(t_{n})}=& (\u^{f,n}, \u^{f,n})_{\Of^{n}}-\rho^f(\u^{f,n-1}\circ \X_{n}^{n-1}, \u^{f,n})_{\Of^{n}} \nonumber\\
			=&\tfrac{1}{2} \|\u^{f,n}\|^2_{L^2({\Of^{n}})}-\tfrac{1}{2}( |\P^{n}(\widehat\u ^{f,n-1}) |^2,\widehat \J_f^{n}  )_{\Orf}+\tfrac{1}{2} \|\u^{f,n-1}\circ\X^{n-1}_{n}-\u^{f,n}\|^2_{L^2(\Of^{n})}
			\nonumber\\=&\tfrac{1}{2} \Bigl(\|\u^{f,n}\|^2_{L^2({\Of^{n}})}-\|\u^{f,n-1}\|^2_{L^2(\Of^{n-1})}\Bigr)+\tfrac{1}{2} \|\u^{f,n-1}\circ\X^{n-1}_{n}-\u^{f,n}\|^2_{L^2(\Of^{n})}
			\nonumber\\&-\tfrac{1}{2}\left(\|\P^{n}(\widehat\u^{f,n-1})\circ (\A^{n})^{-1}\|^2_{L^2(\Of^{n} )}-\|\u^{f,n-1}\|^2_{L^2(\Of^{n-1})}\right).
		\end{align} 
		For the last term of the above equation,  we apply \eqref{GCL_d-3} and \eqref{eq::rem_J} with 
		$\v= \P^{t}(\widehat\u^{f,n-1})\circ (\A^{t})^{-1}  $ to obtain
		\begin{align}\label{sta1_eq_1}
		\begin{split}
			&\|\P^{n}(\widehat\u^{f,n-1})\circ (\A^{n})^{-1}\|^2_{L^2(\Of^{n} )}-\|\u^{f,n-1}\|^2_{L^2(\Of^{n-1})} \\
			=&\int^{t_{n}}_{t_{n-1}} \left((\nabla {\bm{w}_t^{f,n}} -\tfrac12\nabla\cdot \bm{w}_t^{f,n})\P^{t}(\widehat\u^{f,n-1})\circ (\A^{t})^{-1}  ,  \P^{t}(\widehat\u^{f,n-1})\circ (\A^{t})^{-1}\right)_{\Of(t)} dt
			\\ \leq &\,\tau G_{1, n}\|\nabla {\bm{w}_{n-1}^{f,n}}\|_{L^\infty ({\Of^{n-1}})} \|\u^{f,n-1}\|^2_{L^2({\Of^{n-1}})}.
		\end{split}
		\end{align}
		 Moreover, for the second term of \eqref{en-sta-relation1} we have 
		\begin{align}\label{sta1_eq_2}
			-(\nabla{\bm{w}^{f,n}} \u^{f,n}-\tfrac{1}{2} \nabla\cdot{\bm{w}^{f,n}} \u^{f,n}, \u^{f,n})_{\Of^{n}}\leq \tfrac{3}{2}\|\nabla{\bm{w}^{f,n}}\|_{L^\infty ({\Of^{n}})}\|\u^{f,n}\|^2_{L^2({\Of^{n}})}.
		\end{align}
			Combining \eqref{en-sta-relation1}-\eqref{sta1_eq_2},  summing them from $n=1$ to $N$, using the Poincar\'e inequality and Korn's inequality to get $\|\u^{f,n}\|_{L^2(\Omega_f^{n})}^2 \le C_K \|\bm D(\boldsymbol u^{f,n})\|_{L^2(\Omega_f^{n})}^2$, and taking $\epsilon=\tfrac{4C_k}{\rho^f\mu^f}$, we have
		\begin{align*}
			&\tfrac{1}{2}\left (\rho^f\|\u^{f,N}\|^2_{L^2({\Of^{N}})}+\rho^s\|\widehat \u^{s,N}\|^2_{L^2({\Ors})}+2\mu^s\|\D(\widehat{\bm{d}}^{s,N})\|^2_{L^2({\Ors})}+\lambda^s\|\widehat{\nabla}\cdot\widehat{\bm{d}}^{s,N}\|^2_{L^2({\Ors})}\right )
			\\ &+\tfrac{1}{2}\sum_{n=1}^{N}\Bigr(4\tau\rho^f\mu^f\|\nabla\u^{f,n}\|^2_{L^2(\Of^{n})}+\rho^f\|\u^{f,n-1}\circ\X^{n-1}_{n}-\u^{f,n}\|^2_{L^2(\Of^{n})}+\rho^s\|\widehat \u^{s,n}-\widehat \u^{s,n-1}\|^2_{L^2({\Ors})}
			\\ &+ 2\mu^s\|\D(\widehat{\bm{d}}^{s,n})-\D(\widehat{\bm{d}}^{s,n-1})\|^2_{L^2({\Ors})}+\lambda^s\|\widehat{\nabla}\cdot\widehat{\bm{d}}^{s,n}-\widehat{\nabla}\cdot\widehat{\bm{d}}^{s,n-1}\|^2_{L^2({\Ors})}\Bigr )
			\\ \leq &\, \tau \rho^f\sum_{n=1}^{N}\Bigl( G_{1, n}\|\nabla {\bm{w}_{n-1}^{f,n}}\|_{L^\infty ({\Of^{n-1}})} \|\u^{f,n-1}\|^2_{L^2({\Of^{n-1}})}+ \tfrac{3}{2}\|\nabla {\bm{w}^{f,n}}\|_{L^\infty ({\Of^{n}})}\|\u^{f,n}\|^2_{L^2({\Of^{n}})}\Bigr)
			\\&+\epsilon\tau \sum_{n=1}^{N}\|\bm g^{f,n}\|^2 _{L^{2}(\Of^{n}) } +\tfrac{1}{2}\Bigl (\rho^f\|\u^{f,0}\|^2_{L^2({\Of^{N}})}+\rho^s\|\widehat \u^{s,0}\|^2_{L^2({\Ors})}
			+2\mu^s\|\D(\widehat{\bm{d}}^{s,0})\|^2_{L^2({\Ors})}+\lambda^s\|\widehat{\nabla}\cdot\widehat{\bm{d}}^{s,0}\|^2_{L^2({\Ors})}\Bigr).
		\end{align*}
		Applying the discrete Gr\"{o}nwall's inequality in Lemma \ref{Growall}, we finally obtain 
		\begin{align*}
			&E^N+\sum_{n=1}^{N}\left(\tilde E^{n}_f+\tilde E_{s,d}^{n}+\tfrac{\rho^f}{2}\|\u^{f,n-1}\circ\X^{n-1}_{n}-\u^{f,n}\|^2_{L^2(\Of^{n})}\right)\\
			\leq &\exp\left(\tau\sum_{n=1}^{N}\frac{\gamma_{n,1}}{1-\tau\gamma_{n,1}}\right)\left(\epsilon\tau\sum_{n=1}^{N}\|\bm g^{f,n}\|^2 _{L^{2}(\Of^{n})}+E^0\right), 
		\end{align*}
		where 
		\begin{align*}
			&\gamma_{n,1}:=\tfrac{3}{2}\|\nabla {\bm{w}^{f,n}}\|_{L^\infty ({\Of^{n}})}+G_{1,n}\|\nabla {\bm{w}_{n-1}^{f,n}}\|_{L^\infty ({\Of^{n-1}})}.
		\end{align*}
Since $\tau_0=1/\max\limits_{n=1,\cdots,N}\gamma_{n,1}$,   the desired stability result \eqref{sta_1} follows when $\tau<\tau_0$.
	\end{proof}
\end{mytheorem}

\begin{mytheorem}[Stability of   {Scheme-B}]\label{theorem_B}
	 Let $(\u^{f,n},\widehat \u^{s,n},p^{f,n},\widehat{\bm d}^{s,n}, \widehat{\bm \eta}^{f,n})$ solve the semi-discrete system \eqref{semi_discrete_B} for $n=1,2,\cdots,N$. If the time step $\tau$ satisfies
	\[
	\tau^2<\tau_0':=1/\max\limits_{n=1,\cdots,N}\Bigl(G_{2,n}\|\nabla {\bm{w}^{f,n}}\|^2_{L^\infty ({\Of^{n}})}\Bigr ),
	\]
	then  the following  energy estimate holds: 
	\begin{align}\label{sta_2}
		&E^N+\sum_{n=1}^{N}\left(\tilde E^{n}_f+\tilde E_{s,d}^{n}+\tfrac{\rho^f}{2}\|\u^{f,n-1}-\u^{f,n}\circ \X_{n-1}^{n}\|^2_{L^2({\Of^{n-1}})}\right)\lesssim \tau\sum_{n=1}^{N}\|\bm g^{f,n}\|^2 _{L^{2}(\Of^{n}) }+E^0.
	\end{align} 
\begin{proof}
	The proof is similar to that of Theorem \ref{theorem_A}, and  the main difference   lies in the treatment of the ALE time derivative term. Using \eqref{X_de} and identity $-2ab=-a^2-b^2+(a-b)^2$, we have
		\begin{align}\label{themB-pr1}
			&(\u^{f,n}, \u^{f,n})_{\Of^{n}}-(\u^{f,n-1}, \u^{f,n}\circ \X_{n-1}^{n} )_{\Of^{n-1}} \nonumber \\
			=&\|\u^{f,n}\|^2_{L^2(\Of^{n})}-\tfrac{1}{2}\left (\|\u^{f,n-1}\|^2_{L^2(\Of^{n-1})}+( |\P^{n-1}(\widehat\u ^{f,n}) |^2,\widehat \J^{n-1}  )_{\Orf}\right)+\tfrac{1}{2}\|\u^{f,n-1}-\u^{f,n}\circ \X_{n-1}^{n}\|^2_{L^2({\Of^{n-1}})}  \nonumber  \\
			=&\tfrac{1}{2}\left(\|\u^{f,n}\|^2_{L^2(\Of^{n})}-\|\u^{f,n-1}\|^2_{L^2(\Of^{n-1})}\right )+\tfrac{1}{2}\left(\|\u^{f,n}\|^2_{L^2(\Of^{n})}-\|\P^{n-1}(\widehat\u ^{f,n})\circ (\A^{n-1})^{-1}\|^2_{L^2(\Of^{n-1})}\right)
			 \nonumber 
			 \\&+\tfrac{1}{2}\|\u^{f,n-1}-\u^{f,n}\circ \X_{n-1}^{n}\|^2_{L^2({\Of^{n-1}})}. 
		\end{align}
      For the second term on the right hand side of the last ``=", we apply  \eqref{GCL_d-3}  to get
		\begin{align*}
			&\tfrac{1}{2}\left(\|\u^{f,n}\|^2_{L^2(\Of^{n})}-\|\P^{n-1}(\widehat\u ^{f,n})\circ (\A^{n-1})^{-1}\|^2_{L^2(\Of^{n-1})}\right)
            \\=&\int^{t_{n}}_{t_{n-1}} \left((\nabla {\bm{w}_t^{f,n}} -\tfrac12\nabla\cdot \bm{w}_t^{f,n})\P^{t}(\widehat\u^{f,n})\circ (\A^{t})^{-1}  ,  \P^{t}(\widehat\u^{f,n})\circ (\A^{t})^{-1}\right)_{\Of(t)} dt,
		\end{align*}
		which plus  \eqref{eq::rem_J2} gives
		\begin{align}\label{themB-pr2}
			&\tfrac{1}{2}\left(\|\u^{f,n}\|^2_{L^2(\Of^{n})}-\|\P^{n-1}(\widehat\u ^{f,n})\circ (\A^{n-1})^{-1}\|^2_{L^2(\Of^{n-1})}\right)-\tau\Bigl((\nabla{\bm{w}^{f,n}} )^{\top}\u^{f,n}-\tfrac{1}{2} \nabla\cdot {\bm{w}^{f,n}} \u^{f,n}, \u^{f,n}\Bigr)_{\Of^{n}}
			 \nonumber 
			 \\ \leq &\,  \tau^2\,G_{2,n}\|\nabla {\bm{w}^{f,n}}\|^2_{L^\infty ({\Of^{n}})} \|\u^{f,n}\|^2_{L^2(\Of^{n})}.
		\end{align}
				 
		 Taking $(\v^f,\widehat\v^s, q^f)=(\tau\u^{f,n},\tau\widehat \u^{s,n},p^{f,n})$ in \eqref{semi_discrete_B} , using the skew-symmetric property of $T^{f,n}(\cdot,\cdot,\cdot)$, adding the resulting  equations,  summing them from $n=1$ to $N$, and applying  \eqref{themB-pr1}, \eqref{themB-pr2}, and the discrete Gr\"{o}nwall’s inequality in Lemma \ref{Growall}, we finally obtain
		 \begin{align*}
		 	&E^N+\sum_{n=1}^{N}\left(\tilde E^{n}_f+\tilde E_{s,d}^{n}+\tfrac{\rho^f}{2}\|\u^{f,n-1}-\u^{f,n}\circ \X_{n-1}^{n}\|^2_{L^2({\Of^{n-1}})}\right)\\
		 	\leq &\exp\left(\tau\sum_{n=1}^{N}\frac{\tau \gamma_{n,2}}{1-\tau^2\gamma_{n,2}}\right)\left(\epsilon\tau\sum_{n=1}^{N}\|\bm g^{f,n}\|^2 _{L^{2}(\Of^{n})}+E^0\right), 
		\end{align*}
		where $\gamma_{n,2}:= \,G_{2,n}\|\nabla {\bm{w}^{f,n}}\|^2_{L^\infty ({\Of^{n}})}$.	Since $\tau_0'=1/\max\limits_{n=1,\cdots,N}\gamma_{n,2}$, the desired stability result \eqref{sta_2} follows when $\tau^2<\tau_0'$.
		\end{proof}
\end{mytheorem}

\begin{myremark}
Notice that the  admissible time-step sizes for Scheme-A and Scheme-B are respectively 
$$\tau_0=1/\max\limits_{n=1,\cdots,N}\Bigl(\tfrac{3}{2}\|\nabla {\bm{w}^{f,n}}\|_{L^\infty ({\Of^{n}})}+G_{1,n}\|\nabla {\bm{w}_{n-1}^{f,n}}\|_{L^\infty ({\Of^{n-1}})}\Bigr )
$$ and 
$$\tau_0'=1/\max\limits_{n=1,\cdots,N}\Bigl(G_{2,n}\|\nabla {\bm{w}^{f,n}}\|^2_{L^\infty ({\Of^{n}})}\Bigr ),$$
 which depend only on the temporally semi-discrete ALE mapping $\A^t$  defined in \eqref{semi-discrete ALE}.  More precisely, the fluid domain velocities  ${\bm w}^{f,n}$ 
and ${\bm w}^{f,n}_{n-1}$ are, due to \eqref{semi-fluid-velocity},  of the forms 
\begin{align*}
{\bm w}^{f,n}(\x)\circ (\A^t)= \widehat{\bm w}^{f,t}(\widehat \x)=\left.{\partial_t \A^t(\widehat \x,t)}\right|_{t=t_n}   
\quad \text{and}\quad {\bm w}^{f,n}_{n-1}(\x)=\widehat{\bm w}^{f,n}\circ (\A^t)^{-1}(  \x,t_{n-1}) 
\end{align*}
respectively, and the numbers $G_{1,n}$ and $G_{2,n}$, defined in \eqref{def:G},  depend  only on the fluid deformation gradient $\widehat{\bm F}_f^t=\widehat \nabla\A^t$.  Roughly speaking, the smaller $\max\limits_{n=1,\cdots,N}\|\nabla \bm{w}^{f,n}\|_{L^\infty ({\Of^n})} $ is, the   larger the  admissible time steps   will  be.  In particular, if $\nabla \bm{w}^{f,n}=0$,  corresponding to a fixed fluid domain or a rigid translation with a constant velocity, then the time-step sizes can be chosen without restriction.  
	
In addition, we  note that  the stability condition $\tau^2< \tau_0'$ (not $\tau< \tau_0'$) for Scheme-B  is much milder than the condition $\tau< \tau_0$ for  Scheme-A. 
\end{myremark}

\section{Full discretization}\label{Sec_Full}
In this section we will construct two fully discrete schemes based on the two temporally semi-discrete schemes in Section \ref{semi-discrete schemes}, and show the corresponding energy stability results. 

\subsection{Domain triangulation and fully discrete ALE mapping}\label{sec_ALE_fully}
We consider a family of shape regular, 
conforming simplicial triangulations $\{\widehat{\mathcal T}_h\}_{h>0}$ of the reference configuration $\widehat{\Omega}$, consisting of closed triangles for $d=2$ and of closed tetrahedra for $d=3$, which are interface-fitted in the sense that the interface is resolved by element edges (faces in 3D). More precisely, $\widehat{\mathcal T}_h$ is assumed to be a matching mesh with respect to the  $\widehat{\Omega}=\widehat{\Omega}_f\cup\widehat{\Omega}_s\cup\widehat{\Gamma}$ in the sense that for every simplicial element  $\widehat K\in\widehat{\mathcal T}_h$ it holds 
$$
\widehat K\subset \widehat{\Omega}_s \quad \text{or} \quad \widehat K\subset \widehat{\Omega}_f .
$$
Accordingly, we define the sub-triangulations
\begin{align}\label{sub-triangulation}
&\widehat{\mathcal T}_{f,h}:=\{\widehat K\in\widehat{\mathcal T}_h:\ \widehat K\subset\widehat{\Omega}_f\},
\qquad
\widehat{\mathcal T}_{s,h}:=\{\widehat K\in\widehat{\mathcal T}_h:\ \widehat K\subset\widehat{\Omega}_s\}.
\end{align}
Notice that $\widehat{\mathcal T}_h=\widehat{\mathcal T}_{f,h}\cup\widehat{\mathcal T}_{s,h}.$

Let $\widehat h_{\widehat K}$ be the diameter of $\widehat K$ and set
$\widehat h:=\max_{\widehat K\in\widehat{\mathcal T}_h}\widehat h_{\widehat K}$.
Denote by $\widehat{\mathcal E}_h$ the set of all edges/faces  of the mesh skeleton,
and by $\widehat{\mathcal E}_{f,h}$ and $\widehat{\mathcal E}_{s,h}$ the corresponding subsets associated
with $\widehat{\mathcal T}_{f,h}$ and $\widehat{\mathcal T}_{s,h}$, respectively. Since the triangulation is matching, the discrete interface $\widehat{\Gamma}_h$ can be identified with a subset of
mesh edges/faces, namely
\[
\widehat{\Gamma}_h:=\bigcup_{\widehat e\in\widehat{\mathcal E}_{\Gamma,h}}
{\widehat e}
\]
with 
$$\widehat{\mathcal E}_{\Gamma,h}:=\bigl\{\widehat e\in\widehat{\mathcal E}_h:\ 
\widehat e\subset\partial\widehat K_f\cap\partial\widehat K_s
\text{ for some }\widehat K_f\in\widehat{\mathcal T}_{f,h} \text{ and some } \widehat K_s\in\widehat{\mathcal T}_{s,h}\bigr\}.$$
The discrete fluid and structure subdomains induced by the triangulation are defined as
\[
\widehat{\Omega}_{f,h}
:=
\bigcup_{\widehat K\in\widehat{\mathcal T}_{f,h}}
\widehat K,
\qquad
\widehat{\Omega}_{s,h}
:=
\bigcup_{\widehat K\in\widehat{\mathcal T}_{s,h}}
\widehat K.
\]
For $k\ge 0$, let $P_k(\widehat K)$ and $P_k(\widehat e)$ denote the spaces of polynomials of
total degree at most $k$ on   element $\widehat K$ and on  edge/face $\widehat e$, respectively.

To give the spatially discrete ALE mapping, we use the conforming $P_k$-element to approximate the displacement $\widehat \et^f$ of the fluid mesh and introduce the following continuous finite dimensional spaces: 
\begin{align}\label{disALE-displacement}
\left\{\begin{aligned}
	\widehat \V_{h,k}^{\mathcal M}:=&\{\widehat{\bm d}_h^{m}\in \bm C^0(\overline{ {\widehat \Omega}_{f, h}}):  \widehat{\bm d}_h^{m}\mid_{\widehat K}\in [{P_k(\widehat K)}]^{d}, \forall {\widehat K}\in \widehat{\mathcal T}_{f,h} \text{\ and\ }   \widehat{{\bm d}}^{m}_h\mid_{\widehat\Gamma_{f,h}}=\bm 0 \},	\\
	{\widehat \V}_{0,h,k}^{\mathcal M}:=&\{\widehat{{\bm d}}_h^{m}\in \bm C^0(\overline{ {\widehat \Omega}_{f, h}}): \widehat{{\bm d}}_h^{m}\mid_{\widehat K}\in [{P_k(\widehat K)}]^{d}, \forall \widehat K\in \widehat{\mathcal T}_{f,h} \text{\ and\ }  \widehat{{\bm d}}_h^{m}\mid_{\partial\widehat\Omega_{f,h}}=\bm 0\},
\end{aligned}
\right.
\end{align}
where $\widehat\Gamma_{f,h}:=\partial\widehat\Omega_{f,h}\backslash \widehat{\Gamma}_h$. 

Similar to the semi-discrete version  \eqref{semi-discrete ALE-n}, the fully discrete ALE mapping $\A_h^n$ for the fluid domain  at time $t_n$ is constructed from the discrete fluid mesh displacement $\widehat{\bm \eta}^{f,n}_h\in \widehat \V_{h,k}^{\mathcal M}$ with
\begin{align}\label{full-discrete-ALE-n}
\A_h^n(\widehat{\x})=\widehat{\x}+\widehat{\bm \eta}^{f,n}_h(\widehat{\x}):\widehat{\Omega}_{f,h}\rightarrow \Omega_{f,h}^n:=\A_h^n(\widehat{\Omega}_{f,h}), \quad n=1,\dots,N.
\end{align}
Then  
 we introduce  for any $t\in[0,T]$  the fully discrete ALE mapping  
\[
\A_h^t(\widehat{\x},t):
\widehat{\Omega}_{f,h}\rightarrow 
\Omega_{f,h}^t:=\A_h^t(\widehat{\Omega}_{f,h})
\]
with
\begin{align}\label{full-discrete-ALE}
\x=\A_h^t(\widehat{\x},t)
:=\frac{t-t_{n-1}}{\tau}\A_h^{n}(\widehat{\x})
+\frac{t_{n}-t}{\tau}\A_h^{n-1}(\widehat{\x}),
\quad t\in[t_{n-1}, t_{n}],
\end{align}
and define the corresponding fluid deformation
gradient  and its Jacobian determinant as follows:
\begin{align*}
		\widehat{\bm F}_{f,h}^t:=\widehat \nabla\A_h^t, \quad \widehat{\J}_{f,h}^t:=\det(\widehat{\bm {F}}_{f,h}^t).
\end{align*}
For simplicity, we write
 \begin{align*}
		\widehat{\bm F}_{f,h}^n:=\widehat{\bm F}_{f,h}^t|_{t=t_n} =\widehat \nabla\A_h^n, \quad \widehat{\J}_{f,h}^n:=\widehat{\J}_{f,h}^t|_{t=t_n}=\det(\widehat{\bm {F}}_{f,h}^n).
\end{align*}
Similar to \eqref{semi-fluid-velocity},   the fully discrete fluid domain velocities 
$\bm w_h^{f,n}$ on $\Omega_{f,h}^n$ and 
$\bm w_{t,h}^{f,n}$ on $\Omega_{f,h}^t$ are given by
\begin{align}\label{full-fluid-velocity}
\left\{
\begin{array}{l}
\bm w_h^{f,n}(\x)=\bm w_h^{f,n}\circ\A_h^n(\widehat{\x}) =\widehat{\bm w}_h^{f,n}(\widehat{\x}):=\left.\partial_t \A_h^t(\widehat{\x},t)\right|_{t=t_n} =\dfrac1\tau\bigl(\A_h^{n}(\widehat{\x})-\A_h^{n-1}(\widehat{\x})\bigr), \\
\bm w_{t,h}^{f,n}(\x,t):=\widehat{\bm w}_h^{f,n}\circ (\A_h^t)^{-1}(\x,t),\quad t\in[t_{n-1},t_{n}].
\end{array}
\right.
\end{align}
For brevity we set
\begin{align}\label{w-nj-full}
{\bm w}_{j,h}^{f,n}:= {\bm w}^{f,n}_{t,h}\big|_{t=t_j},\quad j=n,n-1,
\end{align}
and we have ${\bm w}_{n,h}^{f,n}={\bm w}_h^{f,n}.$

With the sub-triangulation  $\widehat{\mathcal{T}}_{f,h}$ in \eqref{sub-triangulation} and
the fully discrete ALE mapping $\A_h^n$ in \eqref{full-discrete-ALE-n},  we define the  deformed triangulation of the fluid domain at $t=t_n$ as
\[
\mathcal{T}_{f,h}^n:=\{K=\A_h^n(\widehat K): \;  \widehat K\in \widehat{\mathcal{T}}_{f,h}\}, \quad n=1,\dots,N
\]
and set ${\Omega}_{f, h}^n:=\bigcup_{K\in{\mathcal T}^n_{f,h}}K$. Here we 
  assume that  $\mathcal{T}_{f,h}^n$ remains regular so that no elements possess a degenerate or negative Jacobian determinant. 
Note that 
 ${\Omega}_{f, h}^n$ 
 can be viewed as an approximation of ${\Omega}_{f}^n$ given by \eqref{semi-discrete ALE-n}. 
 We also denote by 
 $$\mathcal E_{f,h}^n:=\{e_f^n = \A_e^n(\widehat{e}_f) : \ \widehat{e}_f 
 \in \widehat{\mathcal E}_{f,h}\}$$ 
 the mesh skeleton of $\Omega^n_{f,h}$, where    the mapping $\A_e^n : \widehat{\mathcal E}_{f,h} \to \mathcal E_{f,h}^n$ is defined by 
 \begin{align}\label{edge-ALE}  
& \A_e^n(\widehat{e}_f):=\A_h^n(\widehat K)|_{\widehat{e}_f}\ \text{ for }\widehat{e}_f\subset \partial\widehat  K , \widehat K\in \widehat{\mathcal{T}}_{f,h}.
 \end{align}

Let $\widehat {\bm{n}}_l $ denote the unit normal vector  along the reference boundary facet $\widehat e_l$, and let $ {\bm{n}}^n_{K,l} $ denote the unit normal vector along the physical boundary facet $e^n_{K,l}$ of the deformed element $K$ at time $t_n$ for $n=1,\dots,N$. We easily have
\begin{align*}
	 {\bm{n}}^n_{K,l} \circ (\A_{h}^n)^{-1} = \frac{(\widehat {\bm F}^n_{h,K})^{-\top} \widehat {\bm{n}}_l}{|(\widehat {\bm F}^n_{h,K})^{-\top} \widehat {\bm{n}}_l|}, 
\end{align*}
where $|\cdot|$ represents the  vector magnitude,  $\widehat{\bm F}^n_{h,K}:=\widehat\nabla\A_h^n|_{\widehat K}$ is the deformation gradient associated with deformed element $K$, and $\widehat J^n_{h,K} = \det(\widehat{\bm F}^n_{h,K})$ is the Jacobian determinant.

For any time $t=t_n$ and any integer $k\geq 1$, we introduce the following discontinuous finite element spaces on the fluid mesh ${\Omega}_{f, h}^n$:
\begin{align}\label{full-fluid-spaces}
\left\{\begin{aligned}
	 	 \V^{f,n}_{h,k}:=&\{\v^{f,n}_h\in \bm L^2( {\Omega_{f,h}^n } ): \v^{f,n}_h\mid_{K}=\tfrac{1}{\widehat J^n_{h,K}}\widehat{\bm F}^n_{h,K}\widehat\v_h^f\circ (\A_{h}^n)^{-1},  \widehat\v_h^f \in {P_{k}(\widehat K)}, K=\A_h^n(\widehat K),\forall  \widehat K\in \widehat {\mathcal T}_{f,h}\},\\
	 \underline{\V}^{f,n}_{h,k} :=&\{\underline{\v}^{f,n}_h\in \bm L^2({\mathcal E}_{f,h}^n): \underline { \v}^{f,n}_h\mid_{e}= \underline {\widehat \v}_h\circ (\A^n_{ e})^{-1},   \underline {\widehat \v}_h \in {P_k(\widehat e)}, e = \A_e^n(\widehat{e}), \forall \widehat e\in \widehat {\mathcal E}_{f,h}, \widehat \v_h\mid_{\widehat\Gamma_{f,h}}=\bm 0 \},\\Q^{f,n}_{h,k-1}:=&\{ q^{f,n}_h\in L^2( {\Omega_{f,h}^n } ): q^{f,n}_h\mid_{K}=\widehat q_h\circ (\A_{h}^n)^{-1}, \widehat q_h \in {P_{k-1}(\widehat K)}, K=\A_h^n(\widehat K),\forall \widehat K\in \widehat {\mathcal T}_{f,h} \}, \\
	\underline{ Q}^{f,n}_{h,k}:=&\{ {\underline q}^{f,n}_h \in  L^2( {\mathcal E }_{f,h}^n): \underline { q}_h^{f,n}\mid_{e}=\underline {\widehat q}_h\circ (\A^n_{ e})^{-1},   \underline {\widehat q}_h \in {P_k(\widehat e)}, e = \A_e^n(\widehat{e}), \forall \widehat e\in \widehat {\mathcal E}_{f,h}\},\\
	 \Bs^{f,n}_{h,k-1}:=&\{ \Bl_h^{f,n}\in [Q^{f,n}_{h,k-1} ]^{d\times d }:  \Bl^{f,n}_h\text{\ is\ symmetric} \}.
\end{aligned}
\right.
\end{align}
We also introduce the following discontinuous finite element spaces on the structure mesh ${\widehat \Omega}_{s, h}$:
\begin{align}\label{full-solid-spaces}
\left\{\begin{aligned}
	\widehat \Bs^{s}_{h,k-1}:=&\{\widehat \Bl_h^s\in [L^2(\widehat {\Omega }_{s,h} )]^{d\times d }: \widehat \Bl^s_h \text{\ is\ symmetric}, \  \widehat \Bl^s_h\mid_{\widehat K}\in [{P_{k-1}(\widehat K)}]^{d\times d}, \forall \widehat K\in \widehat {\mathcal T}_{s,h}  \}, \\
	\widehat \V^{s}_{h,k}:=&\{\widehat\v_h^s\in \bm L^2(\widehat {\Omega }_{s,h} ): \widehat \v^s_h\mid_{\widehat K}\in [{P_k(\widehat K)}]^{d}, \forall \widehat K\in\widehat {\mathcal T}_{s,h}  \},\\
	\underline{\widehat \V}^{s}_{h,k}:=&\{\underline{\widehat\v}^s_h\in \bm L^2(\widehat {\mathcal E}_{s,h}): \widehat \v^s_h\mid_{\widehat e}\in [{P_k(\widehat e)}]^{d}, \forall \widehat e\in\widehat {\mathcal E}_{s,h} \text{\ and\ }  \widehat \v_h^s\mid_{\widehat\Gamma_{s,h}}=\bm 0 \}.
\end{aligned}
\right.
\end{align}

\begin{myremark}
The vector  spaces, $ \V^{f,n}_{h,k}$, $\underline{\V}^{f,n}_{h,k} $,  are used to approximate the fluid velocity $\u^f$ and its trace $ {\underline \u^f}$on element facets, respectively.  Specially,   $\V^{f,n}_{h,k}$ is obtained from the reference configuration through a Piola-type ALE transformation. This choice ensures strong mass-conservation for the proposed HDG schemes, as shown in Theorem \ref{Div-free}.
 The scalar   spaces, $Q^{f,n}_{h,k-1}$ and $\underline{ Q}^{f,n}_{h,k}$, and the  symmetric tensor  space $ \Bs^{f,n}_{h,k-1}$ are employed  to approximate the pressure $p^f$, the pressure trace $\underline p^f$ on element facets, and the stress tensor $\Bl^f=2\rho^f\mu^f\bm{D}(\u^f)$, respectively. 

The vector  spaces $\widehat \V^{s}_{h,k}$ and $\underline{\widehat \V}^{s}_{h,k}$ are utilized to approximate the structure displacement  /   velocity, $ \ds$ / $\widehat\u^s$,  and their traces $ {\underline \ds}$ / $\underline{ \widehat\u^s}$ on element facets, respectively. 
The  symmetric tensor space  $\widehat \Bs^{s}_{h,m}$ is applied  to approximate the stress tensor $ \widehat \Bl^s=2\mu^s\bm{D}(\widehat \u^s)$. 
\end{myremark}

For any triangulation 
$\mathcal T_h \in 
\{\widehat{\mathcal T}_{h}, 
  \widehat{\mathcal T}_{f,h}, 
  \widehat{\mathcal T}_{s,h},
  \mathcal T_{f,h}^n\}$,
we define the following mesh-dependent inner products and norms:
\[
(\u,\v)_{\mathcal T_h} :=\sum_{K\in \mathcal T_h} (\u,\v)_K,
\qquad
\|\u\|^2_{L^2(\mathcal T_h)} := \sum_{K\in \mathcal T_h} \|\u\|^2_{L^2(K)},
\]
\[
\langle \underline\u,\underline\v\rangle_{\partial\mathcal T_h} := \sum_{K\in \mathcal T_h} \langle \underline\u,\underline\v\rangle_{\partial K},
\qquad
\|\underline\u\|^2_{L^2(\partial\mathcal T_h)} := \sum_{K\in \mathcal T_h} \|\underline\u\|^2_{L^2(\partial K)}.
\]
Here $\u, \v, \underline\u, \underline\v$ may be scalar / vector / tensor valued functions, and  $(\cdot,\cdot)_K$ and $\langle\cdot,\cdot\rangle_{\partial K}$ denote the standard $L^2$ inner products on $K$  and $\partial K$, respectively.

Before presenting   the HDG fully  discrete schemes, we need to  introduce, for any
\begin{align*}
	&    \Bl^{f,n}_h, \ \Bm^{f,n}_h \in \Bs^{f,n}_{h,k-1},  \\
	&  \widehat\Bl^{s,n}_h, \widehat\Bm^{s}_h\in \widehat \Bs^{s}_{h,k-1},\\
	&\mathcal U^{f,n}:=(\u_h^{f,n},\underline{\u}_h^{f,n}),\ \mathcal V^{f,n}:=(\v_h^{f,n},\underline{\v}_h^{f,n}) \ \in  \V^{f,n}_{h,k} \times \underline{ \V}^{f,n}_{h,k}, \\	
	& \widetilde P^{f,n}:=(p_h^{f,n},\underline{p}_h^{f,n}) 	 \in  Q^{f,n}_{h,k-1}\times \underline{ Q}^{f,n}_{h,k},
	\\
&	 \widehat{\mathcal U}^{s,n}:=(\widehat \u_h^{s,n}, \underline {\widehat \u}_h^{s,n} ),   \  \widehat{\mathcal V}^{s}:=(\widehat \v_h^{s}, \underline {\widehat \v}_h^{s} ),  \ \widehat{\mathcal D}^{s,n}:=(\widehat{\bm d}_h^{s,n}, \underline{\widehat{\bm d}}_h^{s,n} ) \in \widehat \V^{s}_{h,k}\times \underline{\widehat \V}^{s}_{h,k},
\end{align*}
the following discrete bilinear and trilinear forms: the following discrete bilinear and trilinear forms: 
\begin{align*}
	&A_h^{f,n}(\Bl^{f,n}_h, \Bm^{f,n}_h ):=(\Bl^{f,n}_h, \Bm^{f,n}_h )_{{\mathcal T}^{n}_{f,h}},   \\
	&  A_h^{s}(\widehat\Bl^{s,n}_h, \widehat\Bm^{s}_h):=(\widehat\Bl^{s,n}_h, \widehat\Bm^{s}_h)_{\widehat {\mathcal T}_{s,h}}, \\
	& C_{h,1}^{f,n}(\mathcal U^{f,n}, \Bm^{f}_h):=(\u_h^{f,n},\nabla\cdot \Bm^{f}_h)_{{\mathcal T}^{n}_{f,h}}-\langle \underline{\u}_h^{f,n}, \Bm^{f}_h{\bm n}^f \rangle_{\partial {\mathcal T}^{n}_{f,h}},\\
	& C_{h,2}^{f,n}(\mathcal U^{f,n}, \widetilde P^{f,n}):=(\nabla\cdot \u_h^{f,n},p^{f,n}_h)_{{\mathcal T}^{n}_{f,h}}-\langle ({\u}_h^{f,n}-\underline{\u}_h^{f,n}) \cdot {\bm{n}^f} , \underline p_h^{f,n} \rangle_{\partial {\mathcal T}^{n}_{f,h}},\\
	& S_{h}^{f,n}(\mathcal U^{f,n},  \mathcal V^{f,n}):=\rho^f \mu^f \langle s_1({\u}_h^{f,n}-\underline{{\u}}_h^{f,n}), ({\v }_h^{f,n}-\underline{{\v }}_h^{f,n}) \rangle_{\partial {\mathcal T}^{n}_{f,h}},\\
	& C_{h}^{s}(\widehat{ \mathcal U}^{s,n}, \widehat \Bm^{s}_h):=(\widehat \u_h^{s,n},\widehat \nabla\cdot \widehat{ \Bm}^{s}_h)_{\widehat {\mathcal T}_{s,h}}+\langle  \underline{\widehat \u}_h^{s,n}, \widehat \Bm^{s}_h{\bm n}^s \rangle_{\partial \widehat {\mathcal T}_{s,h} },\\
	& S_{h}^{s}( \widehat{\mathcal D}^{s,n},  \widehat{ \mathcal V}^{s}):= \mu^s\langle s_2(\widehat {\bm d }_h^{s,n}-\underline{\widehat{\bm d}}_h^{s,n} ), (\widehat{\v}_h^{s}-\underline{\widehat{\v }}_h^{s}) \rangle_{\partial \widehat {\mathcal T}_{s,h}},\\
	&T^{f,n}_h({\bm w}^{f,n}_h , \mathcal U^{f,n},\mathcal V^{f,n}):=\tfrac{1}{2}\left (\v_h^{f,n}\cdot(\nabla \u_h^{f,n}) -\u_h^{f,n}\cdot(\nabla \v_h^{f,n}),\u_h^{f,n} - {\bm w}^{f,n}_h \right)_{{\mathcal T}^{n}_{f,h}}\\
	&\hskip5cm \ \  -\tfrac{1}{2}\langle (\underline \u_h^{f,n} -  {\bm w}^{f,n}_h)\cdot {\bm n}^f, \underline \v_h^{f,n}\cdot \u_h^{f,n}-\underline \u_h^{f,n}\cdot \v_h^{f,n}  \rangle_{\partial {\mathcal T}^{n}_{f,h}},
\end{align*}
where the stabilization parameters $s_1$ and $s_2$ in the stabilization terms $S_{h}^{f,n}(\cdot,\cdot) $ and $S_{h}^{s}(\cdot,\cdot)$ are taken as 
$$s_1|_{\widehat e }=1/h_{\widehat e}, \quad \forall \widehat e  \in \widehat {\mathcal E}_{f,h} \quad \text{and}\quad s_2|_{\widehat e  }=1/h_{\widehat e},  \quad \forall \widehat e  \in \widehat {\mathcal E}_{s,h},$$ respectively, with $h_{\widehat e}$ denoting the diameter of face $\widehat e$.

Let $\Pi_{\V^f}$ and $\Pi_{\V^s}$ denote respectively the projection operators from $\bm H^1(\Omega^{f,n})$ onto $\V^{f,n}_{h,k}$ and from $\bm H^1(\widehat \Omega^{s})$ onto $ \widehat \V^{s}_{h,k}$, and give the initial data 
$(\u_h^{f,0},\widehat \u_h^{s,0},\widehat{\bm d}_h^{s,0})=(\Pi_{\V^f}\u_0^f, \Pi_{\V^s}\widehat \u^{s}_0,\Pi_{\V^s}\widehat{\bm d}^{s}_0).$ Then, based on the temporally semi-discrete schemes, Scheme- A and Scheme-B, we propose the following  two fully discrete HDG schemes:
\begin{itemize}
	\item \textbf{Scheme-C}:

Find $(\Bl^{f,n}_h, \mathcal U^{f,n},\widetilde P^{f,n}, \widehat{\bm \eta}_h^{f,n})  \in \Bs^{f,n}_{h,k-1}\times (\V^{f,n}_{h,k} \times \underline{ \V}^{f,n}_{h,k})\times  (Q^{f,n}_{h, k-1}\times \underline{ Q}^{f,n}_{h,k})\times \widehat \V_{h,k}^{\mathcal M}$ 
and $(\widehat\Bl^{s,n}_h,\widehat{\mathcal U}^{s,n},\widehat{\mathcal D}^{s,n} )\in \widehat \Bs^{s}_{h,k-1}\times (\widehat \V^{s}_{h,k}\times \underline{\widehat \V}^{s}_{h,k})\times  (\widehat \V^{s}_{h,k}\times \underline{\widehat \V}^{s}_{h,k})$ for $n=1,2,\cdots,N$, 
with $\widehat{\bm \eta}_h^{f,n}=\underline{ \widehat{\bm d}}_h^{s,n}$ on $\widehat\Gamma_h$, $\underline{\u}_h^{f,n}=\underline{\widehat \u}_h^{s,n}\circ (\A^{n}_e)^{-1}$ on $\Gamma_h^{n}$, and $\bm w_h^{f,n}=\tilde{D}_t\widehat {\bm \eta}_h^{f,n}  \circ (\A_h^{n})^{-1}$, such that
\begin{align}\label{full_discrete_C}
\left\{
\begin{aligned}
A_h^{f,n}(\Bl^{f,n}_h, \Bm^{f,n}_h )+2 \rho^f \mu^f C_{h,1}^{f,n}(\mathcal U^{f,n}, \Bm^{f,n}_h)&=0,\\
A_h^{s}(\widehat\Bl^{s,n}_h, \widehat\Bm^{s}_h)+2\mu^s C_{h}^{s}(\mathcal D^{s,n}, \widehat \Bm^{s}_h)&=0,\\
\rho^f(\tilde{D}_t^M\u_h^{f,n}, \v_h^{f,n})_{{\mathcal T}_{f,h}^{n}}+\rho^f T^{f,n}_h(\bm{w}_h^{f,n}, \mathcal U^{f,n},\mathcal V^{f,n})
\\+\rho^f(\nabla{\bm{w}_h^{f,n}} \u_h^{f,n}-\tfrac{1}{2} \nabla\cdot{\bm{w}_h^{f,n}} \u_h^{f,n}, \v_h^{f,n})_{{\mathcal T}_{f,h}^{n}}-C_{h,1}^{f,n}(\mathcal V^{f,n}, \Bl^{f,n}_h)
\\-C_{h,2}^{f,n}(\mathcal V^{f,n}, \widetilde P^{f,n})+S_{h}^{f,n}(\mathcal U^{f,n},  \mathcal V^{f,n})+\rho^s(\tilde{D}_t\widehat \u_h^{s,n},\widehat \v_h^s )_{\widehat {\mathcal T}_{s,h}}-C_{h}^{s}(\widehat{ \mathcal V}^{s}, \widehat \Bl^{s,n}_h)
\\ +S_{h}^{s}(\mathcal D^{s,n},  \mathcal V^{s})+\lambda^s(\widehat\nabla\cdot \widehat{\bm{d}}_h^{s,n}, \widehat\nabla\cdot\widehat\v_h^s)_{\widehat {\mathcal T}_{s,h}}
-(\bm g^{f,n}, \v_h^{f,n})_{{\mathcal T}_{f,h}^{n}}&=0,\\
    C_{h,2}^{f,n}(\mathcal U^{f,n},  \mathcal Q ^{f,n})&=0,\\
	( \tilde{D}_t \widehat{\bm{d}}_h^{s,n}, \ms_h \rangle _{\widehat {\mathcal T}_{s,h}}-(\widehat \u^{s,n},\ms_h ) _{\widehat {\mathcal T}_{s,h}}&=0,\\
	\langle \tilde{D}_t  \underline{\widehat{\bm{d}}}_h^{s,n},\underline{\widehat{\bm m}}_h^{s}\rangle _{\partial \widehat {\mathcal T}_{s,h}}-\langle \underline{ \widehat \u}^{s,n},\underline{\widehat{\bm m}}_h^{s} \rangle _{\partial \widehat {\mathcal T}_{s,h}}&=0,\\
	(\kappa^f\widehat\nabla \widehat{\bm{\eta}}_h^{f,n},\widehat\nabla \mf_h)_{\widehat {\mathcal T}_{f,h}} &=0, 
\end{aligned}
\right.
\end{align}
\noindent for all $(\Bm^{f,n}_h, \mathcal V^{f,n},\mathcal Q^{f,n}, \widehat{\bm m}_h^{f})\in  \Bs^{f,n}_{h,k-1}\times (\V^{f,n}_{h,k} \times \underline{ \V}^{f,n}_{h,k})\times  (Q^{f,n}_{h, k-1}\times \underline{ Q}^{f,n}_{h,k})\times \widehat \V_{0,h,k}^{\mathcal M}$ and $(\widehat\Bm^{s}_h,\widehat{\mathcal V}^{s},\widehat{\mathcal M}^{s} )\in \widehat \Bs^{s}_{h,k-1}\times (\widehat \V^{s}_{h,k}\times \underline{\widehat \V}^{s}_{h,k})\times  (\widehat \V^{s}_{h,k}\times \underline{\widehat \V}^{s}_{h,k})$ with  $\underline{\v}_h^{f,n}=\underline{\widehat \v}_h^{s}\circ (\A^{n}_e)^{-1}$ on $\Gamma_h^{n}$. 
Here and in Scheme-D, we set $$\mathcal Q^{f,n}:=(q_h^{f,n},\underline{q}_h^{f,n}). \quad \widehat{\mathcal M}^{s}:=(\widehat{\bm m}_h^{s}, \underline{\widehat{\bm m}}_h^{s} ).$$
\item \textbf{Scheme-D}:

Find $(\Bl^{f,n}_h, \mathcal U^{f,n},\widetilde P^{f,n}, \widehat{\bm \eta}_h^{f,n})  \in \Bs^{f,n}_{h,k-1}\times (\V^{f,n}_{h,k} \times \underline{ \V}^{f,n}_{h,k})\times  (Q^{f,n}_{h, k-1}\times \underline{ Q}^{f,n}_{h,k})\times \widehat \V_{h,k}^{\mathcal M}$ 
and $(\widehat\Bl^{s,n}_h,\widehat{\mathcal U}^{s,n},\widehat{\mathcal D}^{s,n} )\in \widehat \Bs^{s}_{h,k-1}\times (\widehat \V^{s}_{h,k}\times \underline{\widehat \V}^{s}_{h,k})\times  (\widehat \V^{s}_{h,k}\times \underline{\widehat \V}^{s}_{h,k})$ for $n=1,2,\cdots,N$, with $\widehat{\bm \eta}_h^{f,n}=\underline{ \widehat{\bm d}}_h^{s,n}$ on $\widehat\Gamma_h$, $\underline{\u}_h^{f,n}=\underline{\widehat \u}_h^{s,n}\circ (\A^{n}_e)^{-1}$ on $\Gamma_h^{n}$, and $\bm w_h^{f,n}=\tilde{D}_t\widehat {\bm \eta}_h^{f,n}  \circ (\A_h^{n})^{-1}$, such that
\begin{align}\label{full_discrete_D}
\left\{
\begin{aligned}
A_h^{f,n}(\Bl^{f,n}_h, \Bm^{f,n}_h )+  2\rho^f \mu^f C_{h,1}^{f,n}(\mathcal U^{f,n}, \Bm^{f,n}_h)&=0,\\
A_h^{s}(\widehat\Bl^{s,n}_h, \widehat\Bm^{s}_h)+2\mu^s C_{h}^{s}(\mathcal D^{s,n}, \widehat \Bm^{s}_h)&=0,\\
\tfrac{\rho^f}{\tau} (\u_h^{f,n}, \v_h^f)_{{\mathcal T}_{f,h}^{n}}-\tfrac{\rho^f}{\tau}(\u_h^{f,n-1}, \v_h^{f,n}\circ\X_{n-1,h}^{n} )_{\mathcal T _{f,h}^{n-1}}
\\-\rho^f((\nabla{\bm{w}_h^{f,n}})^{\top} \u_h^{f,n}-\tfrac{1}{2} \nabla\cdot{\bm{w}_h^{f,n}} \u_h^{f,n}, \v_h^{f,n})_{\mathcal T _{f,h}^{n}}
+\rho^f T^{f,n}_h(\bm{w}_h^{f,n}, \mathcal U^{f,n},\mathcal V^{f,n})
\\-C_{h,1}^{f,n}(\mathcal V^{f,n}, \Bl^{f,n}_h)-C_{h,2}^{f,n}(\mathcal V^{f,n}, \widetilde P^{f,n})+ S_{h}^{f,n}(\mathcal U^{f,n},  \mathcal V^{f,n})+\rho^s(\tilde{D}_t\widehat \u_h^{s,n},\widehat \v_h^s )_{\widehat {\mathcal T}_{s,h}}
\\-C_{h}^{s}(\widehat{ \mathcal V}^{s}, \widehat \Bl^{s,n}_h)+S_{h}^{s}(\mathcal D^{s,n},  \mathcal V^{s})+\lambda^s(\widehat\nabla\cdot \widehat{\bm{d}}_h^{s,n}, \widehat\nabla\cdot\widehat\v_h^s)_{\widehat{\mathcal T}_{s,h}}
-(\bm g^{f,n}, \v_h^f)_{\mathcal T_{f,h} ^{n}}&=0,\\
    C_{h,2}^{f,n}(\mathcal U^{f,n},  \mathcal Q ^{f,n})&=0,\\
	( \tilde{D}_t \widehat{\bm{d}}_h^{s,n}, \ms_h \rangle _{\widehat {\mathcal T}_{s,h}}-(\widehat \u^{s,n},\ms_h ) _{\widehat {\mathcal T}_{s,h}}&=0,\\
	\langle \tilde{D}_t  \underline{\widehat{\bm{d}}}_h^{s,n},\underline{\widehat{\bm m}}_h^{s}\rangle _{\partial \widehat {\mathcal T}_{s,h}}-\langle \underline{ \widehat \u}^{s,n},\underline{\widehat{\bm m}}_h^{s} \rangle _{\partial \widehat {\mathcal T}_{s,h}}&=0,\\
	(\kappa^f\widehat\nabla \widehat{\bm{\eta}}_h^{f,n},\widehat\nabla \mf_h)_{\widehat {\mathcal T} _{f,h}} &=0, 
\end{aligned}
\right.
\end{align}
\noindent for all $(\Bm^{f,n}_h, \mathcal V^{f,n},\mathcal Q^{f,n}, \widehat{\bm m}_h^{f})\in  \Bs^{f,n}_{h,k-1}\times (\V^{f,n}_{h,k} \times \underline{ \V}^{f,n}_{h,k})\times  (Q^{f,n}_{h, k-1}\times \underline{ Q}^{f,n}_{h,k})\times \widehat \V_{0,h,k}^{\mathcal M}$ and $(\widehat\Bm^{s}_h,\widehat{\mathcal V}^{s},\widehat{\mathcal M}^{s} )\in \widehat \Bs^{s}_{h,k-1}\times (\widehat \V^{s}_{h,k}\times \underline{\widehat \V}^{s}_{h,k})\times  (\widehat \V^{s}_{h,k}\times \underline{\widehat \V}^{s}_{h,k})$ with  $\underline{\v}_h^{f,n}=\underline{\widehat \v}_h^{s}\circ (\A^{n}_e)^{-1}$ on $\Gamma_h^{n}$. 
\end{itemize}

\begin{myremark}\label{rem42}
In  Scheme-C as well as Scheme-D, the fluid, structural, and ALE mapping equations form  a monolithic,  strongly coupled nonlinear system  of the following form: Find $\mathbf{X}_h^n := (\bm u_*^{f,n}, \widehat{\bm u}_*^{s,n}, \widehat{\bm \eta}_h^{f,n})$, for $ n=1,\cdots,N$,  where the fluid unknown is given by $\bm u_*^{f,n}:=(\Bl^{f,n}_h, \mathcal U^{f,n}, \widetilde P^{f,n})$ and the structural unknown is given by $\widehat{\bm u}_*^{s,n}:=(\widehat\Bl^{s,n}_h,\widehat{\mathcal U}^{s,n},\widehat{\mathcal D}^{s,n})$, such that
\begin{align}\label{nonlinear_system}
\mathcal{R}_h^n(\mathbf{X}_h^n)=0,
\end{align}
where the nonlinear residual $\mathcal{R}_h^n$ involves     the current computational mesh ${\mathcal T}_{f,h}^{n}$ that depends on the fluid mesh displacement unknown $\widehat{\bm \eta}_h^{f,n}$.  
The resulting system \eqref{nonlinear_system} 
can be solved by a Newton-type iteration through the following steps:

\medskip
\noindent
\textbf{Step 1: Initialization.}  
Given $\mathbf{X}_h^{n-1}$, 
set  the initial guess
\[
\mathbf{X}_h^{n,0}=(\bm u_*^{f,n,0}, \widehat{\bm u}_*^{s,n,0}, \widehat{\bm \eta}_h^{f,n,0}):=\mathbf{X}_h^{n-1}
\]
and the iteration number $k=0$.

\medskip
\noindent
\textbf{Step 2: Domain definition.}  
 With the current iteration  $\mathbf{X}_h^{n,k}=(\bm u_*^{f,n,k}, \widehat{\bm u}_*^{s,n,k}, \widehat{\bm \eta}_h^{f,n,k})$, define the fluid computational mesh ${\mathcal T}_{f,h}^{n,k}$ through the fluid mesh displacement $\widehat{\bm \eta}_h^{f,n,k}$. 

\medskip
\noindent
\textbf{Step 3: Residual and Jacobian assembly.}  
Assemble the nonlinear residual $\mathcal{R}_h^n(\mathbf{X}_h^{n,k})$ on ${\mathcal T}_{f,h}^{n,k}$ (and $\widehat{\mathcal T}_{s,h}$ for the structural part) and compute the Jacobian operator $D\mathcal R_h^n(\mathbf X_h^{n,k})$, i.e. the Fr\'{e}chet derivative of the residual  with respect to all unknowns.

\medskip
\noindent
\textbf{Step 4: Linearly solve and update.}  
Determine the Newton increment $\delta \mathbf{X}_h^{n,k}$ from
\begin{align*}
D\mathcal{R}_h^n(\mathbf{X}_h^{n,k}) \bigl[ \delta \mathbf{X}_h^{n,k} \bigr] = - \mathcal{R}_h^n(\mathbf{X}_h^{n,k}),
\end{align*}
and compute
\[
\mathbf{X}_h^{n,k+1} = \mathbf{X}_h^{n,k} + \delta \mathbf{X}_h^{n,k}.
\]

\medskip
\noindent
\textbf{Step 5: Convergence check.}  
Repeat Steps 2--5 with $k:=k+1$ until $\|\mathcal{R}_h^n(\mathbf{X}_h^{n,k+1})\|_{\ell^2} \le \varepsilon$ or a prescribed maximum iteration number of  $k$ is reached.

\medskip
\noindent
\textbf{Step 6: End of iteration.}  Set 
\[
\mathbf{X}_h^{n}:=\mathbf{X}_h^{n,k+1}, \quad {\mathcal T}_{f,h}^n:={\mathcal T}_{f,h}^{n,k+1}.
\]
 \end{myremark}
 
 \begin{myremark}
	In the equation $C_{h,2}^{f,n}(\mathcal U^{f,n}, \mathcal Q^{f,n})=0$   in \eqref{full_discrete_C} and \eqref{full_discrete_D}, setting $q_h^{f,n}=0$ yields
	\[
	\langle ({\u}_h^{f,n}-\underline{\u}_h^{f,n}) \cdot {\bm{n}^f},\, \underline q_h^{f,n} \rangle_{\partial {\mathcal T}^{n}_{f,h}}=0,
	\]
	which shows that $\underline p_h^{f,n}$ serves as a Lagrange multiplier enforcing the continuity of the normal component of ${\u}_h^{f,n}$ across interelement boundaries (also cf. Theorem \ref{Div-free}).
\end{myremark}
\begin{myremark}\label{CD-equal}
Similar to the case of temporal semi-discretization (cf. Remark \ref{AB-equal}), Scheme-C and Scheme-D are   equivalent  if $\nabla\bm w_h^{f,n}=0$ for any $n$, which means that the fluid   domain is fixed all the time. 
\end{myremark}

The following result shows that the fluid velocity approximations obtained through the above  two  HDG fully discrete schemes are globally divergence-free,  ensuring  the mass-conserving property for the incompressible fluid.  
\begin{mytheorem}\label{Div-free}
	   Scheme-C and Scheme-D both produce  globally divergence-free velocity approximations, i.e. it holds for the discrete fluid velocity solutions $\u_h^{f,n}\in  \V^{f,n}_{h,k}$ ($n=1,\dots,N$) that 
	\begin{align*}
		\u_h^{f,n}\in \bm H({\rm{div}},\Omega_{f,h}^{n}),\quad  \nabla\cdot \u_h^{f,n}=0.
	\end{align*}
	\begin{proof}
	Notice that  from    \eqref{full_discrete_C}, as well as  \eqref{full_discrete_D}, it follows that the discrete solution $\mathcal U^{f,n}=(\u_h^{f,n},\underline{\u}_h^{f,n})\in  \V^{f,n}_{h,k} \times \underline{ \V}^{f,n}_{h,k}$ satisfies
	\begin{align}\label{0=C_{h,2}}
	0=C_{h,2}^{f,n}(\mathcal U^{f,n}, \mathcal Q^{f,n})= (\nabla\cdot \u_h^{f,n},q^{f,n}_h)_{{\mathcal T}^{n}_{f,h}}-\langle {\u}_h^{f,n}\cdot {\bm{n}^f} , \underline q_h^{f,n} \rangle_{\partial {\mathcal T}^{n}_{f,h}}
	\end{align}	
	for any $ \mathcal Q^{f,n}=(q_h^{f,n},\underline{q}_h^{f,n}) \in  Q^{f,n}_{h,k-1}\times \underline{ Q}^{f,n}_{h,k}$.
In view of  \eqref{full-fluid-spaces},  for any $K=\A_h^n(\widehat K)$ and $e = \A_e^n(\widehat{e})$ with  $  \widehat K\in \widehat {\mathcal T}_{f,h}$ and $ \widehat e\in \widehat {\mathcal E}_{f,h}, $  we have 
				\begin{align*}
				 & \u_h^{f,n}\circ \A_h^{n}\mid_{K}=\tfrac{1}{\widehat J^n_{h,K}}\widehat{\bm F}^n_{h,K}\widehat\u_h^f\circ (\A_{h}^n)^{-1}, \quad  \widehat\u_h^f \in {P_{k}(\widehat K)}, \\
				&\underline { \u}^{f,n}_h\mid_{e}= \underline {\widehat \u}_h\circ (\A^n_{ e})^{-1}, \quad   \underline {\widehat \u}_h \in {P_k(\widehat e)}, \\
		&q^{f,n}_h\mid_{K}=\widehat q_h\circ (\A_{h}^n)^{-1}, \quad \widehat q_h \in {P_{k-1}(\widehat K)},\\
	&\underline { q}_h^{f,n}\mid_{e}=\underline {\widehat q}_h\circ (\A^n_{ e})^{-1},  \quad  \underline {\widehat q}_h \in {P_k(\widehat e)},  
	\end{align*}	
which imply (cf. \eqref{div_in} and  \eqref{212})  $(\nabla\cdot \u_h^{f,n})\circ\A_{h}^{n}=\frac{1}{\widehat \J_{f,h}^n}\widehat\nabla\cdot \widehat \u_h^{f,n}$ and 
		\begin{align*}
			&(\nabla\cdot \u_h^{f,n},  q^{f,n}_h)_K=(\widehat \nabla\cdot \widehat \u_h^{f,n},  \widehat q_h )_{\widehat K},\quad 
			\langle \u_h^{f,n}\cdot\bm n^f , q^{f,n}_h\rangle_{\partial K}=\langle\widehat \u_h^{f,n}\cdot\widehat{ \bm n}^f ,\underline {\widehat q}_h\rangle_{\partial \widehat K }.
		\end{align*}
		Then \eqref{0=C_{h,2}} leads to
		\begin{align*}
	0=& (\nabla\cdot \u_h^{f,n},q^{f,n}_h)_{{\mathcal T}^{n}_{f,h}}-\langle {\u}_h^{f,n}\cdot {\bm{n}^f} , \underline q_h^{f,n} \rangle_{\partial {\mathcal T}^{n}_{f,h}}\\
	=&(\widehat \nabla\cdot \widehat \u_h^{f,n},\widehat q^{f}_h)_{\widehat {\mathcal T}_{f,h}}-\langle \widehat {\u}_h^{f,n}\cdot \widehat {\bm{n}^f} , \underline {\widehat q}_h^{f} \rangle_{\partial \widehat {\mathcal T}_{f,h}}.
		\end{align*}
		Take   $(\widehat q_h^{f},\underline{\widehat q}_h^{f})=(\widehat \nabla\cdot \widehat  \u_h^{f,n}-\widehat c, \widehat r_b-\widehat c)$ in this equation, with $\widehat c=\frac{1}{|\widehat \Omega_{f,h}|}\int_{\widehat \Omega_{f,h}}\widehat \nabla\cdot \widehat \u_h^{f,n}d\widehat\x$ and 
\begin{align*}
		\widehat r_b|_{\widehat e}=\left\{
\begin{aligned}
&-(\widehat \u_h^{f,n}\cdot{\widehat{\bm n}}^f )|_{\partial \widehat K_1\cap \widehat e}-(\widehat \u_h^{f,n}\cdot \widehat{ \bm n}^f )|_{\partial \widehat K_2\cap \widehat e},&\quad \text{ if } \widehat e=\partial \widehat K_1\cap \partial \widehat K_2,  \widehat K_1,  \widehat K_2\in \widehat {\mathcal T}_{f,h},\\
&0, &\quad\text{ if } \widehat e\subset \partial\widehat \Omega_{f,h}
\end{aligned}
\right.
\end{align*}
for any $\widehat e\in \widehat {\mathcal E}_{f,h}$, then we get	
		\begin{align*}
			0&=(\widehat \nabla\cdot \widehat \u_h^{f,n},\widehat \nabla\cdot \widehat  \u_h^{f,n}-\widehat c)_{\widehat{\mathcal T}_{f,h}}-\langle \widehat {\u}_h^{f,n}\cdot \widehat {\bm{n}^f} , \widehat r_b-\widehat c \rangle_{\partial \widehat {\mathcal T}_{f,h}}\\
			&=\|\widehat \nabla\cdot \widehat \u_h^{f,n}\|^2_{L^2({\widehat {\mathcal T}_{f,h}} )}-\langle \widehat {\u}_h^{f,n}\cdot \widehat {\bm{n}^f} , \widehat r_b\rangle_{\partial \widehat {\mathcal T}_{f,h}}\\
			&=\|\widehat \nabla\cdot \widehat \u_h^{f,n}\|^2_{L^2({\widehat {\mathcal T}_{f,h}} )}+\sum_{\widehat e\in  \widehat {\mathcal E}_{f,h}/\partial\widehat {\mathcal T}_{f,h}}\|(\widehat \u_h^{f,n}\cdot{\widehat{\bm n}}^f )|_{\partial \widehat K_1\cap \widehat e}+(\widehat \u_h^{f,n}\cdot \widehat{ \bm n}^f )|_{\partial\widehat K_2\cap \widehat e}\|^2_{L^2(\widehat e)}.
		\end{align*}
This means that  $\widehat \nabla\cdot \widehat \u_h^{f,n}|_{\widehat K}=0$ and  $  \widehat \u_h^{f,n}\in H({\widehat{ \rm{div}}},\widehat \Omega_h^{f})$, and the desired conclusion follows.
	\end{proof}
\end{mytheorem}

To show the energy stability of the fully discrete schemes, we first introduce operators $\mathcal K_h^{f,n}: \V^{f,n}_{h,k} \times \underline{ \V}^{f,n}_{h,k}\rightarrow  \Bs^{f,n}_{h,k-1}$ and $\mathcal K_h^{s,n}:\widehat \V^{s}_{h,k}\times \underline{\widehat \V}^{s}_{h,k} \rightarrow \widehat \Bs^{s}_{h,k-1}$ defined by
\begin{align}\label{K_h-f-s}
\left\{\begin{aligned}	 &(\mathcal K_h^{f,n}\mathcal V^{f,n}, \Bm^{f,n}_h)=- C_{h,1}^{f,n}(\mathcal V^{f,n}, \Bm^{f,n}_h), \quad \forall \ \mathcal V^{f,n}\in \V^{f,n}_{h,k} \times \underline{ \V}^{f,n}_{h,k},\ \Bm^{f,n}_h \in \Bs^{f,n}_{h,k-1},  \\
	 &(\mathcal K_h^{s,n}\widehat{ \mathcal V}^{s}, \widehat \Bm^{s}_h)=- C_{h}^{s}(\widehat{ \mathcal V}^{s}, \widehat \Bm^{s}_h), \quad\forall\  \widehat{ \mathcal V}^{s}\in \widehat \V^{s}_{h,k}\times \underline{\widehat \V}^{s}_{h,k},\ \widehat \Bm^{s}_h\in  \widehat \Bs^{s}_{h,k-1}.
	 \end{aligned}
	 \right.
\end{align}
It is evident that $\mathcal K_h^{f,n}$ and  $\mathcal K_h^{s,n}$  are well-defined. From definitions and the first two equations in \eqref{full_discrete_C} or \eqref{full_discrete_D}  , we immediately have
\begin{align}\label{L_defi}
	\Bl^{f,n}_h=2\rho^f\mu^f \mathcal K_h^{f,n}\mathcal U^{f,n}, \quad  \widehat\Bl^{s,n}_h=2\mu^s \mathcal K_h^{s,n}\widehat{ \mathcal U}^{s,n}.
\end{align}
In addition, we introduce the following two norms: for any $\mathcal V^{f,n}=(\v_h^{f,n},\underline{\v}_h^{f,n}) \in \V^{f,n}_{h,k} \times \underline{ \V}^{f,n}_{h,k}$ and $\widehat{ \mathcal V}^s=(\widehat \v_h^{s},\underline{\widehat \v}_h^{s})\in \widehat \V^{s}_{h,k}\times \underline{\widehat \V}^{s}_{h,k}$,
\begin{align*}
	&\|\mathcal V^{f,n}\|_{\V^{f,n}_{h,k} \times \underline{ \V}^{f,n}_{h,k}}:=\|\mathcal K_h^{f,n}\mathcal V^{f,n}\|_{L^2({{\mathcal T}_{f,h}^{n}})}+\|s_1^\frac{1}{2}(\v_h^{f,n}-\underline{\v}_h^{f,n})\|_{L^2(\partial{\mathcal T}_{f,h}^{n})},\\
	&\| \widehat{\mathcal V}^{s}\|_{\widehat \V^{s}_{h,k}\times \underline{\widehat \V}^{s}_{h,k}}:=\|\mathcal K_h^{s,n}\widehat { \mathcal V}^{s}\|_{L^2(\partial\widehat{\mathcal T}_{s,h})}+\|s_2^\frac{1}{2}(\widehat \v_h^{s}-\underline{\widehat \v}_h^{s})\|_{L^2(\partial\widehat {\mathcal T}_{s,h})}.
\end{align*}
 The following HDG Sobolev embedding inequality holds (see \cite[Lemma~3.3]{Chen2023}):
\begin{align}\label{eq:emb}
	\|\v^{f,n}\|_{L^2(\mathcal T^n_{f,h})}\leq C\|\mathcal V^{f,n}\|_{\V^{f,n}_{h,k} \times \underline{ \V}^{f,n}_{h,k}},\qquad \forall\ \mathcal V^{f,n}=(\v_h^{f,n},\underline{\v}_h^{f,n}) \in \V^{f,n}_{h,k} \times \underline{ \V}^{f,n}_{h,k}.
\end{align}
Similar to \eqref{def:G}, for $n=1,2,\cdots,N$ and $t\in [t_{n-1},t_n]$ we define  
\begin{align}\label{def:Gh}
\left\{
\begin{aligned}
	G_{1,n,h}:=& \sup_{t\in[t_{n-1},t_n]}
	\Bigl(1+ \tfrac12\|(\mathbb{G}_{n-1,h}(t))^{-1}\|_{L^\infty({\mathcal T}_{f,h}^{n-1})} \|\mathbb{G}_{n-1,h}(t)\|_{L^\infty({\mathcal T}_{f,h}^{n-1}))}\Bigr)  \|\mathbb{G}_{n-1,h}(t)\|_{L^\infty({\mathcal T}_{f,h}^{n-1})},\\
	 G_{2,n,h}  :=&\sup_{t\in [t_{n-1},t_n]}\Bigl(\|\mathbb{J}_{n,h}(t)\|_{L^\infty({\mathcal T}_{f,h}^{n})}+\bigl(2+\tfrac12\|\mathbb{F}_{n,h}(t)\|^2_{L^\infty({\mathcal T}_{f,h}^{n})}\bigr) \|(\mathbb{F}_{n,h}(t))^{-\top}\|_{L^\infty({\mathcal T}_{f,h}^{n})}\|\mathbb{G}_{n,h}(t)\|_{L^\infty({\mathcal T}_{f,h}^{n})} 
	\\ &\ +\tfrac12\|(\mathbb{F}_{n,h}(t))^{-\top}\|^2_{L^\infty({\mathcal T}_{f,h}^{n})}\|\mathbb{F}_{n,h}(t)\|_{L^\infty({\mathcal T}_{f,h}^{n})}\|\mathbb{G}_{n,h}(t)\|_{L^\infty({\mathcal T}_{f,h}^{n})}
	 \Bigr),
	\end{aligned}
	\right.	
\end{align}
where  $$\mathbb{F}_{n,h}(t):=\widehat{\bm F}^t_{f,h}(\widehat{\bm F}^n_{f,h})^{-1}, \quad \mathbb{J}_{n,h}(t):=\frac{\widehat J_{f,h}^n}{\widehat J_{f,h}^t}, \quad \mathbb{G}_{n,h}(t):=\mathbb{J}_{n,h}(t)\mathbb{F}_{n,h}(t).$$
Take $\X_{i,h}^j:= \A_h^j\circ (\A_h^i)^{-1}$ for $i,j = n-1,n$, similar to \eqref{X_de}  we have
\begin{align*}
	\v_h \circ\X_{i,h}^j&=(\frac{\widehat{\bm F}_{f,h}^{i}}{\widehat \J_{f,h}^{i}}\widehat\v_h) \circ (\A^i_h)^{-1} = \left (\frac{\widehat \J_{f,h}^{j}}{\widehat \J_{f,h}^{i}}\widehat{\bm F}_{f,h}^{i}(\widehat{\bm F}^j_{f,h})^{-1}\v_h\circ \A_h^j\right )\circ (\A_h^i)^{-1}, \quad \forall \v_h\in \V^{f,j}_{h,k}.
\end{align*}

We now define the fully discrete energy.
Let $E_h^{n}$ denote the sum of the full discrete kinetic energy of the fluid and the discrete kinetic and discrete elastic energy of the structure at time $t_n$, i.e.
\begin{align}\label{fully discrete energy}
	E^{n}_h:=\tfrac{\rho^f}{2}\|\u_h^{f,n}\|^2_{L^2({{\mathcal T}_{f,h}^n})}+\tfrac{\rho^s}{2} \|\widehat \u_h^{s,n}\|^2_{L^2(\widehat {\mathcal T}_{s,h}) }+\left (\mu^s\|\widehat{\mathcal D}^{s,n} \|^2_{\widehat \V^{s}_{h,k}\times \underline{\widehat \V}^{s}_{h,k} }+\tfrac{\lambda^s }{2}\|\widehat\nabla\cdot {\widehat{\bm{d}}_h^{s,n} }\|^2_{L^2(\widehat {\mathcal T}_{s,h} )}\right ).
\end{align}
And let $\tilde E _{f,h}^{n}$ denote the fluid viscous dissipation and $\tilde E_{s, d, h}^{n}$ the structural numerical dissipation, defined respectively by
\begin{align*}
	\tilde E^{n}_{f,h}=&\tau\rho^f\mu^f\|\mathcal U^{f,n}\|^2_{ \V^{f,n}_{h,k} \times \underline{ \V}^{f,n}_{h,k}},\\
	\tilde E_{s,d,h}^{n}=&\tfrac{1}{2}(\rho^s\|\widehat \u_h^{s,n}-\widehat \u_h^{s,n-1}\|^2_{L^2({\widehat {\mathcal T}_{s,h}})}+ 2\mu^s\|\widehat{\mathcal D}^{s,n}-\widehat{\mathcal D}^{s,n-1}\|^2_{\widehat \V^{s}_{h,k}\times \underline{\widehat \V}^{s}_{h,k}}+\lambda^s\|\widehat{\nabla}\cdot\widehat{\bm{d}}_h^{s,n}-\widehat{\nabla}\cdot\widehat{\bm{d}}_h^{s,n-1}\|^2_{L^2({\widehat {\mathcal T}_{s,h}})}).
\end{align*}

By following  the same   routines as in the proofs of Theorems \ref{theorem_A} and \ref{theorem_B}, we can obtain  energy stability results for  the  proposed fully discrete schemes. 
\begin{mytheorem}[Stability of {Scheme-C}]\label{th_fullC}
Let $(\u_h^{f,n},\widehat \u_h^{s,n},p^{f,n},\widehat{\bm d}_h^{s,n}, \widehat{\bm \eta}_h^{f,n})$ solve the fully discrete system \eqref{full_discrete_C} for $n=1,2,\cdots,N$.  If the time step $\tau$ satisfies 
\[
	\tau<\widetilde \tau_0:=1/\max\limits_{n=1,\cdots,N}\Bigl(\tfrac{3}{2}\|\nabla {\bm{w}_h^{f,n}}\|_{L^\infty ({\mathcal T_{f,h} ^{n}})}+G_{1,n,h}\|\nabla {\bm{w}_{h,n-1}^{f,n}}\|_{L^\infty ({\mathcal T_{f,h}^{n-1}})} \Bigr ),
\]
	 then the following  energy estimate holds:  
\begin{align}\label{sta_4}
		&E_h^N+\sum_{n=1}^{N}\left(\tilde E^{n}_{f,h}+\tilde E_{s,d,h}^{n}+\frac{\rho^f}{2}\|\u_h^{f,n-1}\circ\X^{n-1}_{n,h}-\u_h^{f,n}\|^2_{L^2({\mathcal T}_{f,h}^{n})}\right)\lesssim \tau\sum_{n=1}^{N}\|\bm g^{f,n}\|^2 _{L^{2}({\mathcal T}_{f,h}^{n}) }+E_h^0.
	\end{align}
	\begin{proof}
		The fifth and sixth equations of   \eqref{full_discrete_C} imply 
		\begin{align*}
			\widehat \u_h^{s,n}|_{\widehat {\mathcal T}_{s,h} }=\tilde{D}_t\widehat {\bm d}_h^{s,n}\quad \text{and}\quad \underline{ \widehat \u}_h^{s,n}|_{\widehat {\mathcal E}_{s,h}}=\tilde{D}_t \underline {\widehat {\bm d}}_h^{s,n},
		\end{align*}
	respectively.
		Applying the Young's inequality and the HDG Sobolev embedding inequality \eqref{eq:emb}, we obtain
		\begin{align*}
		(\bm g^{f,n}, \v_h^{f,n})_{{\mathcal T}_{f,h}^{n}} &\le \epsilon \|\bm g^{f,n}\|_{L^2({\mathcal T}_{f,h}^{n})}^2+ \frac{1}{4\epsilon}\|\v_h^{f,n}\|_{L^2({\mathcal T}_{f,h}^{n})}^2
		\\&\le \epsilon \|\bm g^{f,n}\|_{L^2({\mathcal T}_{f,h}^{n})}^2+ \frac{C}{4\epsilon}\|\widehat{\mathcal V}^{s,n}\|_{\widehat \V^{s}_{h,k}\times \underline{\widehat \V}^{s}_{h,k}}^2,
		\qquad \forall\,\epsilon>0.
		\end{align*}
		By taking $(\Bm^{f,n}_h, \mathcal V^{f,n},\mathcal Q^{f,n})=(\Bl^{f,n}_h, \mathcal U^{f,n},\widetilde P^{f,n})$ and $(\widehat\Bm^{s}_h,\widehat{\mathcal V}^{s},\widehat{\mathcal M}^{s} )=(\widehat\Bl^{s,n}_h,\widehat{\mathcal U}^{s,n},\widehat{\mathcal D}^{s,n} )$ in \eqref{full_discrete_C},  using \eqref{L_defi} and the identity $T^{f,n}_h(\bm w_h^{f,n}  , \mathcal U^{f,n},\mathcal U^{f,n})=0$, and  then following  the same line as in the proof  of Theorem~\ref{theorem_A},  we can get the desired estimate \eqref{sta_4}. Here for brevity  we omit the details of proof.
	\end{proof}
\end{mytheorem}

\begin{mytheorem}[Stability of {Scheme-D}]\label{th_fullD}
	 Let $(\u_h^{f,n},\widehat \u_h^{s,n},p_h^{f,n},\widehat{\bm d}_h^{s,n}, \widehat{\bm \eta}_h^{f,n})$ solve the fully discrete system \eqref{full_discrete_D} for $n=1,\dots,N$.  If the time step $\tau$ satisfies 
\[
	\tau^2<\widetilde \tau_0':=1/\max\limits_{n=1,\cdots,N}\Bigl(G_{2,n,h}\| \nabla {\bm{w}_h^{f,n}}\|^2_{L^\infty ({\mathcal T_{f,h}^{n}})} \Bigr ),
\]
	 then the following  energy estimate holds:
	\begin{align}\label{sta_5}
		&E_h^N+\sum_{n=1}^{N}\left(\tilde E^{n}_{f,h}+\tilde E_{s,d,h}^{n}+\frac{\rho^f}{2}\|\u_h^{f,n-1}-\u_h^{f,n}\circ \X_{n-1,h}^{n}\|^2_{L^2({{\mathcal T}_{f,h}^{n-1}})}\right) \lesssim \tau\sum_{n=1}^{N}\|\bm g^{f,n}\|^2 _{L^{2}({\mathcal T}_{f,h}^{n}) }+E_h^0.
	\end{align}
	\begin{proof}
		By choosing $(\Bm^{f,n}_h, \mathcal V^{f,n},\mathcal Q^{f,n})=(\Bl^{f,n}_h, \mathcal U^{f,n},\widetilde P^{f,n})$ and $(\widehat\Bm^{s}_h,\widehat{\mathcal V}^{s},\widehat{\mathcal M}^{s} )=(\widehat\Bl^{s,n}_h,\widehat{\mathcal U}^{s,n},\widehat{\mathcal D}^{s,n} )$ in \eqref{full_discrete_D}, with the help of   \eqref{L_defi} and $T^{f,n}_h(\bm w_h^{f,n}  , \mathcal U^{f,n},\mathcal U^{f,n})=0$, the desired estimate follows   analogously to that of of Theorem \ref{theorem_B}. Here for brevity  we omit the details of proof.
	\end{proof}
\end{mytheorem}

\begin{myremark} Note that the following spatially semi-discrete HDG method, corresponding to the fully discrete HDG methods (Scheme-C and Scheme-D), is energy conservative if the external forcing $\bm g^f=\bm 0$. 
For any $t\in (0,T]$, find $(\Bl^{f}_h(t), \mathcal U^{f}(t),\widetilde P^{f}(t), \widehat{\bm \eta}_h^{f}(t))  \in \Bs^{f,t}_{h,k-1}\times (\V^{f,t}_{h,k} \times \underline{ \V}^{f,t}_{h,k})\times  (Q^{f,t}_{h, k-1}\times \underline{ Q}^{f,t}_{h,k})\times \widehat \V_{h,k}^{\mathcal M}$ and $(\widehat\Bl^{s}_h(t),\widehat{\mathcal U}^{s}(t),\widehat{\mathcal D}^{s}(t) )\in \widehat \Bs^{s}_{h,k-1}\times (\widehat \V^{s}_{h,k}\times \underline{\widehat \V}^{s}_{h,k})\times  (\widehat \V^{s}_{h,k}\times \underline{\widehat \V}^{s}_{h,k})$, 
with $\widehat{\bm \eta}_h^{f}(t)=\underline{ \widehat{\bm d}}_h^{s}(t)$ on $\widehat\Gamma_h$, $\underline{\u}_h^{f}(t)=\underline{\widehat \u}_h^{s}(t)\circ (\A_e(t))^{-1}$ on $\Gamma_h(t)$, and $\bm w_h^{f}(t)=\partial_t\widehat {\bm \eta}_h^{f}(t)  \circ (\A_h(t))^{-1}$, such that
\begin{align}\label{eq:space_discrete}
\left\{
\begin{aligned}
A_h^{f,t}(\Bl^{f}_h(t), \Bm^{f}_h )+2 \rho^f \mu^f C_{h,1}^{f,t}(\mathcal U^{f}(t), \Bm^{f}_h)&=0,\\
A_h^{s}(\widehat\Bl^{s}_h(t), \widehat\Bm^{s}_h)+2\mu^s C_{h}^{s}(\mathcal D^{s}(t), \widehat \Bm^{s}_h)&=0,\\
\rho^f(\partial_t^M\u_h^{f}(t), \v_h^{f})_{{\mathcal T}_{f,h}(t)}+\rho^f T^{f,t}_h(\bm{w}_h^{f}(t), \mathcal U^{f}(t),\mathcal V^{f})
\\+\rho^f(\nabla{\bm{w}_h^{f}(t)} \u_h^{f}(t)-\tfrac{1}{2} \nabla\cdot{\bm{w}_h^{f}(t)} \u_h^{f}(t), \v_h^{f})_{{{\mathcal T}_{f,h}}(t)}-C_{h,1}^{f,t}(\mathcal V^{f}, \Bl^{f}_h(t))
\\-C_{h,2}^{f,t}(\mathcal V^{f}, \widetilde P^{f}(t))+S_{h}^{f,t}(\mathcal U^{f}(t),  \mathcal V^{f})+\rho^s(\partial_t\widehat \u_h^{s}(t),\widehat \v_h^s )_{\widehat {\mathcal T}_{s,h}}-C_{h}^{s}(\widehat{ \mathcal V}^{s}, \widehat \Bl^{s}_h(t))
\\+\lambda^s(\widehat\nabla\cdot \widehat{\bm{d}}_h^{s}(t), \widehat\nabla\cdot\widehat\v_h^s)_{\widehat {\mathcal T}_{s,h}}+S_{h}^{s}(\mathcal D^{s}(t),  \mathcal V^{s})-(\bm g^{f}(t), \v_h^{f})_{{\mathcal T}_{f,h}^{t}}&=0,\\
    C_{h,2}^{f,t}(\mathcal U^{f}(t),  \mathcal Q ^{f})&=0,\\
	( \partial_t \widehat{\bm{d}}_h^{s}(t), \ms_h \rangle _{\widehat {\mathcal T}_{s,h}}-(\widehat \u^{s}(t),\ms_h ) _{\widehat {\mathcal T}_{s,h}}&=0,\\
	\langle \partial_t  \underline{\widehat{\bm{d}}}_h^{s}(t),\underline{\widehat{\bm m}}_h^{s}\rangle _{\partial \widehat {\mathcal T}_{s,h}}-\langle \underline{ \widehat \u}^{s}(t),\underline{\widehat{\bm m}}_h^{s} \rangle _{\partial \widehat {\mathcal T}_{s,h}}&=0,\\
	(\kappa^f\widehat\nabla \widehat{\bm{\eta}}_h^{f}(t),\widehat\nabla \mf_h)_{\widehat {\mathcal T}_{f,h}} &=0, 
\end{aligned}
\right.
\end{align}
\noindent for all $(\Bm^{f}_h, \mathcal V^{f},\mathcal Q^{f}, \widehat{\bm m}_h^{f})\in  \Bs^{f,t}_{h,k-1}\times (\V^{f,t}_{h,k} \times \underline{ \V}^{f,t}_{h,k})\times  (Q^{f,t}_{h, k-1}\times \underline{ Q}^{f,t}_{h,k})\times \widehat \V_{0,h,k}^{\mathcal M}$ and $(\widehat\Bm^{s}_h,\widehat{\mathcal V}^{s},\widehat{\mathcal M}^{s} )\in \widehat \Bs^{s}_{h,k-1}\times (\widehat \V^{s}_{h,k}\times \underline{\widehat \V}^{s}_{h,k})\times  (\widehat \V^{s}_{h,k}\times \underline{\widehat \V}^{s}_{h,k})$ with  $\underline{\v}_h^{f}=\underline{\widehat \v}_h^{s}\circ (\A_e(t))^{-1}$ on $\Gamma_h(t)$. Here 
\begin{align*}
	&\mathcal U^{f}:=(\u_h^{f},\underline{\u}_h^{f}),\ \mathcal V^{f}:=(\v_h^{f},\underline{\v}_h^{f}) ,\quad
	 \widetilde P^{f}:=(p_h^{f},\underline{p}_h^{f}) ,\  \mathcal Q^{f}:=(q_h^{f},\underline{q}_h^{f}), \\
	&	 \widehat{\mathcal U}^{s}:=(\widehat \u_h^{s}, \underline {\widehat \u}_h^{s} ),   \  \widehat{\mathcal V}^{s}:=(\widehat \v_h^{s}, \underline {\widehat \v}_h^{s} ),  \ \widehat{\mathcal D}^{s}:=(\widehat{\bm d}_h^{s}, \underline{\widehat{\bm d}}_h^{s} ),   \  \widehat{\mathcal M}^{s}:=(\widehat{\bm m}_h^{s}, \underline{\widehat{\bm m}}_h^{s} ), 
\end{align*}
the time-dependent fluid mesh ${{\mathcal T}_{f,h}}(t)$ and the ALE mapping $\A_h(\cdot, t)$ are respectively defined by
\[
{{\mathcal T}_{f,h}}(t):=\{K(t)=\A_h(\widehat K, t): \;  \widehat K\in \widehat{\mathcal{T}}_{f,h}\},
\]
\[
\x=\A_h(\widehat \x, t):=\widehat\x+\etaf_h(\widehat\x,t),\ \forall \widehat \x\in\widehat{K},  \widehat K\in \widehat{\mathcal{T}}_{f,h},
\]
and the finite element spaces $\Bs^{f,t}_{h,k-1}, \V^{f,t}_{h,k}, \underline{ \V}^{f,t}_{h,k}, Q^{f,t}_{h, k-1}$ and $\underline{ Q}^{f,t}_{h,k}$ are defined in the same way as $\Bs^{f,n}_{h,k-1}, \V^{f,n}_{h,k}, \underline{ \V}^{f,n}_{h,k}, Q^{f,n}_{h, k-1}$ and $\underline{ Q}^{f,n}_{h,k}$, with the triangulation ${{\mathcal T}^n_{f,h}}$ replaced by the mesh  ${{\mathcal T}_{f,h}}(t)$. The associated bilinear and trilinear forms are defined accordingly with respect to  ${{\mathcal T}_{f,h}}(t)$. 

	In fact,  for the above semi-discrete scheme, by taking $\bm g^f=\bm 0$ and using \eqref{L_defi}, \eqref{Reynolds} and the skew-symmetric property $T^{f}_h(\cdot, \mathcal U^{f},\mathcal U^{f})=0$,  we can easily derive the following energy identity:
	\begin{align*}
		E_h(T)+\int_0^T\tau\rho^f\mu^f\|\mathcal U^{f}(t)\|^2_{ \V^{f,t}_{h,k} \times \underline{ \V}^{f,t}_{h,k}}dt=E_h(0),
	\end{align*}
	where  the spatially semi-discrete energy
	\begin{align*}
		E_h(t):=\tfrac{\rho^f}{2}\|\u_h^{f}(t)\|^2_{L^2({{\mathcal T}_{f,h}}(t))}+\tfrac{\rho^s}{2} \|\widehat \u_h^{s}(t)\|^2_{L^2(\widehat {\mathcal T}_{s,h}) }+\left (\mu^s\|\widehat{\mathcal D}^{s}(t) \|^2_{ \widehat \V^{s}_{h,k}\times \underline{\widehat \V}^{s}_{h,k} }+\tfrac{\lambda^s }{2}\|\widehat\nabla\cdot {\widehat{\bm{d}}_h^{s} }(t)\|^2_{L^2(\widehat {\mathcal T}_{s,h} )}\right ),
	\end{align*}
with 
	\begin{align*}
		&\|\mathcal U^{f}(t)\|^2_{\V^{f,t}_{h,k} \times \underline{ \V}^{f,t}_{h,k}}:=\|\mathcal K_h^{f}\mathcal U^{f}(t)\|^2_{L^2({{\mathcal T}_{f,h}}(t))}+\|s_1^\frac{1}{2}(\u_h^{f}(t)-\underline{\u}_h^{f}(t))\|^2_{L^2(\partial{\mathcal T}_{f,h}(t))},\\
	&\| \widehat{\mathcal U}^{s}(t)\|^2_{ \widehat \V^{s}_{h,k}\times \underline{\widehat \V}^{s}_{h,k}}:=\|\mathcal K_h^{s}\widehat { \mathcal U}^{s}(t)\|^2_{L^2(\partial\widehat{\mathcal T}_{s,h})}+\|s_2^\frac{1}{2}(\widehat \u_h^{s}(t)-\underline{\widehat \u}_h^{s}(t))\|^2_{L^2(\partial\widehat {\mathcal T}_{s,h})}.
	\end{align*}
	and $\mathcal K_h^{f}$ and $\mathcal K_h^{s}$ are defined in the same way as $\mathcal K_h^{f,n}$ and $ \mathcal K_h^{s,n}$ in \eqref{K_h-f-s}.
\end{myremark}
\section{Numerical experiments}\label{Sec_Num}
In this section, we present three numerical examples. The first two demonstrate the energy stability and convergence of the two  fully discrete  schemes (Scheme-C  and Scheme-D), while the third serves as a  modified Turek-Hron FSI benchmark test with linear elasticity.  All simulations are carried out using the open-source finite element library NGSolve \cite{Sch}, with computational meshes automatically generated from the geometric models via NETGEN \cite{Sch1997} (cf.\ Figure \ref{fig:meshcon}). The nonlinear systems are solved using Newton’s method (cf.\ Remark \ref{rem42}), and the resulting nonsymmetric linearized systems are efficiently addressed by the direct solver UMFPACK \cite{Davis}.
\begin{figure}[htbp]
\centering
\subfigure{\includegraphics[width=0.3\textwidth,height=40.5mm
]{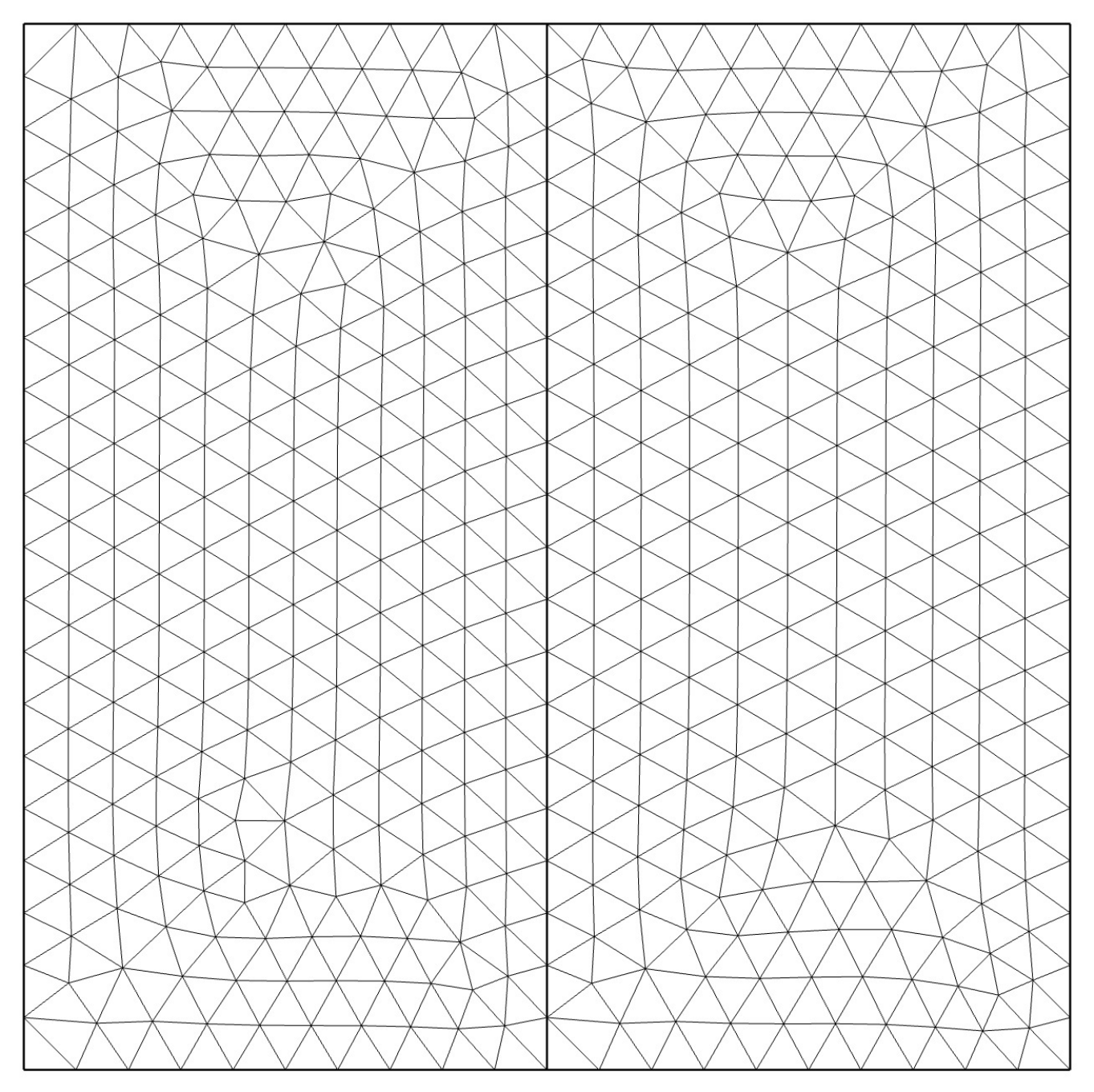}}
\subfigure{\includegraphics[width=0.3\textwidth,height=40mm
]{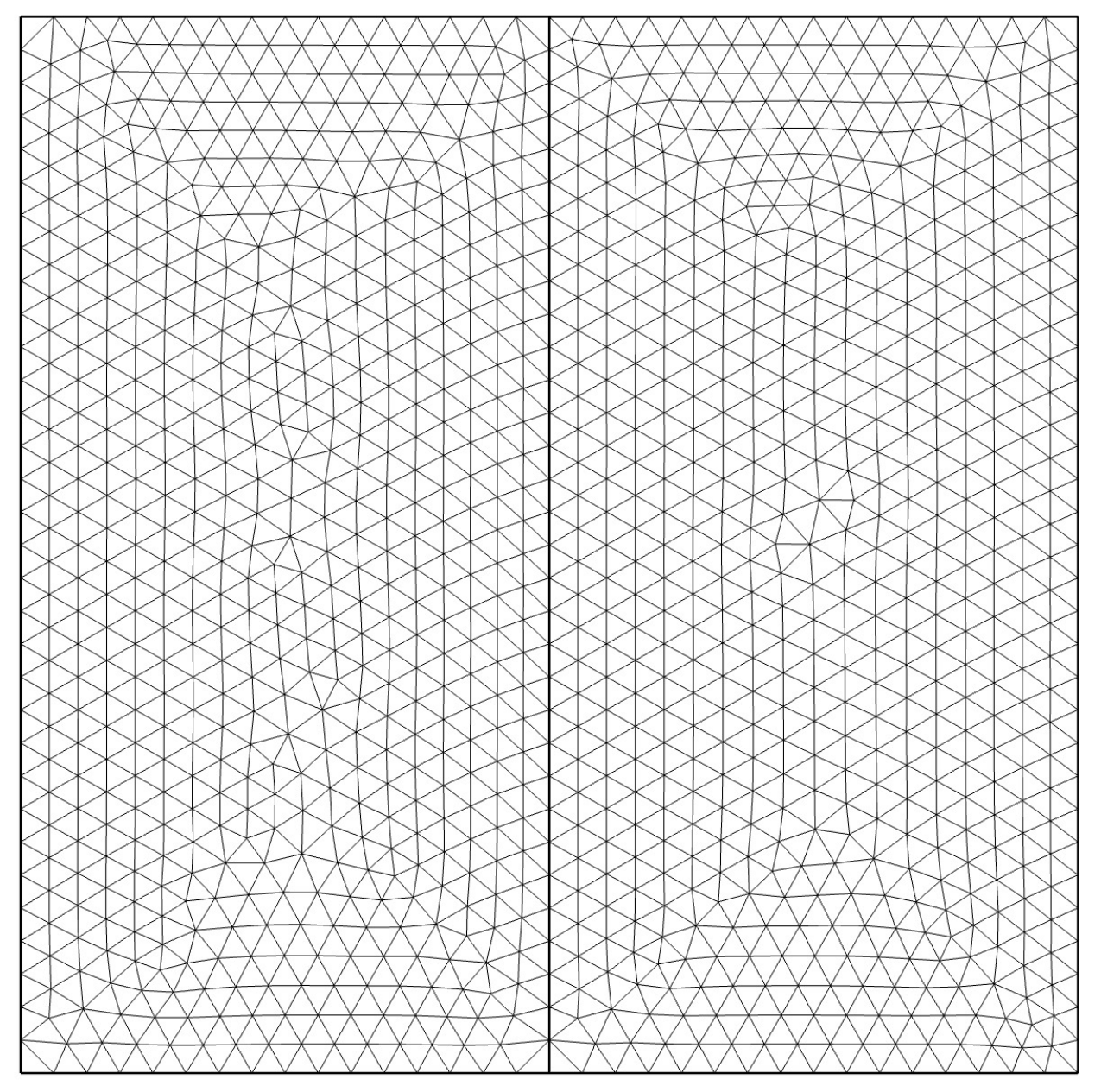}}
\caption{\footnotesize 
Unstructured triangular meshes generated by NETGEN 
with $\texttt{maxh}=1/20$ (left, Example 5.1) and $\texttt{maxh}=1/32$ (right, Example 5.2).}
\label{fig:meshcon}
\end{figure}

\noindent {\bf Example 5.1.} Energy stability test

In the FSI model \eqref{fluid_eq}-\eqref{Dyn_cond}, we take $\Omega=(0,1)\times(0,1)$, with the fluid subdomain $\Orf = (0, \tfrac{1}{2}) \times (0,1)$ and and the structure subdomain $\Ors = (\tfrac{1}{2}, 1) \times (0,1)$. The interface $\widehat\Gamma$ is located at the vertical line $x=\tfrac{1}{2}$. 
All  the physical parameters  are set as $1$, i.e. 
$$\rho^f= \mu^f=\rho^s=\mu^s=\lambda^s=\kappa^f=1,$$ 
and the initial conditions are   prescribed as follows:
\begin{align*}
	\u^f_0&=\left(-1000x ^ 3 y ^ 2  (2 y - 1) (x - 1) ^ 3 (y - 1) ^ 2, 1000x ^ 2 y ^ 3 (2 x - 1) (x - 1) ^ 2 (y - 1) ^ 3) \right)^\top,\\
	\ds_0&= \widehat \u^s_0=\bm 0.
\end{align*}
The initial fluid motion $\u^f_0$  generates a force on the structure via the interface  $\widehat\Gamma$ , causing structural deformation. The resulting displacement of the structure in turn alters the fluid domain, thereby establishing a fully two-way coupled FSI system.

We perform simulations with $k=2$ and  a fixed mesh size $h=1/20$  (cf. Figure \ref{fig:meshcon}) and consider a sequence of decreasing time steps: $\tau=0.02,0.005,0.001, 0.0002$. Figure~\ref{fig:ex1} plots the evolution of the discrete energy 
\begin{align*}
	E^{n}_h=\frac{1}{2}\|\u_h^{f,n}\|^2_{L^2({\mathcal T _{f,h}^n})}+\frac{1}{2} \|\widehat \u_h^{s,n}\|^2_{L^2(\widehat {\mathcal T}_{s,h}) }+\left (\|\widehat{\mathcal D}^{s,n} \|^2_{\widehat \V^{s}_{h,k}\times \underline{\widehat \V}^{s}_{h,k} }+\frac{1 }{2}\|\widehat\nabla\cdot {\widehat{\bm{d}}_h^{s,n} }\|^2_{L^2(\widehat {\mathcal T}_{s,h} )}\right )
\end{align*}
  for both  Scheme-C and Scheme-D at time steps $n=0,\dots,N$. We can see that both of two schemes produce numerical solutions with monotonically decreasing energy in all cases.
\begin{figure}[htbp]
\centering
\subfigure{\includegraphics[width=0.49\textwidth,
height=45mm]{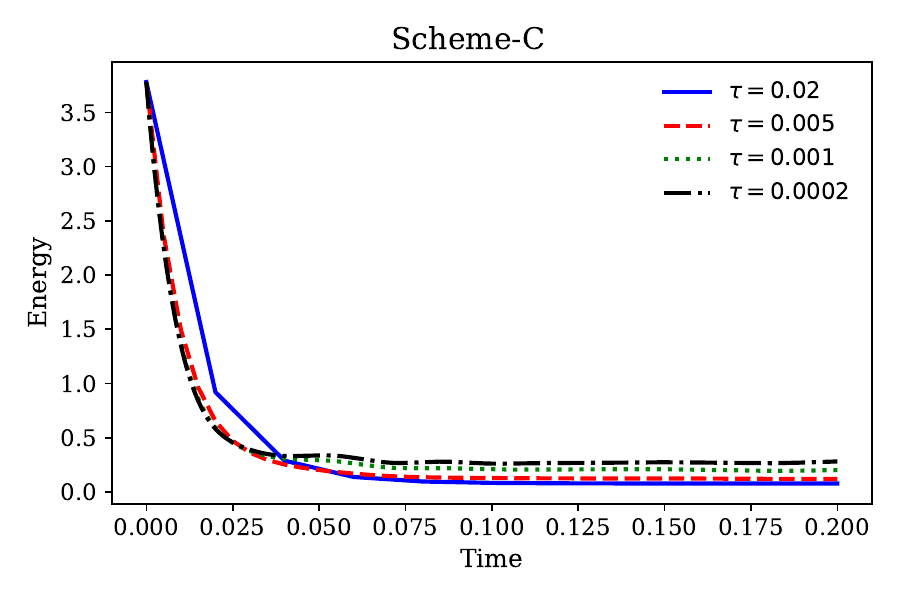}}
\subfigure{\includegraphics[width=0.49\textwidth,
height=45mm]{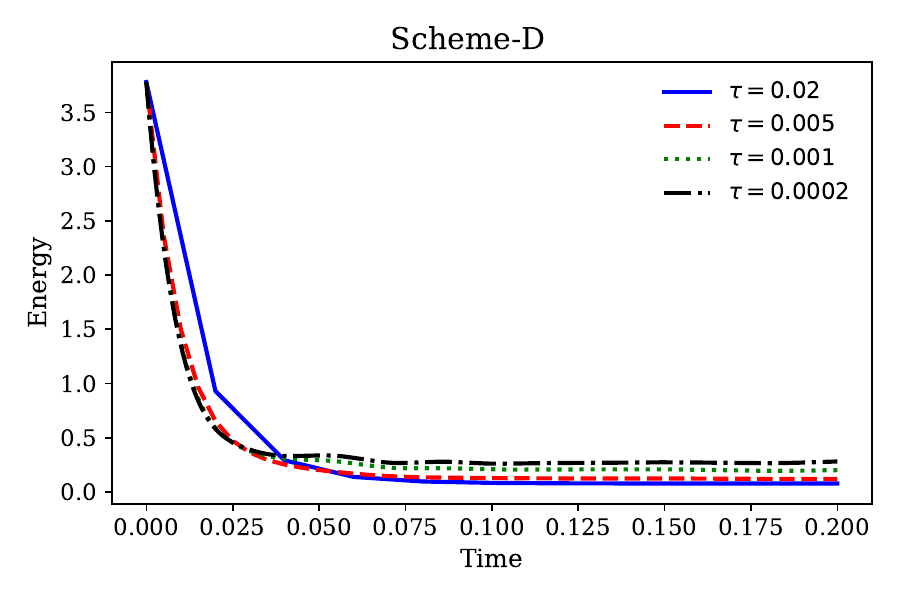}}
\caption{ \footnotesize Energy $E^{n}_h$ of  Scheme-C and Scheme-D with   $k=2$ and $h=1/20$. }
\label{fig:ex1}
\end{figure}

\noindent {\bf Example 5.2.} \ {Convergence test}

In this example, we investigate the accuracy of {Scheme-C}   and {Scheme-D}  when applied to a moving domain FSI problem. The domains $\Omega$, $\Orf$ and $\Ors$  are taken as same as those in Example 5.1. 

The exact solutions of the fluid velocity $\u^f$, the fluid pressure $p^f$, and the structure displacement $\ds$ are given as follows:
\begin{align*}
	\u^f&=(-x ^ 3 y ^ 2  (2 y - 1) (x - 1) ^ 3 (y - 1) ^ 2\sin(2 t), x ^ 2 y ^ 3 (2 x - 1) (x - 1) ^ 2 (y - 1) ^ 3)\sin(2t) )^\top,\\
	\ds&=(-x ^ 3 y ^ 2 (2 y - 1) (x - 1) ^ 3 (y - 1) ^ 2\sin^ 2(t), x ^ 2 y ^ 3  (2 x - 1) (x - 1) ^ 2 (y - 1) ^ 3 \sin^ 2(t))^\top,\\
	p^f&=(x - \tfrac{1}{2})^4 (y - 1)^4\sin(2t).
\end{align*}
We further define the exact displacement in the fluid domain as $\widehat{\bm{\eta}}^f = \ds$, and the exact fluid mesh velocity as $\widehat{\bm{w}}^f = \partial_t \widehat{\bm{\eta}}^f$. Based on these exact solutions, we compute the corresponding forcing terms $\mathbf{g}^f$, $\mathbf{g}^s$, and $\mathbf{g}^D$ for the Navier-Stokes equations \eqref{fluid_eq}, the linear elastodynamics equations \eqref{structure_eq}, and the fluid motion equations \eqref{mesh_equa}, respectively. Moreover, to ensure the continuity of normal stresses across the fluid-structure interface $\widehat\Gamma$, we additionally introduce a forcing term $\mathbf{g}^\Gamma=\widehat \J_f\widehat{\bm{\sigma}}^f\widehat {\bm F}_f^{-\top}\widehat{\bm n}^f-\widehat {\bm{\sigma}}^s\widehat{\bm n}^f$ on the interface $\widehat\Gamma$. 

To investigate the convergence histories for materials with different physical parameters and $k$, we consider the following five cases:
\begin{align*}
	&\text{Case\ 1}: \rho^f= \mu^f=\rho^s=\mu^s=\lambda^s=\kappa^f=1, \quad  k=1,2;\\
	&\text{Case\ 2}: \rho^f=\rho^s=10^{3}, \mu^f=10^{-3}, \mu^s=\lambda^s=\kappa^f=1, \quad  k=1,2;\\
	&\text{Case\ 3}: \rho^s=10^{3}, \mu^s=10^{2}, \lambda^s=10^{6}, \rho^f=\mu^f=\kappa^f=1, \quad  k=1,2;\\
	&\text{Case\ 4}: \rho^f= \rho^s=\mu^s=\lambda^s=\kappa^f=1,  \mu^f=10^{-6}, 10^{-8}, \quad  k=2.
\end{align*}
In particular, Case 3   corresponds to a nearly incompressible elastic structure ($ \lambda^s=10^{6}$), and Case  4  corresponds to very small fluid viscosity ($\mu^f=10^{-6}, 10^{-8}$). 

We conduct simulations on a series of uniform triangular meshes with mesh sizes $h=1/4$, $1/8,$ $1/16,$ $1/32$ (cf. Figure \ref{fig:meshcon}), and list the corresponding numerical results    in Tables 1-6, where the errors are defined as follows:
\begin{align*}
	&e_{\u}:=\frac{\|\u^{f,N}_h-\u^f(T)\|_{L^2( {\mathcal T}_{f,h}^{n})}}{\|\u^{f,N}_h\|_{L^2( {\mathcal T}_{f,h}^{n})}}+\frac{\|\widehat \u^{s,N}_h-\widehat \u^s(T)\|_{L^2(\widehat {\mathcal T}_{s,h})}}{\|\widehat \u^{s,N}_h\|_{L^2(\widehat {\mathcal T}_{s,h})}},\\
	&e_{\nabla\cdot \u^f}:=\|\nabla\cdot  \u^{f,N}_h\|_{L^2( {\mathcal T}_{f,h}^{n})}, \qquad e_{\widehat{\bm{\eta}}^f}:=\frac{\|{\widehat{\bm{\eta}}}^{f,N}_h-{\widehat{\bm{\eta}}}^f(T)\|_{L^2(\widehat  {\mathcal T}_{f,h})}}{\|{\widehat{\bm{\eta}}}^{f,N}_h\|_{L^2(\widehat  {\mathcal T}_{f,h})}},\\
	&e_{{\widehat{ \bm{d}}}^s}:=\frac{\|\widehat {\bm{d}}^{s,N}_h-\widehat {\bm{d}}^s(T)\|_{L^2(\widehat {\mathcal T}_{s,h})}}{\|\widehat {\bm{d}}^{s,N}_h\|_{L^2(\widehat {\mathcal T}_{s,h})}},\qquad  \ e_{p}:=\frac{\|p^{f,N}_h-p^f(T)\|_{L^2( {\mathcal T}_{f,h}^{n})}}{\|p^{f,N}_h\|_{L^2( {\mathcal T}_{f,h}^{n})}}
\end{align*}
\begin{align*}
&e_{\widehat \nabla {\widehat{\bm{\eta}}}^f}:=\frac{\|\widehat \nabla(\widehat{ {\bm{\eta}}}^{f,N}_h-\widehat{\bm{\eta}}^f(T))\|_{L^2(\widehat {\mathcal T}_{f,h})}}{\|\widehat \nabla \widehat {\bm{\eta}}^{f,N}_h\|_{L^2(\widehat  {\mathcal T}_{f,h})}},\qquad e_{\widehat \nabla {\widehat{ \bm{d}}}^s}:=\frac{\|\widehat \nabla(\widehat{ {\bm{d}}}^{s,N}_h- \widehat {\bm{d}}^s(T))\|_{L^2(\widehat {\mathcal T}_{s,h})}}{\|\widehat \nabla \widehat {\bm{d}}^{s,N}_h\|_{L^2(\widehat {\mathcal T}_{s,h})}},
\\
&e_{\nabla\u}:=\frac{\|\nabla(\u^{f,N}_h-\u^f(T))\|_{L^2( {\mathcal T}_{f,h}^{n})}}{\|\nabla\u^{f,N}_h\|_{L^2( {\mathcal T}_{f,h}^{n})}}+\frac{\|\widehat \nabla(\widehat \u^{s,N}_h-\widehat\u^s(T))\|_{L^2(\widehat {\mathcal T}_{s,h})}}{\|\widehat \nabla \widehat \u^{s,N}_h\|_{L^2(\widehat {\mathcal T}_{s,h})}}.
\end{align*}
As the expected optimal convergence rates are $\mathcal O(\tau + h^{k+1})$ for   $e_{\u} $, $e_{{\widehat{ \bm{d}}}^s}$ and $e_{\widehat{\bm{\eta}}^f}$,  and are $\mathcal O(\tau + h^{k})$ for $e_{\nabla\u}, e_{\widehat \nabla {\widehat{ \bm{d}}}^s}$ and $ e_{p}$,  in the computation  we take $\tau = h^{k+1}$ to ensure the optimal spatial accuracy.

Tables \ref{tab:2-1}-\ref{tab:2-3} (resp. Tables \ref{tab:2-4}-\ref{tab:2-6}) show the numerical results of  Scheme-C   and Scheme-D   with $k=1$ (resp. $k=2$) under various parameter configurations (Cases 1-3). The observed convergence rates are consistent with theoretical predictions: $e_{\u} $, $e_{{\widehat{ \bm{d}}}^s}$ and $e_{\widehat{\bm{\eta}}^f}$ are of     $(k+1)^{th}$-order convergence, and $e_{\nabla\u}, e_{\widehat \nabla {\widehat{ \bm{d}}}^s}$ and $ e_{p}$ are of   $k^{th}$-order convergence.  Notably,   the numerical results of  $e_{\nabla\cdot \u^f}$ in Tables~\ref{tab:2-1}-\ref{tab:2-6} is nearly zero with machine precision, which indicates that the obtained fluid velocity approximation is point-wise divergence-free, 
as is conformable to Theorem \ref{Div-free}.  It should be mentioned that most of the numerical results of the two schemes are almost the same (4 effective digits for each result).  In fact, as pointed out in Remark \ref{CD-equal}, Scheme-C and Scheme-D are   equivalent  if $\nabla\bm w_h^{f,n}=0$ for any $n$. This means that they may be very close when $|\nabla\bm w_h^{f,n}|$ is very small.

\begin{table}[htbp]
\centering
\tabcolsep 1.5mm {\footnotesize \caption{ Errors and convergence rates  for Case 1 with $k=1$ and $\tau=h^2$ at $T=0.3$.}
\label{tab:2-1}
\begin{tabular}{c|ccccccccccr}
\hline
\rule[0pt]{0pt}{10pt}
&$h$& $e_{\u}$ &$e_{p}$&$e_{\widehat{\bm{\eta}}^f}$&$e_{{\widehat{ \bm{d}}}^s}$&$e_{\nabla\cdot \u^f}$&$e_{\nabla\u}$&$e_{\widehat \nabla \widehat{\bm{\eta}}^f}$&$ e_{\widehat \nabla {\widehat{ \bm{d}}}^s}$&\\[2pt]
\hline
& 1/4 &1.687e-01& 4.612e-01& 2.877e-01&1.851e-01&4.641e-20&5.648e-01&5.819e-01& 5.625e-01  \\
{Scheme-C}&1/8 &3.177e-02& 2.360e-01& 6.434e-02&4.061e-02&1.661e-19&2.801e-01&2.938e-01& 2.816e-01  \\
&1/16&8.977e-03& 1.381e-01& 1.898e-02&9.861e-03&3.799e-19&1.450e-01&1.600e-01& 1.377e-01\\
&1/32&2.202e-03& 6.736e-02& 4.700e-03&2.405e-03&7.178e-19&7.242e-02&8.002e-02& 6.863e-02  \\
& Order& 2.087& 0.925& 1.979& 2.089& --& 0.988& 0.954& 1.012&\\[2pt]
\hline
& 1/4 &1.687e-01& 4.612e-01& 2.877e-01&1.851e-01&4.819e-20&5.648e-01&5.819e-01& 5.625e-01   \\
{Scheme-D}&1/8 &3.177e-02& 2.360e-01& 6.434e-02&4.061e-02&1.350e-19&2.801e-01&2.938e-01& 2.816e-01 \\
&1/16&8.977e-03& 1.381e-01& 1.898e-02&9.861e-03&3.290e-19&1.450e-01&1.600e-01& 1.377e-01\\
&1/32& 2.202e-03& 6.736e-02& 4.700e-03&2.405e-03&6.544e-19&7.242e-02&8.002e-02& 6.863e-02 \\
& Order&2.087& 0.925& 1.979& 2.089& --& 0.988& 0.954& 1.012& \\[2pt]
\hline
\end{tabular}}
\end{table}

\begin{table}[htbp]
\begin{center}
\tabcolsep 1.5mm {\footnotesize \caption{ Errors and convergence rates  for Case 2 with $k=1$ and $\tau=h^2$ at $T=0.3$.}
\label{tab:2-2}
\begin{tabular}{c|ccccccccccr}
\hline
\rule[0pt]{0pt}{10pt}
&$h$& $e_{\u}$ &$e_{p}$&$e_{\widehat{\bm{\eta}}^f}$&$e_{{\widehat{ \bm{d}}}^s}$&$e_{\nabla\cdot \u^f}$&$e_{\nabla\u}$&$e_{\widehat \nabla \widehat{\bm{\eta}}^f}$&$ e_{\widehat \nabla {\widehat{ \bm{d}}}^s}$&\\[2pt]
\hline
& 1/4 & 1.776e-01& 4.611e-01& 3.164e-01&2.403e-01&2.779e-20&5.603e-01&2.318e-01& 2.209e-01   \\
{Scheme-C}&1/8 &3.063e-02& 2.361e-01& 6.466e-02&6.435e-02&1.581e-19&2.708e-01&2.937e-01& 2.772e-01  \\
&1/16&8.697e-03& 1.381e-01& 1.922e-02&1.527e-02&3.251e-19&1.412e-01&1.600e-01& 1.340e-01\\
&1/32& 2.163e-03& 6.737e-02& 4.751e-03&3.611e-03&6.316e-19&7.128e-02&8.001e-02& 6.544e-02\\
& Order& 2.083& 0.938& 1.973& 2.038&  --& 0.970& 0.953& 1.033&\\[2pt]
\hline
& 1/4 & 1.646e-01& 4.734e-01& 2.877e-01&2.499e-01&4.929e-20&5.363e-01&5.810e-01& 5.609e-01  \\
{Scheme-D}&1/8 & 3.063e-02& 2.361e-01& 6.466e-02&6.435e-02&1.374e-19&2.708e-01&2.937e-01& 2.772e-01\\
&1/16&8.697e-03& 1.381e-01& 1.922e-02&1.527e-02&3.267e-19&1.412e-01&1.600e-01& 1.340e-01\\
&1/32&2.163e-03& 6.737e-02& 4.750e-03&3.611e-03&6.367e-19&7.128e-02&8.001e-02& 6.544e-02 \\
& Order& 2.083& 0.938& 1.973& 2.038& --&  0.970& 0.953& 1.033&\\[2pt]
\hline
\end{tabular}}
\end{center}
\end{table}

\begin{table}[htbp]
\begin{center}
\tabcolsep 1.5mm {\footnotesize\caption{ Errors and convergence rates  for Case 3 with $k=1$ and $\tau=h^2$ at $T=0.3$.}
\label{tab:2-3}
\begin{tabular}{c|ccccccccccr}
\hline
\rule[0pt]{0pt}{10pt}
&$h$& $e_{\u}$ &$e_{p}$&$e_{\widehat{\bm{\eta}}^f}$&$e_{{\widehat{ \bm{d}}}^s}$&$e_{\nabla\cdot \u^f}$&$e_{\nabla\u}$&$e_{\widehat \nabla \widehat{\bm{\eta}}^f}$&$ e_{\widehat \nabla {\widehat{ \bm{d}}}^s}$&\\[2pt]
\hline
& 1/4 &1.237e+00& 5.492e-01& 3.290e-01&1.939e-01&1.379e-19&8.433e-01&6.249e-01& 5.346e-01    \\

{Scheme-C}&1/8 &1.795e-01& 2.556e-01& 6.758e-02&5.441e-02&1.538e-19&3.700e-01&2.994e-01& 2.787e-01  \\
&1/16&3.500e-02& 1.419e-01& 1.928e-02&1.361e-02&3.526e-19&1.636e-01&1.608e-01& 1.371e-01  \\
&1/32& 7.083e-03& 6.821e-02& 4.742e-03&3.425e-03&6.652e-19&7.663e-02&8.011e-02& 6.845e-02 \\
& Order&2.483& 1.003& 2.039& 1.941& --& 1.153& 0.988& 0.988& \\[2pt]
\hline
& 1/4 & 1.237e+00& 5.492e-01& 3.290e-01&1.939e-01&1.193e-19&8.433e-01&6.249e-01& 5.346e-01  \\
{Scheme-D}&1/8 &1.795e-01& 2.556e-01& 6.758e-02&5.441e-02&1.394e-19&3.700e-01&2.994e-01& 2.787e-01  \\
&1/16&3.500e-02& 1.419e-01& 1.928e-02&1.361e-02&3.354e-19&1.636e-01&1.608e-01& 1.371e-01 \\
&1/32& 7.083e-03& 6.821e-02& 4.742e-03&3.425e-03&6.748e-19&7.663e-02&8.011e-02& 6.845e-02\\
& Order&2.483& 1.003& 2.039& 1.941& --& 1.153& 0.988& 0.988& \\[2pt]
\hline
\end{tabular}}
\end{center}
\end{table}

\begin{table}[htbp]
\begin{center}
\tabcolsep 1.5mm {\footnotesize\caption{ Errors and convergence rates  for Case 1 with $k=2$ and $\tau=h^3$ at $T=0.1$.}
\label{tab:2-4}
\begin{tabular}{c|ccccccccccr}
\hline
\rule[0pt]{0pt}{10pt}
&$h$& $e_{\u}$ &$e_{p}$&$e_{\widehat{\bm{\eta}}^f}$&$e_{{\widehat{ \bm{d}}}^s}$&$e_{\nabla\cdot \u^f}$&$e_{\nabla\u}$&$e_{\widehat \nabla \widehat{\bm{\eta}}^f}$&$ e_{\widehat \nabla {\widehat{ \bm{d}}}^s}$&\\[2pt]
\hline
& 1/4 &5.837e-02& 1.026e-01& 8.012e-02&1.420e-01&2.187e-20&2.446e-01&2.662e-01& 3.022e-01   \\
{Scheme-C}&1/8&5.584e-03& 2.542e-02& 9.107e-03&1.743e-02&5.630e-20&5.330e-02&5.555e-02& 5.776e-02 \\
&1/16&7.250e-04& 8.430e-03& 1.248e-03&2.177e-03&1.101e-19&1.368e-02&1.487e-02& 1.330e-02 \\
&1/32&8.303e-05& 1.996e-03& 1.538e-04&2.722e-04&2.222e-19&3.315e-03&3.575e-03& 3.184e-03  \\
& Order & 3.153 & 1.895 & 3.008 & 3.009 & -- & 2.068 & 2.073 & 2.190 \\[2pt]
\hline
& 1/4 & 5.837e-02& 1.026e-01& 8.012e-02&1.420e-01&1.899e-20&2.446e-01&2.662e-01& 3.022e-01  \\
{Scheme-D}&1/8 &5.584e-03& 2.542e-02& 9.107e-03&1.743e-02&5.458e-20&5.330e-02&5.555e-02& 5.776e-02 \\
&1/16&7.250e-04& 8.430e-03& 1.248e-03&2.177e-03&1.134e-19&1.368e-02&1.487e-02& 1.330e-02 \\
&1/32&8.303e-05& 1.996e-03& 1.538e-04&2.722e-04&2.255e-19&3.315e-03&3.575e-03& 3.184e-03 \\
& Order & 3.153 & 1.895 & 3.008 & 3.009 & -- & 2.068 & 2.073 & 2.190 \\[2pt]
\hline
\end{tabular}}
\end{center}
\end{table}

\begin{table}[htbp]
\centering
\tabcolsep 1.5mm {\footnotesize\caption{ Errors and convergence rates  for Case 2 with $k=2$ and $\tau=h^3$ at $T=0.1$.}
\label{tab:2-5}
\begin{tabular}{c|ccccccccccr}
\hline
\rule[0pt]{0pt}{10pt}
&$h$& $e_{\u}$ &$e_{p}$&$e_{\widehat{\bm{\eta}}^f}$&$e_{{\widehat{ \bm{d}}}^s}$&$e_{\nabla\cdot \u^f}$&$e_{\nabla\u}$&$e_{\widehat \nabla \widehat{\bm{\eta}}^f}$&$ e_{\widehat \nabla {\widehat{ \bm{d}}}^s}$&\\[2pt]
\hline
& 1/4 &4.532e-02& 1.604e-01& 8.156e-02&1.693e-01&2.110e-20&2.627e-01&2.682e-01& 3.634e-01   \\
{Scheme-C}&1/8 & 4.152e-03& 2.679e-02& 9.186e-03&1.973e-02&5.201e-20&5.530e-02&5.558e-02& 6.519e-02 \\
&1/16&6.255e-04& 8.505e-03& 1.254e-03&2.449e-03&1.146e-19&1.406e-02&1.487e-02& 1.452e-02 \\
&1/32&7.919e-05& 2.007e-03& 1.541e-04&3.061e-04&2.221e-19&3.345e-03&3.575e-03& 3.361e-03 \\
& Order & 3.054 & 2.107 & 3.016 & 3.037 & -- & 2.098 & 2.076 & 2.252&\\[2pt]
\hline
& 1/4 &4.532e-02& 1.604e-01& 8.156e-02&1.693e-01&2.053e-20&2.627e-01&2.682e-01& 3.634e-01    \\
{Scheme-D}&1/8 &4.152e-03& 2.679e-02& 9.186e-03&1.973e-02&5.659e-20&5.530e-02&5.558e-02& 6.519e-02 \\
&1/16&6.255e-04& 8.505e-03& 1.254e-03&2.449e-03&1.116e-19&1.406e-02&1.487e-02& 1.452e-02 \\
&1/32&7.919e-05& 2.007e-03& 1.541e-04&3.061e-04&2.196e-19&3.345e-03&3.575e-03& 3.361e-03 \\
& Order & 3.054 & 2.107 & 3.016 & 3.037 & -- & 2.098 & 2.076 & 2.252&\\[2pt]
\hline
\end{tabular}}
\end{table}

\begin{table}[htbp]
\centering
\tabcolsep 1.5mm {\footnotesize \caption{ Errors and convergence rates  for Case 3 with $k=2$ and $\tau=h^3$ at $T=0.1$.}
\label{tab:2-6}
\begin{tabular}{c|ccccccccccr}
\hline
\rule[0pt]{0pt}{10pt}
&$h$& $e_{\u}$ &$e_{p}$&$e_{\widehat{\bm{\eta}}^f}$&$e_{{\widehat{ \bm{d}}}^s}$&$e_{\nabla\cdot \u^f}$&$e_{\nabla\u}$&$e_{\widehat \nabla \widehat{\bm{\eta}}^f}$&$ e_{\widehat \nabla {\widehat{ \bm{d}}}^s}$&\\[2pt]
\hline
& 1/4 &8.118e-02& 1.132e-01& 9.708e-02&1.669e-01&1.830e-20&2.701e-01&2.814e-01& 3.333e-01   \\
{Scheme-C}&1/8 & 8.755e-03& 2.696e-02& 1.098e-02&1.962e-02&5.061e-20&5.819e-02&5.660e-02& 5.919e-02 \\
&1/16&1.099e-03& 8.593e-03& 1.452e-03&2.447e-03&1.088e-19&1.438e-02&1.492e-02& 1.348e-02 \\
&1/32&1.439e-04& 2.016e-03& 1.794e-04&3.058e-04&2.219e-19&3.413e-03&3.578e-03& 3.220e-03 \\
& Order & 3.047 & 1.937 & 3.027 & 3.031 & -- & 2.102 & 2.099 & 2.231& \\[2pt]
\hline
& 1/4 & 8.118e-02& 1.132e-01& 9.708e-02&1.669e-01&2.164e-20&2.701e-01&2.814e-01& 3.333e-01   \\
{Scheme-D}&1/8 &8.755e-03& 2.696e-02& 1.098e-02&1.962e-02&5.451e-20&5.819e-02&5.660e-02& 5.919e-02 \\
&1/16&1.099e-03& 8.593e-03& 1.452e-03&2.447e-03&1.140e-19&1.438e-02&1.492e-02& 1.348e-02 \\
&1/32&1.439e-04& 2.016e-03& 1.794e-04&3.058e-04&2.231e-19&3.413e-03&3.578e-03& 3.220e-03  \\
& Order & 3.047 & 1.937 & 3.027 & 3.031 & -- & 2.102 & 2.099 & 2.231&\\[2pt]
\hline
\end{tabular}}
\end{table}

For incompressible fluid flow problems it is known that the divergence-free property of the discrete fluid velocity is key  to the pressure-robustness of the used  divergence-free  numerical method  \cite{Garc, John}. To  illustrate this effect for our proposed methods,  we compute Case 4, where  the fluid  viscosity parameter  $\nu^f = 10^{-6}, 10^{-8}$, and list the numerical results of $e_u$ and $e_{\nabla u}$  in Table \ref{tab:2-7}.    For comparison, we also compute the conforming finite element scheme that applies the P2-P1 Taylor--Hood element to discretize the fluid velocity-pressure  pair and  the continuous P2 element to discretize  the structural velocity and displacement. Due to   enhancing the fluid  incompressibility      only
weakly, this scheme (abbr. P2-P1 Taylor--Hood discretization)    does not yield the divergence-free velocity approximation.

From Table \ref{tab:2-7}  we can see that the velocity errors of  the   Taylor--Hood discretization  on each fixed mesh resolution increase significantly  as the viscosity $\nu^f$ decreases.  This means that this scheme  is not pressure-robust.
In contrast to  the   Taylor--Hood discretization,  both  Scheme-C and Scheme-D are pressure-robust and  give almost uniformly accurate numerical results at $\nu^f = 10^{-6}, 10^{-8}$.

\begin{table}[htbp]
\centering
\footnotesize
\renewcommand{\arraystretch}{1.2}
\caption{Velocity errors for different viscosity parameters $\nu^f$ at $T=0.1$ with $\tau=h^3$. }
\label{tab:2-7}
\medskip
\textbf{P2--P1 Taylor--Hood discretization}
\medskip
\begin{tabular}{c|cc|cc|cc|cc}
\hline 
$h$ 
& \multicolumn{2}{c|}{$\nu^f = 1$}
& \multicolumn{2}{c|}{$\nu^f = 10^{-2}$}
& \multicolumn{2}{c|}{$\nu^f = 10^{-4}$} 
& \multicolumn{2}{c}{$\nu^f = 10^{-6}$}\\
\cline{2-9}
& $e_u$ & $e_{\nabla u}$
& $e_u$ & $e_{\nabla u}$
& $e_u$ & $e_{\nabla u}$
& $e_u$ & $e_{\nabla u}$ \\
\hline
$1/4$  & 9.267e-02 & 2.890e-01 & 1.567e+00 & 2.864e+00 & 2.311e+00 & 4.313e+00 & 2.322e+00 & 4.335e+00 \\
$1/8$  & 9.945e-03 & 6.534e-02 & 4.599e-01 & 1.617e+00 & 1.477e+00 & 6.169e+00 & 1.515e+00 & 6.358e+00 \\
$1/16$ & 1.293e-03 & 1.655e-02 & 7.461e-02 & 5.171e-01 & 7.690e-01 & 6.199e+00 & 8.403e-01 & 6.820e+00 \\
$1/32$ & 1.343e-04 & 3.784e-03 & 6.566e-03 & 9.205e-02 & 2.051e-01 & 3.430e+00 & 2.906e-01 & 4.995e+00 \\
\hline
\end{tabular}

\medskip
\textbf{Scheme-C and Scheme-D ($k=2$)}
\medskip
\begin{tabular}{c|c|cccc|cccc}
\hline
 & $h$
 & \multicolumn{4}{c|}{Scheme-C}
 & \multicolumn{4}{c}{Scheme-D} \\
\cline{3-10}
 & 
 & $e_u$ & order & $e_{\nabla u}$ & order
 & $e_u$ & order & $e_{\nabla u}$ & order \\
\hline
$\nu^f= 10^{-6}$ & $1/4$  & 7.495e-02 & --   & 3.070e-01 & --   & 7.495e-02 & --   & 3.070e-01 & --   \\
          & $1/8$  & 1.255e-02 & 2.58 & 7.742e-02 & 1.99 & 1.255e-02 & 2.58 & 7.742e-02 & 1.99 \\
          & $1/16$ & 4.130e-03 & 1.60 & 3.625e-02 & 1.09 & 4.130e-03 & 1.60 & 3.625e-02 & 1.09 \\
          & $1/32$ & 1.481e-03 & 1.48 & 2.175e-02 & 0.74 & 1.481e-03 & 1.48 & 2.175e-02 & 0.74 \\
\hline
$\nu^f=10^{-8}$ & $1/4$  & 2.193e+00 & --   & 4.969e+00 & --   & 2.194e+00 & --   & 4.971e+00 & --   \\
          & $1/8$  & 3.911e-02 & 5.81 & 1.491e-01 & 5.06 & 3.911e-02 & 5.81 & 1.491e-01 & 5.06 \\
          & $1/16$ & 7.645e-03 & 2.35 & 6.483e-02 & 1.20 & 7.645e-03 & 2.35 & 6.484e-02 & 1.20 \\
          & $1/32$ & 1.487e-03 & 2.36 & 2.190e-02 & 1.57 & 1.487e-03 & 2.36 & 2.190e-02 & 1.57 \\
\hline
\end{tabular}
\end{table}

\noindent {\bf Example 5.3.} \ {Flow past a cylinder with a flexible bar}

In our final example, we assess the performance of the two proposed numerical schemes with $k=3$ on the classical benchmark problem introduced by Turek and Hron \cite{Turek2006,Turek2010}, which is based on the well-known configuration of incompressible channel flow past a rigid cylinder. The original setup involves a two-dimensional incompressible flow interacting with a nonlinearly elastic bar attached to the downstream side of the cylinder. In our study, we retain the same geometric and flow configuration but replace the nonlinear elasticity with a linear elastic model, aiming to investigate whether oscillatory behavior persists under this simplified structural assumption. The benchmark geometry consists of a slightly asymmetric channel with a circular cylinder mounted near the inlet.  For the FSI configuration, an elastic bar  is affixed to the downstream side of the cylinder.  The detailed geometric specifications are illustrated in Figure~\ref{fig:THgeo}.

\begin{figure}[htbp]
\centering
\subfigure{\includegraphics[width=0.98\textwidth,
height=40mm]{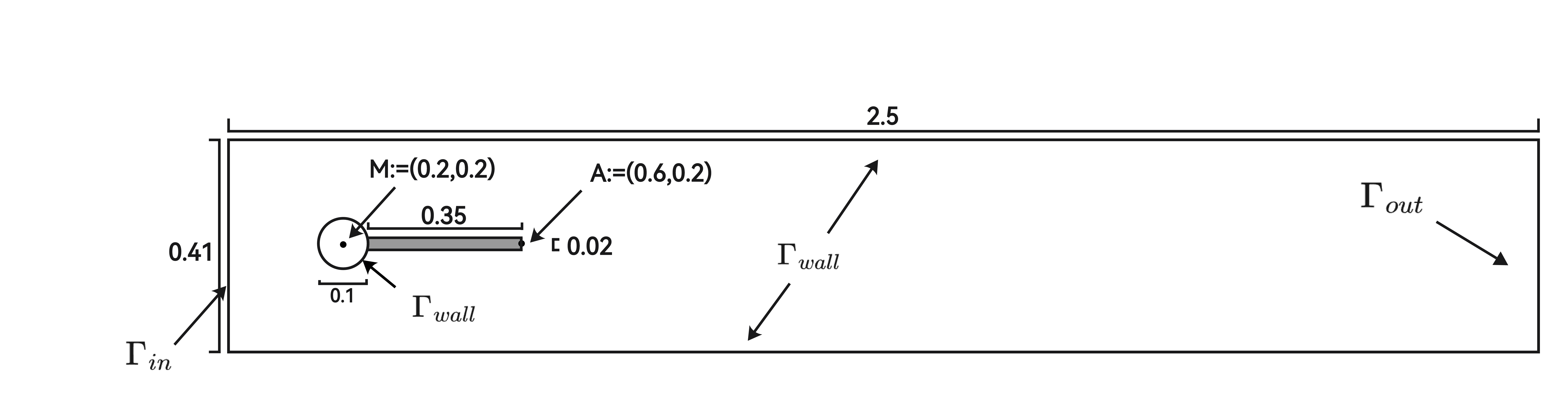}}
\caption{ \footnotesize Geometric configuration of the benchmark problem.}
\label{fig:THgeo}
\end{figure}

The fluid domain is governed by the incompressible Navier-Stokes equations~\eqref{fluid_eq} with zero external forcing $\bm g^f=\mathbf 0$, while the structure is modeled using the linear elasticity equations~\eqref{structure_eq}. The fluid-structure interaction is enforced via the coupling conditions~\eqref{Kine_cond}-\eqref{Dyn_cond} on the interface $\Gamma(t)$. The following boundary conditions are applied:

\begin{itemize}
  \item A time-dependent parabolic inflow profile is prescribed at the inlet of the fluid channel:
  \begin{align*}
  	 \mathbf{u}^f(0, y, t) =
  \begin{cases}
  \mathbf{u}^f(0, y) \cdot \dfrac{1 - \cos\left( \frac{\pi}{2} t \right)}{2}, & \text{if}\ t < 2, \\
  \mathbf{u}^f(0, y), & \text{otherwise},
  \end{cases}
  \end{align*}
  where the steady inflow profile is given by
  \begin{align*}
  	\mathbf{u}^f(0, y) = 1.5\, \bar{U} \cdot \frac{4.0}{0.1681} \cdot y(0.41 - y).
  \end{align*}
  \item A stress-free boundary condition $\bm{\sigma}^f \mathbf{n} = \mathbf{0}$ is imposed at the outlet.
  \item A no-slip condition is enforced on all other fluid boundaries, i.e., $ \mathbf{u}^f = 0 $. On the structure, a clamped boundary condition $\widehat{ \bm{d}}^s = \widehat {\u}^s = 0 $ is applied on the fixed boundaries.
\end{itemize}
Two test cases with time-periodic solutions are considered, denoted as FSI2 and FSI3 in \cite{Turek2006}. The associated material parameters are presented in Table~\ref{tab:para}. One of the quantities used for comparison is the displacement of control point A, located at the free end of the elastic flag (see Figure~\ref{fig:THgeo}). 
The computational mesh employed in the simulations is generated automatically from the geometric model using NETGEN \cite{Sch1997}, and consists of 1,992 triangular elements (see Figure~\ref{fig:THmesh}). The time step size is taken as $\tau = 0.001$.

To assess the performance of the two proposed schemes, we also compute, for comparison,   the continuous Galerkin (CG)  method where the fluid velocity-pressure pair is discretized using the $P_3-P_{2}$ Taylor-Hood element and  the structural displacement /velocity are approximated by the continuous $P_3$  element.
 The $P_3$ curved element  is adopted in the vicinity of the cylinder to better capture the curved boundary geometry.  For the corresponding nonlinear systems, the Newton’s method is employed with a stopping criterion based on a residual tolerance of $10^{-6}$. At each Newton iteration, the resulting linear system is solved using a sparse direct solver.

\begin{table}[htbp]
\centering
\renewcommand{\arraystretch}{1.4}  
\tabcolsep=2mm
\tabcolsep 2mm {\footnotesize\caption{ Parameter settings for the two test cases.}
\label{tab:para}
\begin{tabular}{c|c|c|c|c|c|c}
\hline
Parameter       & $\rho^s [10^3 \, \frac{\text{kg}}{\text{m}^3}]$ & $\lambda^s [10^6 \, \frac{\text{kg}}{\text{m} \cdot \text{s}^2}]$ & $\mu^s [10^6 \, \frac{\text{kg}}{\text{m} \cdot \text{s}^2}]$ & $\rho^f [10^3 \, \frac{\text{kg}}{\text{m}^3}]$ & $\mu^f [10^{-3}\frac{\text{kg}}{\text{m}\cdot\text {s}}]$ & $\bar{U} [\frac{\text{m}}{\text{s}}]$ \\ \hline
FSI2 & 10     & 2.0   & 0.5   & 1  & 1    & 1       \\ \hline
FSI3 & 1      & 8.0   & 2.0   & 1  & 1    & 2       \\ \hline
\end{tabular}}
\end{table}

\begin{figure}[htbp]
\centering
\subfigure{\includegraphics[width=0.82\textwidth,
height=20mm]{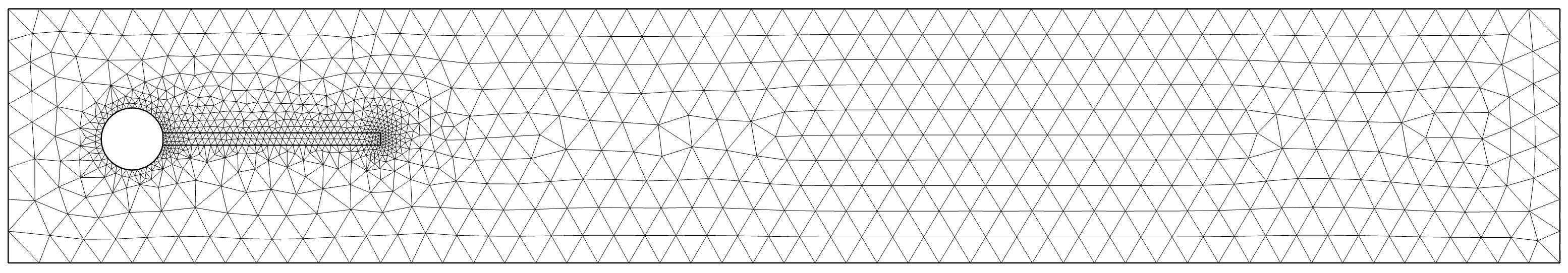}}
\caption{ \scriptsize Mesh with 1992 elements. }
\label{fig:THmesh}
\end{figure}

\begin{figure}[htbp]
\centering
\subfigure{\includegraphics[width=0.4\textwidth,
height=40mm]{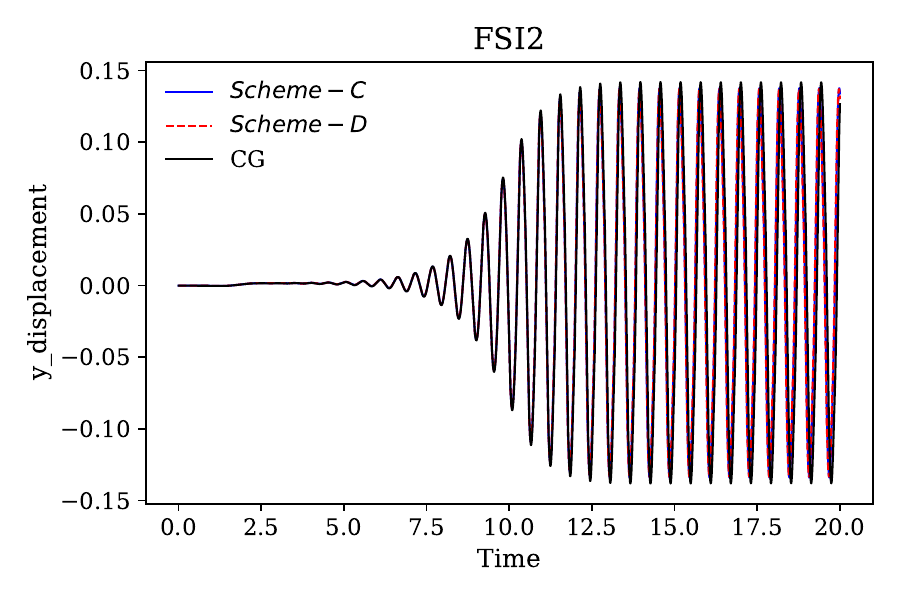}}\quad 
\subfigure{\includegraphics[width=0.4\textwidth,
height=40mm]{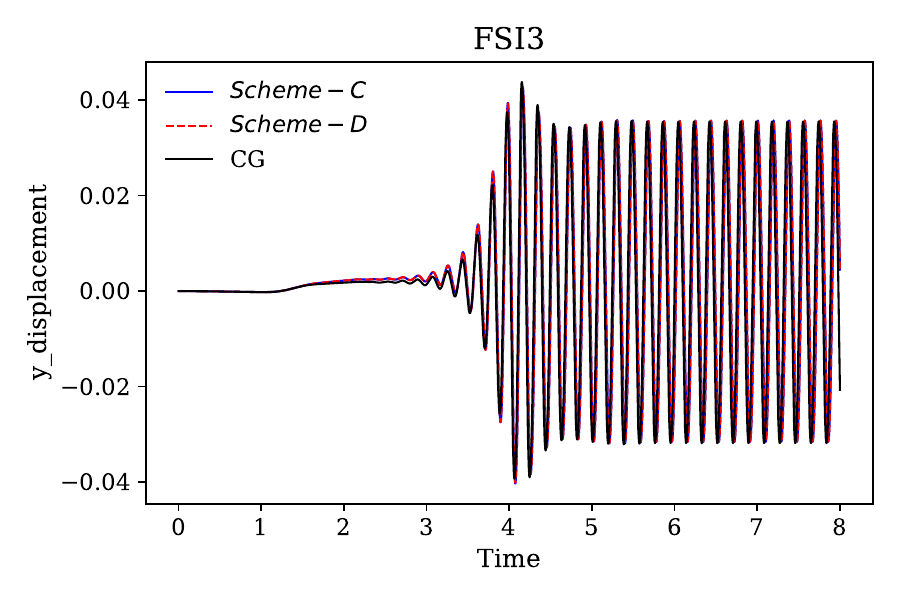}}
\caption{ \footnotesize y-component of displacements at point A for FSI2 (left) and FSI3 (right) with  $k = 3$ and  $\tau=0.001$. }
\label{fig:THdis}
\end{figure}

For the FSI2 case, which involves apparent deformations of the elastic beam, the simulation is carried out up to the final time $T = 20$.  As shown in Figure~\ref{fig:THdis} (left), the discrete solutions exhibit noticeable oscillations after $t = 12$, and the results obtained from the two proposed schemes remain in close agreement with those from the CG method. Figures \ref{fig:THC2} and \ref{fig:THD2} demonstrate the numerical velocity fields in a part of the computational domain at two time instants, $t=18.5s$ and $t=19.5s$, obtained by Scheme-C and Scheme-D, respectively.

For the FSI3 case, the beam undergoes less severe deformations, yet the higher fluid velocity leads to faster oscillations.  The simulation is performed up to $T = 8$ and,  as illustrated in Figure~\ref{fig:THdis} (right), significant oscillatory behavior emerges after $t = 5$.  Nevertheless, the solutions produced by the two proposed schemes still closely match those obtained using the CG method. In addition, the numerical velocity fields   obtained by Scheme-C and Scheme-D at $t=7s$ and $t=7.5s$  are shown in Figures \ref{fig:THC3} and \ref{fig:THD3},   respectively.

\begin{figure}[htbp]
    \centering
    \includegraphics[width=0.45\textwidth]{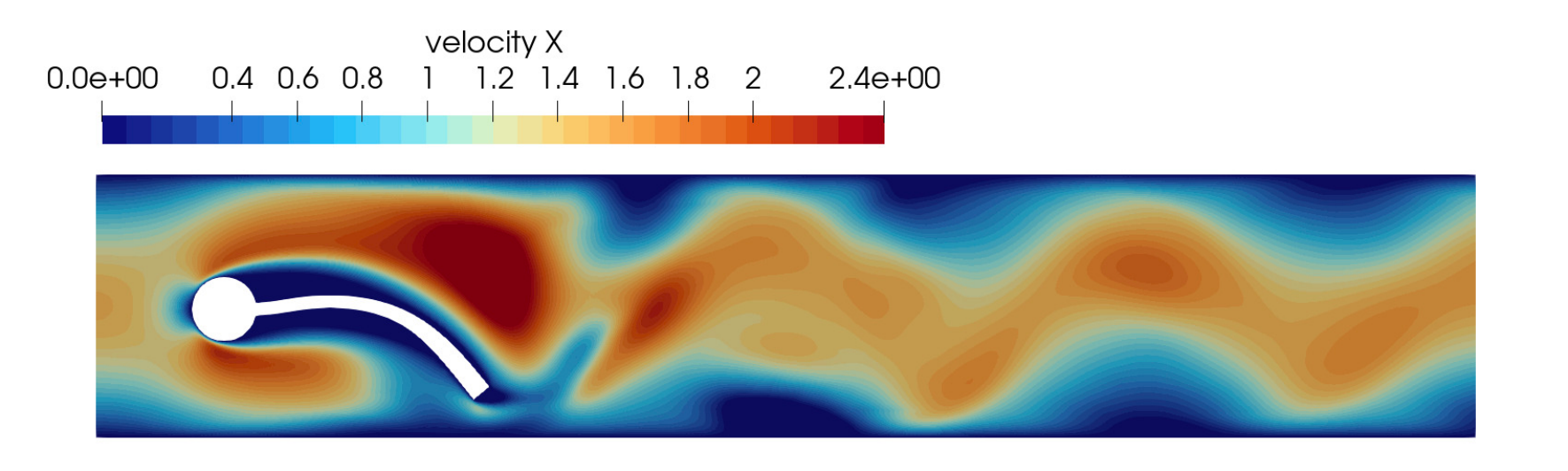}
    \includegraphics[width=0.45\textwidth]{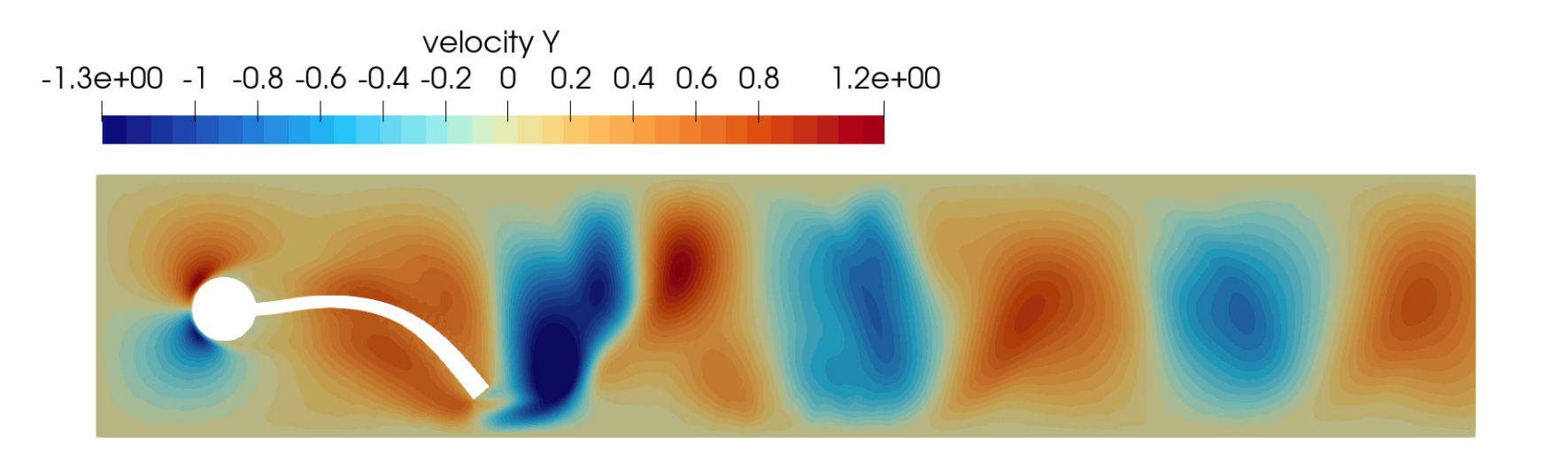}

    \includegraphics[width=0.45\textwidth]{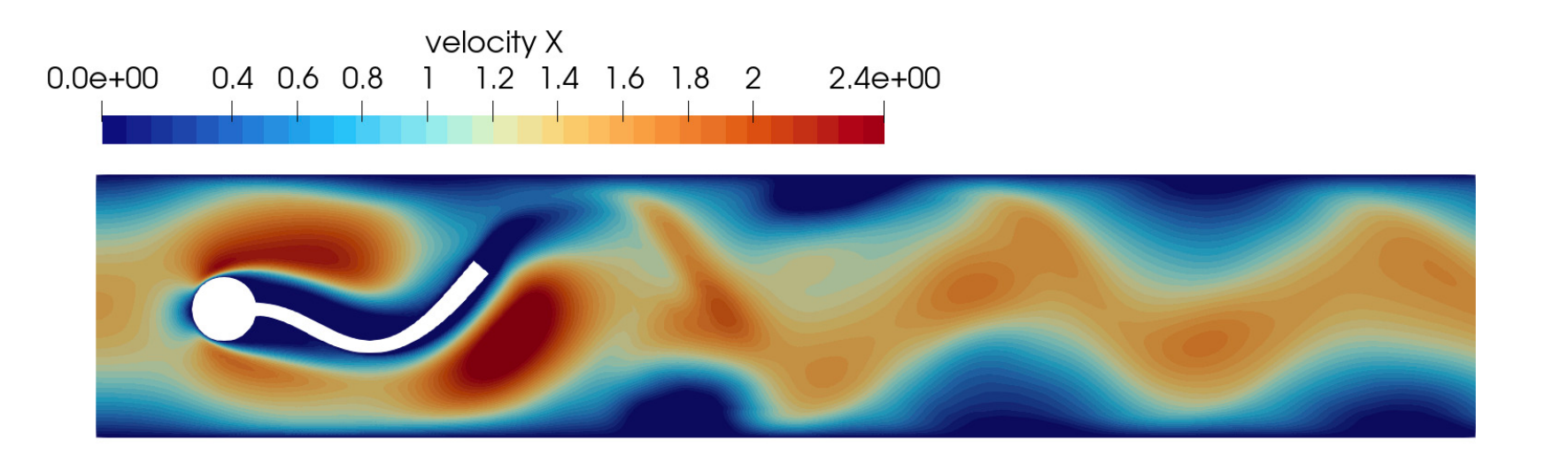}
    \includegraphics[width=0.45\textwidth]{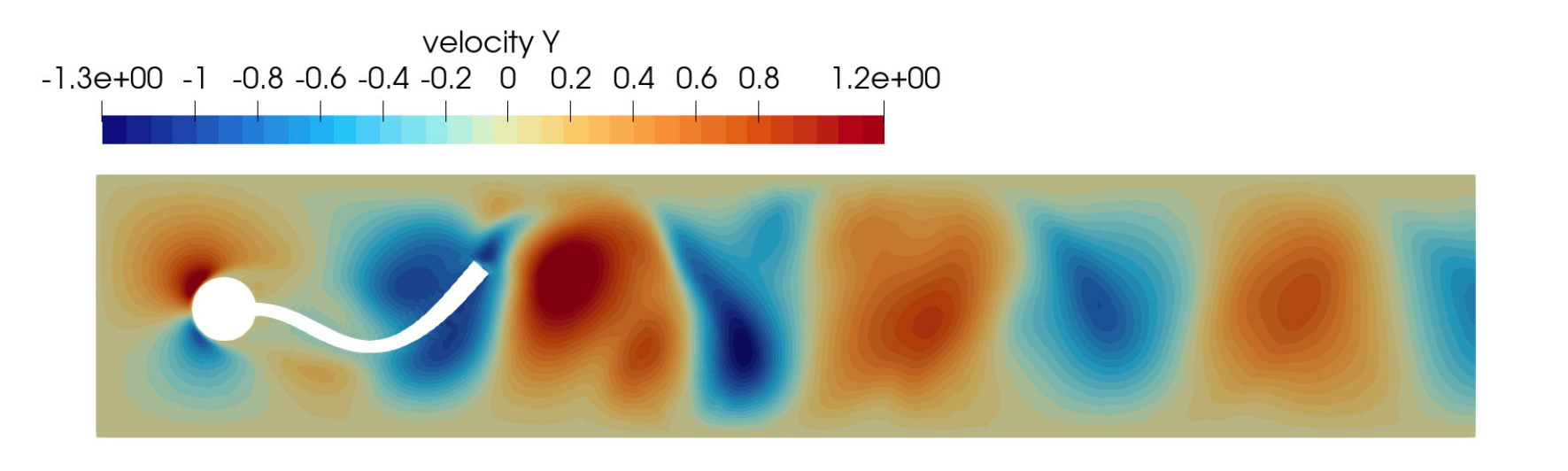}
    \caption{\footnotesize Numerical solutions of the velocity using Scheme-C for FSI2 (Left: $x$-component; right: $y$-component; Top: $t$ =  18.5s; bottom: $t$ = 19.5 s).}
\label{fig:THC2}
\end{figure}

\begin{figure}[htbp]
    \centering
    \includegraphics[width=0.45\textwidth]{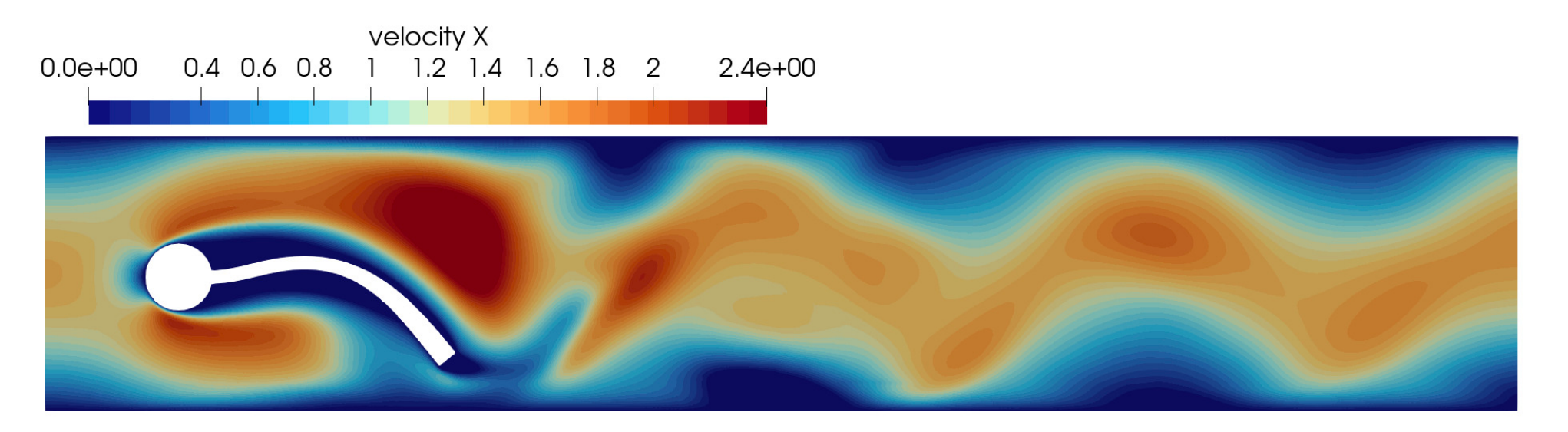}
    \includegraphics[width=0.45\textwidth]{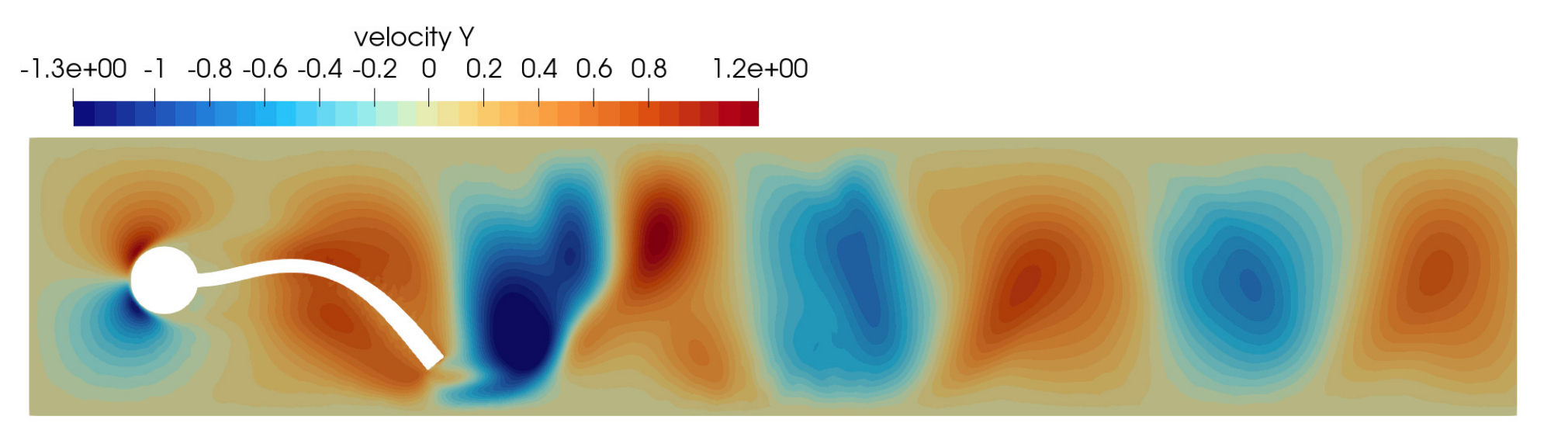}

    \includegraphics[width=0.45\textwidth]{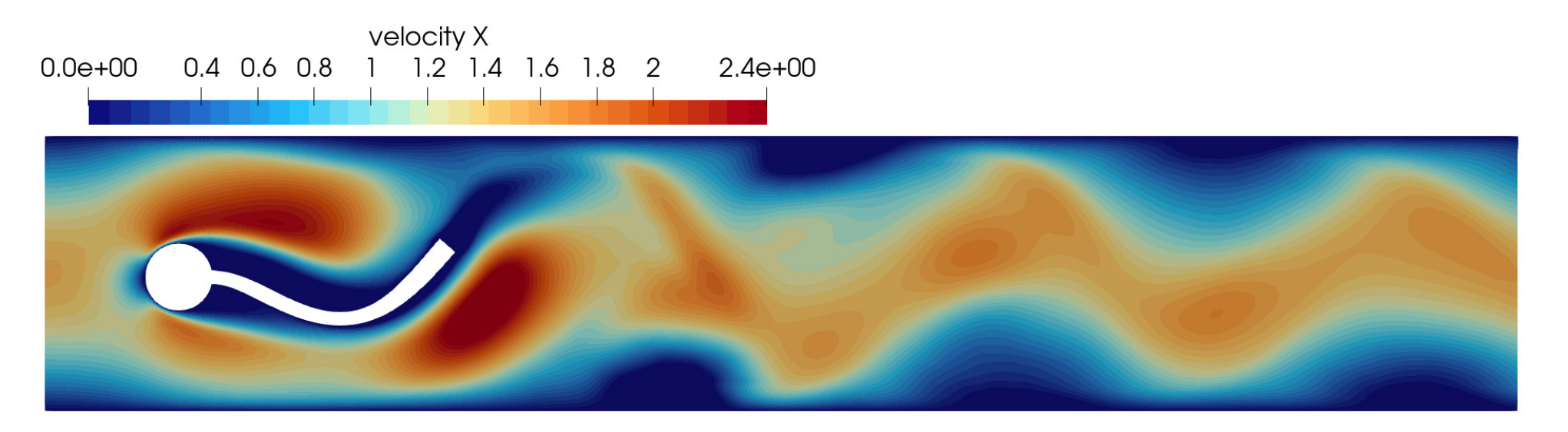}
    \includegraphics[width=0.45\textwidth]{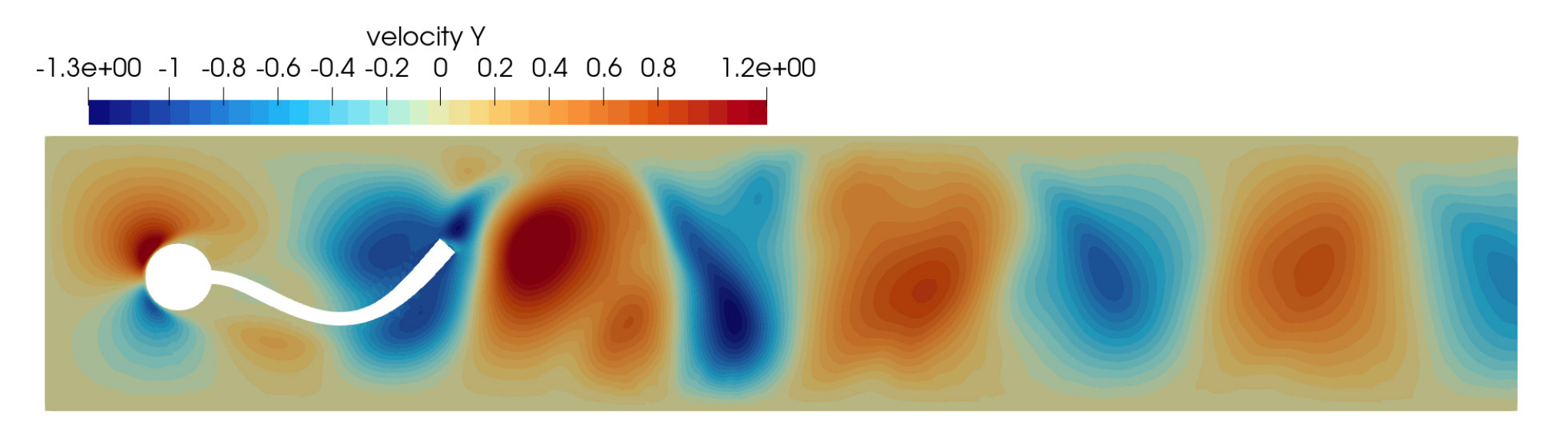}
    \caption{\footnotesize Numerical solutions of the velocity using Scheme-D for FSI2 (Left: $x$-component; Right: $y$-component; Top: $t$ =  18.5s; Bottom: $t$ = 19.5 s).}
\label{fig:THD2}
\end{figure}

\begin{figure}[htbp]
    \centering
    \includegraphics[width=0.45\textwidth]{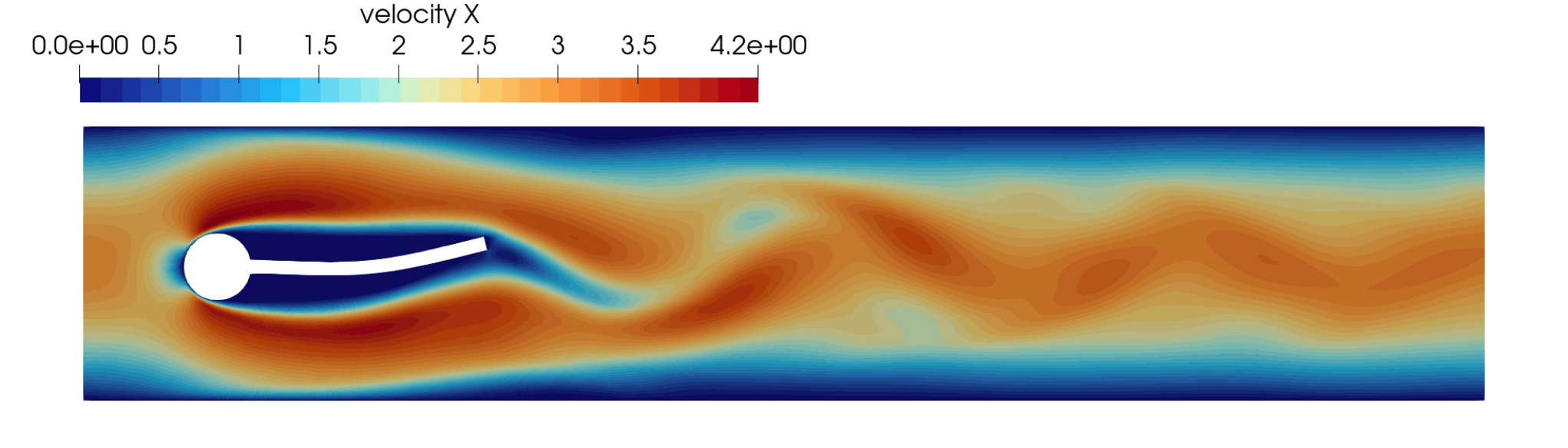}
    \includegraphics[width=0.45\textwidth]{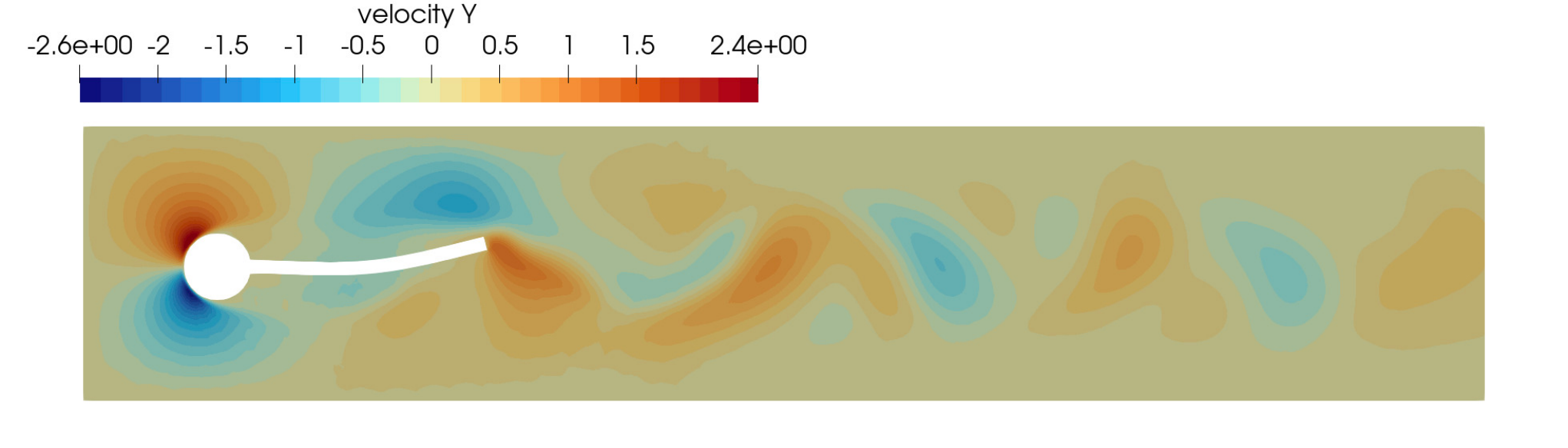}

    \includegraphics[width=0.45\textwidth]{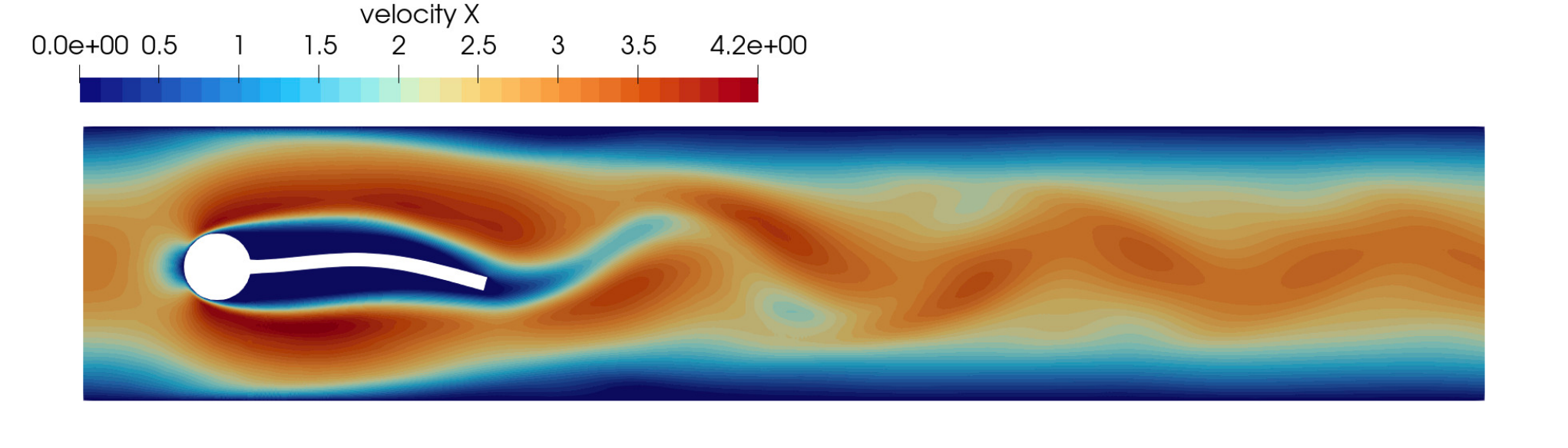}
    \includegraphics[width=0.45\textwidth]{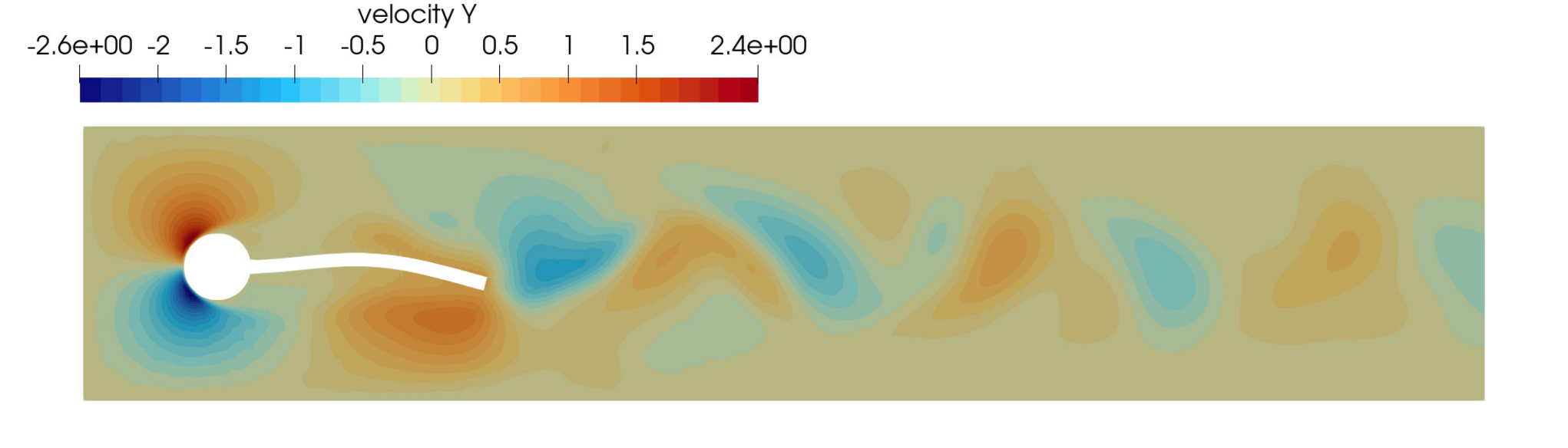}
   \caption{\footnotesize Numerical solutions of the velocity using Scheme-C for FSI3 (Left: $x$-component; Right: $y$-component; Top: $t$ =  7s; Bottom: $t$ = 7.5 s).}
\label{fig:THC3}
\end{figure}

\begin{figure}[htbp]
    \centering
    \includegraphics[width=0.45\textwidth]{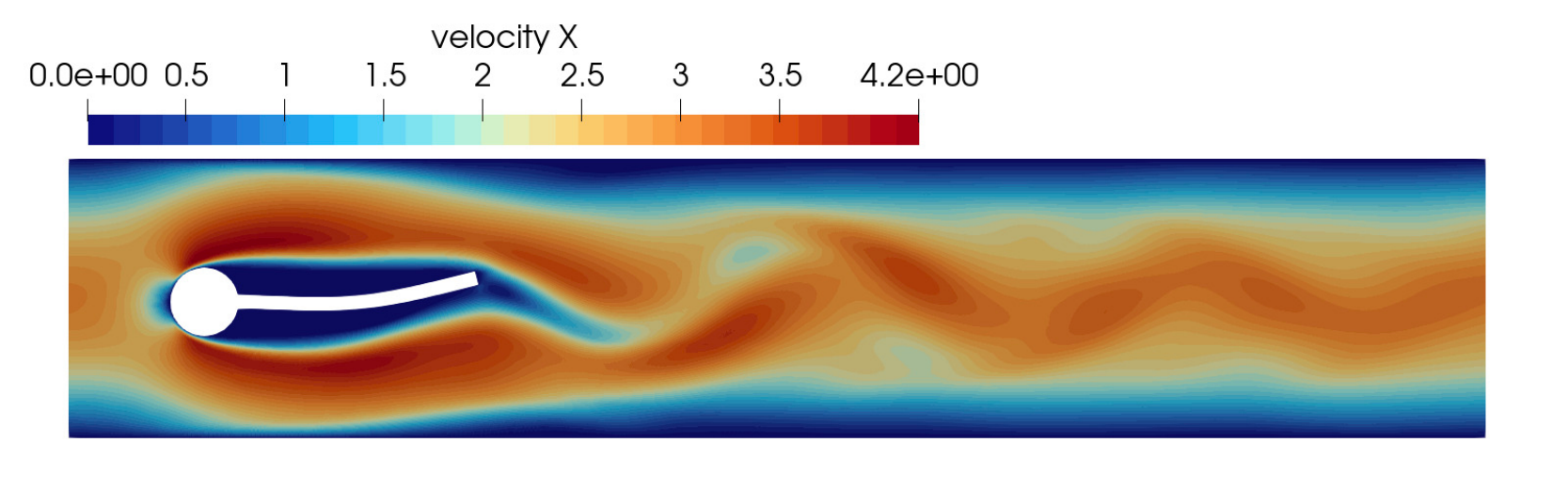}
    \includegraphics[width=0.45\textwidth]{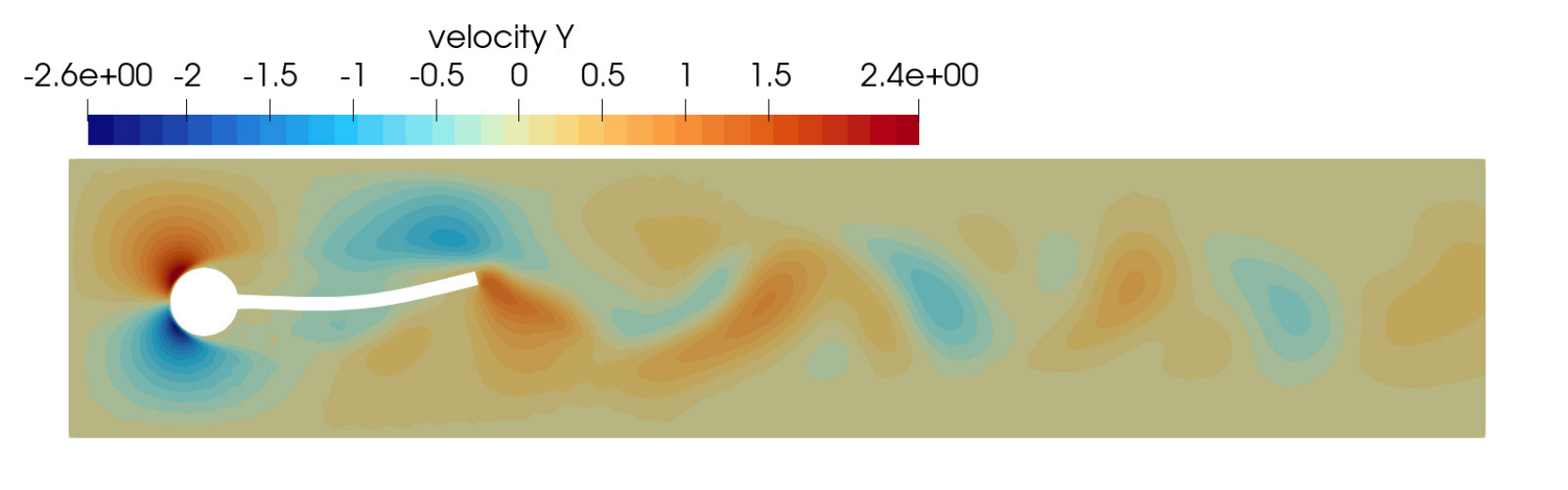}

    \includegraphics[width=0.45\textwidth]{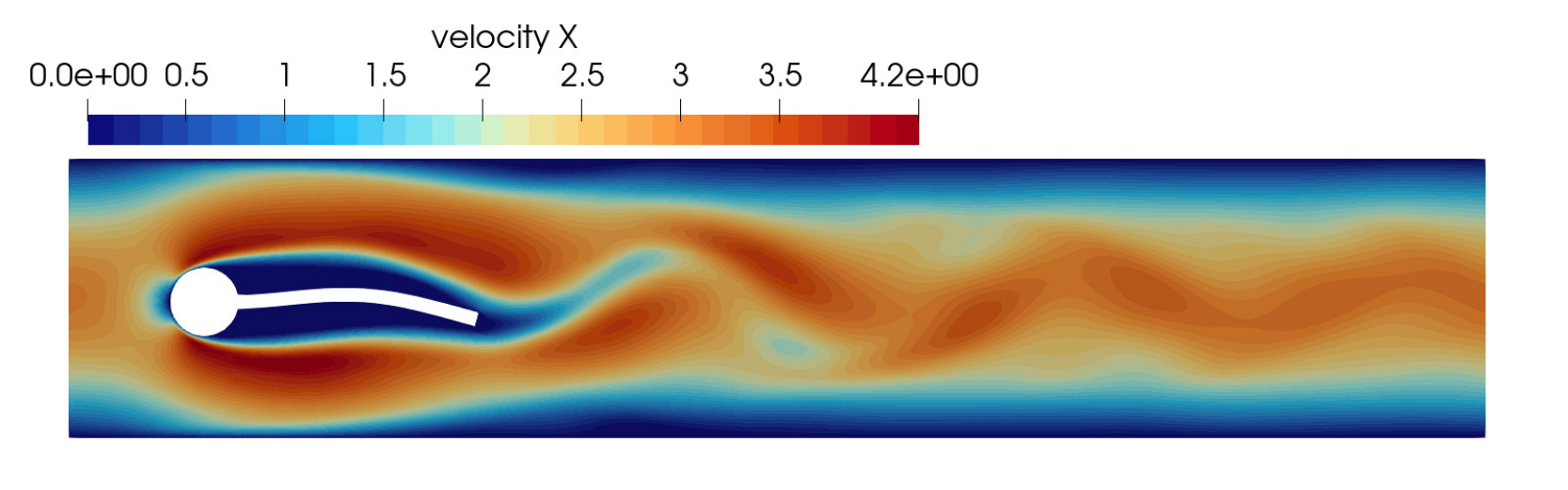}
    \includegraphics[width=0.45\textwidth]{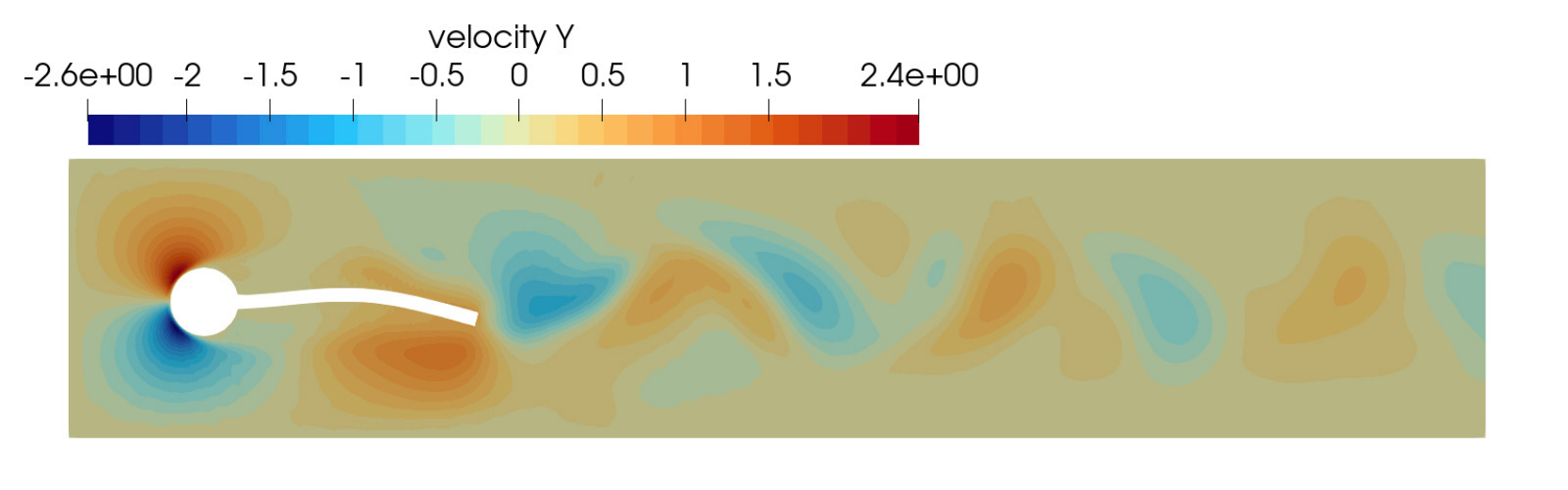}
    \caption{\footnotesize Numerical solutions of the velocity using Scheme-D for FSI3 (Left: $x$-component; Right: $y$-component; Top: $t$ =  7s; Bottom: $t$ = 7.5 s).}
\label{fig:THD3}
\end{figure}

\section{Conclusions}\label{Sec_con}

We have developed two monolithic numerical methods for FSI based on the Piola-type ALE mapping, where the backward Euler scheme is used for  the time discretization, and  the HDG framework   for the spatial discretization of both the fluid and structural equations.  The methods   lead to  globally divergence-free   fluid velocity approximations.  Energy stability results have been established for both the semi-discrete and fully discrete schemes. The provided numerical results have  confirmed the stability analysis and  demonstrated the performance of the proposed methods.


\end{document}